\documentclass[]{interact}
\usepackage{dsfont}
\usepackage{epstopdf}
\usepackage[caption=false]{subfig}

\usepackage[numbers,sort&compress]{natbib}
\bibpunct[, ]{[}{]}{,}{n}{,}{,}

\theoremstyle{plain}
\newtheorem{theorem}{Theorem}[section]
\newtheorem{lemma}[theorem]{Lemma}

\theoremstyle{definition}
\newtheorem{definition}[theorem]{Definition}
\newtheorem{example}[theorem]{Example}

\theoremstyle{remark}
\newtheorem{remark}{Remark}

\usepackage{graphicx}
\usepackage[dvipsnames]{xcolor}

\usepackage{algorithm}
\usepackage{algorithmic}
\usepackage{caption}

\usepackage{mathtools}
\usepackage{amsfonts}
\usepackage{amssymb}
\usepackage{url}
\usepackage{dirtytalk}
\usepackage{tikz}
\usepackage{nicefrac}
\usepackage{hyperref}
\usepackage{multirow}
\usepackage{cancel}

\newcommand\eqdef{\mathrel{\stackrel{\makebox[0pt]{\mbox{\normalfont\tiny def}}}{=}}}
\newcommand*{\R}{{\mathbb R}}

\newcommand*{\e}{\varepsilon}
\newcommand*{\la}{\langle}
\newcommand*{\ra}{\rangle}

\newcommand*{\norm}[1]{\left\lVert#1\right\rVert}

\DeclareMathOperator*{\argmin}{argmin}

\def\pd#1{{\color{black}#1}} 
\def\gav#1{{\color{black}#1}}
\def\dd#1{{\color{black}#1}} 
\def\att#1{{\color{black}#1}}
\def\attt#1{{\color{black}#1}}

\def\ga#1{{\color{black}#1}}

\def\rev1#1{{\color{black}#1}}

\def\revv#1{{\color{black}#1}} 
\def\moh#1{{\color{black}#1}} 
\def\revvat#1{{\color{black}#1}} 
\def\dip#1{{\color{black}#1}} 

\begin{document}

\articletype{ARTICLE TEMPLATE}

\title{Inexact Relative Smoothness and Strong Convexity for Optimization and Variational Inequalities by Inexact Model\footnote{The work in Sections 2, 3.1, Definitions 4.1 and 4.6, Algorithm 3 and Appendix C was supported by the Ministry of Science and Higher Education of the Russian Federation (Goszadaniye in MIPT, project 075-00337-20-03). The work in Sections 3.2 and 3.3 was funded by Russian Foundation for Basic Research grant, project number 19-31-90062. The work in Definition 4.9 and Theorems 4.8, 4.11 was supported by the grant of the President of the Russian Federation, code MC-15.2020.1. The work in Section 5 was supported by Russian Science Foundation, project number 18-71-10044. The work in Section 6 was supported by Russian Foundation for Basic Research, project number 18-29-03071 mk}.
}

\author{
\name{Fedor Stonyakin\textsuperscript{a,b}\thanks{CONTACT Fedor Stonyakin. Email: fedyor@mail.ru} and Alexander Tyurin\textsuperscript{b,d} and Alexander Gasnikov\textsuperscript{b,c,d} and Pavel Dvurechensky\textsuperscript{e,c} and Artem Agafonov\textsuperscript{b} and Darina Dvinskikh\textsuperscript{b,e} and Mohammad Alkousa\textsuperscript{d,b} and Dmitry Pasechnyuk\textsuperscript{b}
and Sergei Artamonov\textsuperscript{d}
and Victorya Piskunova\textsuperscript{a}}
\affil{
\textsuperscript{a}V. Vernadsky Crimean Federal University, Simferopol, Republic of Crimea; \textsuperscript{b}Moscow Institute of Physics and Technology, Dolgoprudny, Russia;
\textsuperscript{c}Institute for Information Transmission Problems, Moscow, Russia;
\textsuperscript{d}National Research University Higher School of Economics, Moscow, Russia;
\textsuperscript{e}Weierstrass Institute for Applied Analysis and Stochastics, Berlin, Germany
}
}

\maketitle

\begin{abstract}
In this paper, we propose a general algorithmic framework for first-order methods in optimization in a broad sense, including minimization problems, saddle-point problems, and variational inequalities. This framework allows obtaining many known methods as a special case, the list including accelerated gradient method, composite optimization methods, level-set methods, Bregman proximal methods. The idea of the framework is based on constructing an inexact model of the main problem component, i.e. objective function in optimization or operator in variational inequalities. Besides reproducing known results, our framework allows constructing new methods, which we illustrate by constructing a universal conditional gradient method and a universal method for variational inequalities with a composite structure. This method works for smooth and non-smooth problems with optimal complexity without a priori knowledge of the problem's smoothness. As a particular case of our general framework, we introduce relative smoothness for operators and propose an algorithm for variational inequalities (VIs) with such operators. We also generalize our framework for relatively strongly convex objectives and strongly monotone variational inequalities.
\end{abstract}


\section{Introduction}

In this paper we consider the following optimization problem
\begin{equation}
\label{eq:Problem}
\min_{x\in Q} f(x),
\end{equation}
where $Q$ is a convex subset of finite-dimensional vector space $E$, $f$ is generally a non-convex function.

Most minimization methods for such problems are constructed using some model of the objective $f$ at the current iterate $x_k$. This can be a quadratic model based on the $L$-smoothness of the objective
\begin{equation}\label{eq:quadr_model}
    f(x_k) + \la \nabla f(x_k), x-x_k \ra + \frac{L}{2}\|x-x_k\|_2^2.
\end{equation}
The step of gradient method is obtained by the minimization of this model \cite{nesterov2018lectures} (\attt{see the orange, dashed line in Fig. \ref{Fig:model_vis}}). More general models are constructed based on regularized second-order Taylor expansion \cite{nesterov2006cubic} or other Taylor-like models \cite{drusvyatskiy2019nonsmooth} as well as other objective surrogates \cite{mairal2013optimization}. Another example is the conditional gradient method \cite{frank1956algorithm}, where a linear model of the objective is minimized on every iteration. Adaptive choice of the parameter of the model with provably small computational overhead was proposed in \cite{nesterov2006cubic} and applied to first-order methods in \cite{nesterov2013gradient,nesterov2015universal, dvurechensky2018computational}. Recently, the first-order optimization methods were generalized to the so-called relative smoothness framework \cite{bauschke2016descent,lu2018relatively,ochs2017non}, where $\frac{1}{2} \|x-x_k\|_2^2$ in the quadratic model \eqref{eq:quadr_model} for the objective is replaced by general Bregman divergence.

The literature on the first-order methods \cite{devolder2014first,bogolubsky2016learning,dvurechensky2017universal} considers also gradient methods with inexact information, relaxing the model \eqref{eq:quadr_model} to
\begin{equation}\label{eq:inex_quadr_model}
	 f_{\delta}(x_k) + \la \nabla f_\delta(x_k), x-x_k \ra + \frac{L}{2}\|x-x_k\|^2_2 + \delta,
\end{equation}
with $(f_\delta,\nabla f_\delta)$ called inexact oracle and this model being an upper bound for the objective. In particular, this relaxation allows to obtain universal gradient methods \cite{nesterov2015universal}.

One of the goals of this paper is to describe and analyze the first-order optimization methods which use a very general \textit{inexact model} of the objective function, the idea being to replace the linear part in \eqref{eq:inex_quadr_model} by a general function $\psi_{\delta}(x,x_k)$ and the squared norm by general Bregman divergence (\attt{see the red, dotted line in Fig. \ref{Fig:model_vis}}). The resulting model includes as a particular case inexact oracle model and relative smoothness framework and allows one to obtain many optimization methods as a particular case, including conditional gradient method \cite{frank1956algorithm}, Bregman proximal gradient method \cite{chen1993convergence} and its application to optimal transport \cite{xie2018fast} and Wasserstein barycenter \cite{stonyakin2019gradient} problems, general Catalyst acceleration technique \cite{lin2015universal}, (accelerated) composite gradient methods \cite{nesterov2013gradient,beck2009fast},
(accelerated) level and bundle-type methods \cite{nemirovskii1985optimal,lan2015bundle,nesterov2021gradient}, gradient methods in the relative smoothness framework \cite{bauschke2016descent,lu2018relatively}. First attempts to propose inexact model generalization were made in \cite{stonyakin2019gradient,gasnikov2017universal} for non-accelerated methods
and in \cite{tyurin2017fast} for accelerated methods, yet without a relative smoothness paradigm. In this paper, we propose the inexact model in a very general setting, including
adaptivity of the algorithms to the parameter $L$, possible relative strong convexity, and relative smoothness. We also provide convergence rates for the gradient method and accelerated gradient method using an inexact model of the objective. As an application of our general framework, we develop a universal conditional gradient method, providing a parameter-free generalization of the results in \cite{nesterov2018complexity}.

\begin{figure}
  \centering
  \includegraphics[width=0.65\textwidth]{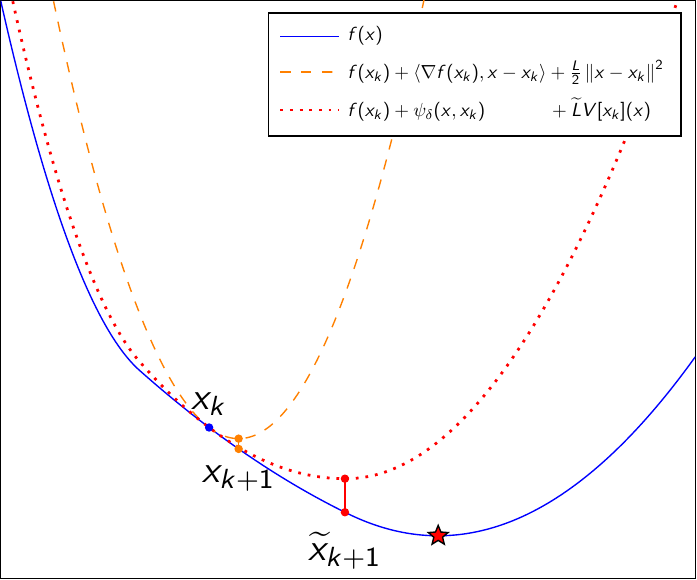}
  \caption{{\color{orange} Orange}, dashed line and $x_{k+1}$ illustrate the standard upper-bound and a step of the classical gradient method. {\color{red} Red}, dotted line and $\widetilde{x}_{k+1}$ illustrate the upper-bound based on our notion of inexact model and a step of the proposed method. \attt{We can see that with better upper-bound, the next iterate is closer to the minimum point.}\label{Fig:model_vis}}
\end{figure}

We believe that our model is flexible enough to be extended for problems with primal-dual structure\footnote{see recent results on this generalization in \cite{tyurin2019primal}.} \cite{nesterov2009primal-dual,nemirovski2010accuracy,nesterov2018complexity}, e.g. for problems with linear constraints \cite{anikin2017dual,chernov2016fast,guminov2019accelerated,nesterov2018primal-dual}; for random block-coordinate descent \cite{dvurechensky2017randomized}; for tensor methods \cite{nesterov2018implementable,gasnikov2019near}; for distributed optimization setting \cite{scaman2017optimal,uribe2018distributed,dvurechensky2018decentralize,dvinskikh2019primal}; and adaptive stochastic optimization \cite{iusem2019variance,ogaltsov2019adaptive}.

Optimization problem \eqref{eq:Problem} is tightly connected with variational inequality (VI)
\begin{equation}
    \label{eq:VI}
    \text{Find} \; x_* \in Q \; \text{s.t.} \; \la g(x_*),x_*-x\ra \leq 0, \; \forall x \in Q,
\end{equation}
where $g(x) = \nabla f(x)$. A special VI is also equivalent to finding a saddle-point of a convex-concave function
\begin{equation}
    \label{eq:SP}
    \min_{u\in Q_1}\max_{v\in Q_2} f(u,v)
\end{equation}
for $x=(u,v)$ and $g(x) = (\nabla_u f(u,v),-\nabla_v f(u,v))$. This motivates the second part of this paper, which consists of the generalization of the inexact model of the objective function to an inexact model for an operator in variational inequality. In particular, we extend the relative smoothness paradigm to variational inequalities with monotone and strongly monotone operators and provide a generalization of the Mirror-Prox method \cite{nemirovski2004prox}, its adaptive version \cite{gasnikov2019adaptive} (see also \cite{malitsky2020proximal,antonakopoulos2019adaptive,malitsky2020golden,antonakopoulos2020adaptive} for other types of adaptive extra-gradient methods), and universal version \cite{stonyakin2018generalized} to variational inequalities with such general inexact model of the operator. As a particular case, our approach allows us to partially \dd{reproduce} the results of \cite{chambolle2011first-order}.
We also apply the general framework for variational inequalities to saddle-point problems.

To sum up, we present a unified view on inexact models for optimization problems, variational inequalities, and saddle-point problems.

The structure of the paper is the following. In Section \ref{S:Inexact_min}, we
introduce an inexact model of the objective in optimization and provide several examples to illustrate the flexibility and generality of the proposed model. In particular, we demonstrate that relative smoothness and strong convexity are the special cases of our general framework.

In Section \ref{GM}, we consider an adaptive gradient method (GM) and an adaptive fast gradient method (FGM). FGM has a better convergence rate yet is not adapted to the relative smoothness paradigm. In Subsection \ref{univ_frank_wolfe}, we construct a universal conditional gradient (Frank--Wolfe) method using FGM with inexact projection. To the best of our knowledge, this is the first attempt to combine the Frank--Wolfe method \cite{har2015cond, jaggi2013revisiting} and the universal method \cite{nesterov2015universal}.
In Section \ref{VI}, we generalize the inexact model to variational inequalities and saddle-point problems for the case of monotone and strongly monotone operators.
In the former case, we construct an adaptive generalization of the Mirror-Prox algorithm for variational inequalities and saddle-point problems with such an inexact model. In the latter case, the proposed algorithm is accelerated by the restart technique to have the linear rate of convergence.
We especially consider the case of $m$-strong convexity of the model. The natural motivations for such a formulation are composite saddle-point problems and mixed variational inequalities with an $m$-strongly convex composite.

The contribution of this paper is follows:
\begin{enumerate}
    \item \att{We introduce an inexact $(\delta, L, \mu, m, V)$-model for optimization problems and obtain convergence rates for adaptive GM for optimization problems with this model.
    }
    \item \att{We introduce an inexact $(\delta, L, \mu, m, V,\|\cdot \|)$-model for optimization problems and obtain convergence rates for adaptive FGM for optimization problems with such a model.
    Using FGM with the inexact model we construct a universal conditional gradient (Frank--Wolfe) method.}
    \item \att{We propose generalizations of the above models, namely $(\delta, L, V)$-model and $(\delta, L, \mu, V)$-model for variational inequalities and saddle-point problems. \pd{As a special case, we introduce relative smoothness for operators in variational inequalities, thus, generalizing  \cite{lu2018relatively} from optimization to variational inequalities.} We obtain convergence rates for adaptive versions of the Mirror-Prox algorithm for problems with an inexact model.}
    \rev1{In Example \ref{L_condition_Example}, we show that our method is able to solve variational inequalities with singularities, for which the application of existing Mirror-Prox algorithms is not theoretically justified.\footnote{As an exception we mention works \cite{antonakopoulos2019adaptive,antonakopoulos2020adaptive} that appeared after the first version of this paper appeared as a preprint \cite{stonyakin2019inexact}.}}
\end{enumerate}

\rev1{
To demonstrate the practical performance of the proposed algorithms, 
several numerical experiments were performed for different examples. 
First, in Appendix \ref{app_numerical_1}, we consider a constrained convex optimization problem with non-smooth functional constraints motivated by three different geometrical problems: an analogue of the {\it smallest covering circle problem, the problem of the best approximation of the distance from the given set, and the Fermat--Torricelli--Steiner problem}. 
Further, in Appendix \ref{appendix_relative_obj}, we test our algorithms on a minimization problem of a polynomial-like function, which is shown in \cite{lu2018relatively} to be relatively smooth, under non-smooth functional constraints.
This is followed by numerical experiments in Appendix \ref{D_optimal_problem} for the D-optimal design problem with the relatively smooth objective and simplex constraints, which has wide applications in statistics and computational geometry. We report the advantages of the proposed method and the possibility to control the convergence of the method by adjusting the line-search parameter.
Finally, in Appendix \ref{app_RSP} we describe the results of the experiments for the resource sharing problem considered as an example in \cite{antonakopoulos2019adaptive}.
}

\section{Inexact Model in Minimization Problems. Definition\dd{s} and Examples}
\label{S:Inexact_min}
\rev1{In this section, we introduce a general notation which is used in the paper and our main definition of the inexact model for optimization problems. To motivate this definition and illustrate its generality, we give several examples of known situations which are covered by our inexact model.}

As said, we start with the general notation. Let $E$ be an $n$-dimensional real vector space and $E^*$ be its dual. We denote the value of a linear function $g \in E^*$ at $x\in E$ by $\la g, x \ra$. Let $\|\cdot\|$ be some norm on $E$, $\|\cdot\|_{*}$ be its dual, defined by $\|g\|_{*} = \max\limits_{x} \big\{ \la g, x \ra, \| x \| \leq 1 \big\}$. 
\rev1{We use $\nabla f(x)$ to denote the gradient of a differentiable function $f$ at a point $x \in {\rm dom} f$ and, with a slight abuse of notation, a subgradient of a convex function $f$ at a point $x \in {\rm dom} f$. The whole subdifferential of a convex function $f$ at a point $x \in {\rm dom} f$ is denoted by $\partial f(x)$.} 
\rev1{Let $d$ be a convex function on $Q$, which is continuously differentiable on the relative interior ${\rm ri}Q$ of $Q$. Let $V[y](x) = d(x) - d(y) - \la \nabla d(y), x - y \ra$, $x \in Q, y \in {\rm ri}Q$ be the corresponding Bregman divergence.}
Most typically it is assumed that $d$ is $1$-strongly convex on $Q$ w.r.t. $\|\cdot\|$-norm, which we refer to as (1-SC) assumption w.r.t. $\|\cdot\|$-norm. 
\rev1{More precisely, this assumption means that,} for all $x \in Q,y \in {\rm ri} Q$, $d(x) - d(y) - \la \nabla d(y), x - y \ra \geqslant \frac{1}{2}\|x-y\|^2$.
We underline that, in general, we do not make this assumption, and, in what follows, we explicitly write if this assumption is made.

\begin{definition}[\rev1{Inexact model in optimization}]\label{defRelStronglyConvex}
Let $\delta,L,\mu,m \geq 0$. We say that $\psi_{\delta}(x, y)$ is a $(\delta, L, \mu, m, V)$-model of the function $f$ at a given point $y$ iff, for all $x \in Q$,
\begin{equation}
\label{eq:Str_Conv_Model}
{\mu V[y](x)} \leq {f(x) - \left(f_{\delta}(y) + \psi_{\delta}(x, y)\right)} \leq {LV[y](x) + \delta},
\end{equation}
 $\psi_{\delta}(x, y)$ \rev1{is convex in $x$}, satisfies $\psi_{\delta}(x,x)=0$ for all $x \in Q$ and \begin{equation}\label{mstongmodel}
\psi(x) \geqslant \psi(z) + \la \attt{g}, x- z \ra + m V[z](x), \quad \forall x, z \in Q, \,\,\attt{\forall g \in \partial \psi(z)},
\end{equation}
where for fixed $y \in Q$ and any $x \in Q$ we denote $\psi(x) = \psi_{\delta}(x, y)$. \attt{In other words, $\psi(x)$ is relatively $m$-strongly convex relative to $d$.}
\end{definition}

Note that in Definition \ref{defRelStronglyConvex} we allow $L$ to depend on $\delta$. Definition \ref{defRelStronglyConvex} is a generalization of $(\delta, L)$-model from \cite{gasnikov2017universal,tyurin2017fast, stonyakin2019gradient}, where $\mu=0$ and $m = 0$. Below, we denote $(\delta, L,0,0,V)$-model as $(\delta, L)$-model. Thus, our framework allows obtaining (accelerated) gradient methods for composite optimization and their counterparts for inexact oracle models.

Let us illustrate the above definition by several examples.
\begin{example}{\bf Relative smoothness and relative strong convexity, \cite{bauschke2016descent, lu2018relatively}.}
Assume that \pd{$d(x)$ is differentiable}, and in \eqref{eq:Problem}, the objective $f$ is \pd{differentiable} and relatively smooth \cite{bauschke2016descent,lu2018relatively} relative to $d$, i.e.
\[
f(x)-f(y) -\la \nabla f(y), x-y \ra  \leq \left(d(x)-d(y) -\la \nabla d(y), x-y \ra\right) = LV[y](x), \; \forall x,y \in Q
\]
and relatively strongly convex \cite{lu2018relatively} relative to $d$, i.e.
\[
\mu V[y](x) = \mu\left(d(x)-d(y) -\la \nabla d(y), x-y \ra\right)  \leq f(x)-f(y) -\la \nabla f(y), x-y \ra, \; \forall x,y \in Q.
\]
Then, clearly, Definition \ref{defRelStronglyConvex} holds with $m = 0$, $\delta = 0$, $\psi_\delta(x,y) = \langle \nabla f(y), x - y \rangle$. Importantly, the function $d$ is not necessarily strongly convex.

Note that if $V[y](x) \leq C_n\norm{x-y}^2$ for some constant $C_n = O(\log n)$,
the condition of $(\mu C_n)$-strong convexity w.r.t. norm $\norm{\cdot}$, namely
$\mu C_n \norm{x-y}^2 + f_\delta(y) + \psi_\delta(x,y) \leq f(x)$ implies the left inequality in \eqref{eq:Str_Conv_Model}.

One of the main applications of general relative smoothness and strong convexity is the step of tensor methods which use the derivatives of the objective of the order higher than 2 \cite{nesterov2018implementable,gasnikov2019near}.  Thus, our framework allows obtaining gradient method for optimization with relative smoothness and strong convexity and extending them to the case of inexact oracle setting.
\end{example}

\begin{example}{\bf Composite optimization, \cite{beck2009fast, nesterov2013gradient}.}
Assume that in \eqref{eq:Problem}, $f(x) = g(x) + h(x)$ with $L$-smooth w.r.t. norm $\|\cdot\|$ term $g$ and simple convex term $h$, which is usually called composite. In this case we assume that $V[y](x)$ satisfies (1-SC) condition w.r.t $\|\cdot\|$, and define $f_{\delta}(x) = \rev1{g(x)} + h(x)$ and $\psi_\delta(x,y) = \langle \nabla g(y), x - y \rangle + h(x) - h(y)$. 
\rev1{We have 
\begin{gather*}
f(x) - \left(f_{\delta}(y) + \psi_{\delta}(x, y)\right) = g(x) + h(x) - (g(y) + h(y) + \langle \nabla g(y), x - y \rangle + h(x) - h(y) ) \\ = g(x) - (g(y) + \langle \nabla g(y), x - y \rangle ).
\end{gather*}
By convexity of $g$, we have $0 \leq g(x) - (g(y) + \langle \nabla g(y), x - y \rangle )$. At the same time, by the $L$-smoothness of $g$ and (1-SC) condition for $V[y](x)$, 
\[
g(x) - (g(y) + \langle \nabla g(y), x - y \rangle )\leq \frac{L}{2}\|x-y\|^2 \leq LV[y](x).
\]
From the combination of the above three relations,} it is clear that \eqref{eq:Str_Conv_Model} holds with $\delta=0$ and $\mu=0$ and we are in the situation of Definition \ref{defRelStronglyConvex} with $m=0$ since $\psi_\delta(x,y)$ is convex in $x$. If $h$ turns out to be \pd{differentiable} and relatively $m$-strongly convex \attt{relative to $d$} \cite{lu2018relatively}, i.e. $h(x)-h(y)-\la\nabla h(y),x-y\ra \geq mV[y](x)$, then \eqref{mstongmodel} holds, but in \eqref{eq:Str_Conv_Model} we have $\mu=0$.
On the other hand, if $g$ turns out to be relatively $\mu$-strongly convex \attt{relative to $d$} \cite{lu2018relatively}, i.e. $g(x)-g(y)-\la\nabla g(y),x-y\ra \geq \mu V[y](x)$, then \eqref{eq:Str_Conv_Model} holds with $\delta=0$, but in \eqref{mstongmodel} we have $m=0$. 

A particular example is the following minimization problem \cite{anikin2015modern} motivated by traffic demands matrix estimation from link loads:
$$
f(x) = \frac{1}{2}\|Ax - b\|^2_2 + m \sum \limits_{k=1}^n x_k\ln{x_k} \rightarrow \min \limits_{x\in S_n(1)},
$$
where $S_n(1) = \left\{x\in \mathbb{R}^n_+| \sum_{i=1}^{n}x_i = 1\right\}$ is the standard unit simplex in $\mathbb{R}^n$. In this case $g(x) = \frac{1}{2}\|Ax - b\|^2_2$ and $h(x) = m \sum \limits_{k=1}^n x_k\ln{x_k}$. \rev1{If we choose $\|\cdot\| = \|\cdot\|_1$ and $d(x) = \sum \limits_{k=1}^n x_k\ln{x_k}$, then $d(x)$ satisfies (1-SC) condition, $V[y](x) = \sum \limits_{k=1}^n x_k\ln{(x_k/y_k)}$, and $h(x)$ is relatively $m$-strongly convex relative to $d$. Since $g$ is a quadratic function, by using the fact that \moh{Lipschitz} gradient is equivalent to the bounded Hessian, we obtain that $g$ has a Lipschitz-continuous gradient w.r.t. $\|\cdot\|_1$} with the constant $L=\max_{\|h\|_1 \leq 1} \la h,A^TA h\ra=\max_{k=1,...,n}\|A_k\|_2^2$, where $A_k$ is the $k$-th column of $A$.Thus, we obtain that $\psi_\delta(x,y) = \langle \nabla g(y), x - y \rangle + h(x) - h(y)$ is a $(0,L,0,m,V)$-model. At the same time, the part $g$ is not necessarily strongly convex. 
\end{example}


\begin{example}{\bf Superposition of functions, \cite{nemirovskii1985optimal}}.
Assume that in \eqref{eq:Problem} \cite{nemirovskii1985optimal,lan2015bundle}
$f(x) := g(g_1(x), \dots, g_m(x))$,
where each function $g_k(x)$ is a smooth convex function with $L_k$-Lipschitz gradient w.r.t. $\|\cdot\|$-norm for all $k$. The function $g(x)$ is an $M$-Lipschitz convex function w.r.t 1-norm, non-decreasing \dd{in} each of its arguments. The chosen Bregman divergence $V[y](x)$ is assumed to satisfy (1-SC). 
\rev1{We define $f_{\delta}(y) := f(y)$ and 
$$
\psi_{\delta}(x,y) = g(g_1(y) + \langle\nabla g_1(y), x - y \rangle, \dots, g_m(y)+\langle\nabla g_m(y), x - y \rangle) - f(y).
$$
By convexity and $L_k$-smoothness of each $g_k(x)$, we have, for all $k=1,...,m$
\[
g_k(y) +  \langle\nabla g_k(y), x - y \rangle \leq g_k(x) \leq g_k(y) +  \langle\nabla g_k(y), x - y \rangle +\frac{L_k}{2}\|y-x\|^2.
\]
By monotonicity and Lipschitz-continuity of $g$ w.r.t. 1-norm, we have
\begin{gather*}
g(g_1(y) + \langle\nabla g_1(y), x - y \rangle, \dots, g_m(y)+\langle\nabla g_m(y), x - y \rangle) \leq g(g_1(x), \dots, g_m(x)) \\
\leq g\left(g_1(y) + \langle\nabla g_1(y), x - y \rangle+\frac{L_1}{2}\|y-x\|^2, \dots, g_m(y)+\langle\nabla g_m(y), x - y \rangle +\frac{L_2}{2}\|y-x\|^2\right)  \\
\leq g(g_1(y) + \langle\nabla g_1(y), x - y \rangle, \dots, g_m(y)+\langle\nabla g_m(y), x - y \rangle) + M\frac{\sum_{i=1}^{m}L_i}{2}\norm{x - y}^2.
\end{gather*}
}
From this, by the (1-SC) condition for $V[y](x)$, we obtain
\begin{gather*}
0 \leq f(x) - (f(y) + g(g_1(y) + \langle\nabla g_1(y), x - y \rangle, \dots, g_m(y)+\langle\nabla g_m(y), x - y \rangle) - f(y)) \\\leq M\frac{\sum_{i=1}^{m}L_i}{2}\norm{x - y}^2 \leq MV[y](x)\sum_{i=1}^{m}L_i,  \,\,\,\, \forall x,y \in Q,
\end{gather*}
\rev1{which is \eqref{eq:Str_Conv_Model} with $\mu=0$ and $\delta=0$ given in the above definition of $f_{\delta}(y)$ and $\psi_{\delta}(x,y)$. Clearly,  $\psi_{\delta}(x,y)$ is convex in $x$ as a composition of non-decreasing convex function $g$ with linear functions. 
Thus, \eqref{mstongmodel} holds with $m=0$, and we obtain a $\left(0,M\cdot\left(\sum_{i=1}^{m}L_i\right) \right)$-model of $f$.}
As a result, our framework allows to obtain (accelerated) level  gradient methods considered in \cite{nemirovskii1985optimal,lan2015bundle} as a special case. Moreover, we generalize these methods for the case of inexact oracle information.
\end{example}

\begin{example}\label{E:Proximal}{\bf Proximal method, \cite{chen1993convergence}.}
\label{prox_ex}
Let us consider the optimization problem \eqref{eq:Problem}, where $f$ is an arbitrary convex function (not necessarily smooth).
Then, for arbitrary $L \ge 0$, $\psi_{\delta}(x,y) = f(x) - f(y)$
is  $(0, L)$-model of $f$ with $f_{\delta}(y) = f(y)$ at a given point $y$. Thus, our framework allows obtaining (Bregman) proximal gradient methods  \cite{chen1993convergence,parikh2014prox} as a special case and extends them to the case of inexact oracle setting.
In particular, based on this model (with Bregman divergence to be Kullback--Leibler divergence), we propose in \cite{stonyakin2019gradient} proximal Sinkhorn's algorithm for Wasserstein distance calculation problem and in \cite{kroshnin2019complexity} proximal IBP for Wasserstein barycenter problem.

\end{example}

\begin{example} {\bf Min-min problem.}
Assume that in \eqref{eq:Problem} $f(x) := \min_{z \in Z}F(z,x)$, where
the set $Z$ is convex and bounded,
the function $F$ is smooth and convex w.r.t. all variables. Moreover, assume that
$$\norm{\nabla F(z',x') -\nabla F(z,x)}_2 \leq L \norm{(z',x') -(z,x)}_2,\,\,\forall z,z'\in Z,\,x,x'\in \R^n.$$
Let $V[y](x)=\frac{1}{2}\|x-y\|_2^2$. \rev1{Assume that for any $y \in Q$ it is possible to find an approximate solution of $\min_{z \in Z}F(z,x)$, i.e. a point $\widetilde{z}_\delta(y) \in Z$ such that}
\begin{gather*}
\langle\nabla_z F(\widetilde{z}_\delta(y), y), z - \widetilde{z}_\delta(y)\rangle \geq -\delta, \,\,\,\forall z \in Z.
\end{gather*}
Then
$F(\widetilde{z}_\delta(y), y) - f(y) \leq \delta$
and $\psi_{\delta}(x,y) = \langle\nabla_z F(\widetilde{z}_\delta(y), y), x - y\rangle$
is  $(6\delta, 2L,0,0,V)$-model of $f$ with $f_\delta(y) = F(\widetilde{z}_\delta(y), y) - 2\delta$ at a given point $y$.
\end{example}

\attt{
\begin{example}{\bf Inexact oracle, \cite{devolder2014first}.}
We say that $f$ is equipped with $(\delta, L)$-oracle, if for any $y \in Q$ we can find a pair $(f_{\delta}(y), \nabla f_{\delta}(y))$ such that
\begin{equation*}
0 \leq {f(x) - \left(f_{\delta}(y) + \la\nabla f_{\delta}(y), x - y \ra\right)} \leq {\frac{L}{2}\norm{x - y}^2 + \delta},
\end{equation*}
for all $x \in Q$ (see \cite{devolder2014first}). Clearly, under the assumption (1-SC),  the function $\la\nabla f_{\delta}(y), x - y \ra$ is a $(\delta, L)$-model.
Examples of inexact models in \ref{saddle_point_problem}, \ref{augmented_lagrangians}, \ref{Moreau_ex} are also described by inexact $(\delta, L)$--oracle. The detailed proofs can be found in \cite{devolder2014first}.
\end{example}}

 \begin{example}{\bf Saddle-point problem, \cite{devolder2014first}.}
 \label{saddle_point_problem}
Assume that in \eqref{eq:Problem}  $Q =  \R^n $, 
 $f(x) = \max_{z \in Z}\left[\langle x, b - Az\rangle - \phi(z)\right]$,
 where $\phi(z)$ is a $\mu$-strongly convex function w.r.t. $p$-norm ($1\leq p\leq2$). Then $f$ is smooth and convex and its gradient is Lipschitz continuous with constant $L = \frac{1}{\mu}\max_{\norm{z}_p\leq1}\norm{Az}_2^2$.
 If $z_\delta(y)\in Z$ is an approximate solution to the auxiliary max-problem, i.e.
 $$\max_{z \in Z}\left[\langle y, b - Az\rangle - \phi(z)\right] - \left[\langle y, b - Az_{\delta}(y)\rangle - \phi(z_{\delta}(y))\right]\le \delta,$$
 then
$
 \psi_{\delta}(x,y) =\langle b - Az_\delta(y), x - y\rangle
 $
 is  $(\delta, 2L,0,0,V)$-model of $f$ with $f_\delta(y) = \langle y, b - Az_\delta(y)\rangle - \phi(z_\delta(y))$ at the point $y$ if we define $V[y](x)=\frac{1}{2}\|x-y\|_2^2$.
 \end{example}

 \begin{example}
 {\bf Augmented Lagrangians, \cite{devolder2014first}.}
 \label{augmented_lagrangians}
 Let $\phi(z)$ be a smooth function. Consider the problem
 \begin{gather*}
 \min_{Az=b,\,z\in Q} \phi(z) + \frac{\mu}{2}\norm{Az -b}_2^2
 \end{gather*}
 and the corresponding dual problem
 \begin{align*}
 \min_{x \in \R^n} \left\{f(x) = \max_{z \in Q}\underbrace{\left(\langle x, b - Az\rangle - \phi(z) - \frac{\mu}{2}\norm{Az -b}_2^2\right)}_{\Lambda(x,z)}\right\} .
 \end{align*}
 If $z_\delta(y)$ is an approximate solution of auxiliary max-problem, i.e.
 \att{
  \begin{gather*}
 \max_{z \in Q}\left\langle\nabla_z\Lambda(y,z_\delta(y)),z - z_\delta(y)\right\rangle \leq \delta,
 \end{gather*}}
 then $ \psi_{\delta}(x,y) =\langle b - Az_\delta(y), x - y\rangle$
 is  $(\delta, \mu^{-1},0,0,V)$-model of $f$ with $$f_\delta(y) = \langle y, b - Az_\delta(y)\rangle - \phi(z_\delta(y))- \frac{\mu}{2}\norm{Az_\delta(y) -b}_2^2$$ at the point $y$ if we take $V[y](x)=\frac{1}{2}\|x-y\|_2^2$.
\end{example}

\begin{example}
{\bf Moreau envelope of the objective function, \cite{devolder2014first}.}
 \label{Moreau_ex}
 Let us consider the following optimization problem:
 \begin{align}
 \min_{x \in \R^n} \left\{f_L(x) := \min_{z \in Q}\underbrace{\left\{f(z) + \frac{L}{2}\norm{z - x}^2_2\right\}}_{\Lambda(x,z)}\right\}.
 \label{prox}
 \end{align}
 Assume that $f$ is convex  and, for some $z_L(y)$,
 \begin{gather*}
 \max_{z \in Q}\left\{\Lambda(y,z_L(y)) - \Lambda(y,z) + \frac{L}{2}\norm{y - z_L(y)}^2_2\right\} \leq \delta.
 \end{gather*}
 Then $ \psi_{\delta}(x,y) =\langle L(y - z_L(y)), x - y\rangle$  is $(\delta, L,0,0,V)$-model of $f$ with $$f_\delta(y) = f(z_L(y)) + \frac{L}{2}\norm{z_L(y) - y}^2_2 - \delta$$ at the point $y$ if we take $V[y](x)=\frac{1}{2}\|x-y\|_2^2$.
 \end{example}

\begin{example}
{\bf Clustering by Electorial Model, \cite{stonyakin2019gradient}.}
 \label{Clustering_electorial}
Another example of an optimization problem that allows for $(\delta, L, 0, m, V)$-model with strong convexity of the function $\psi_\delta(x, y)$  is proposed in \cite{stonyakin2019gradient} to address a \emph{non-convex} optimization problem which arises in an electoral model for clustering introduced in \cite{RePEc:cor:louvco:2018001}.
In this model, voters (data points) select a party (cluster) iteratively by minimizing the following function
\begin{equation*}\label{Nesterov_Electoral_Model}
\min_{z \in S_n(1), p \in \mathbb{R}^m_{+}} \left\{f_{\mu_1, \mu_2}(x = (z, p)) = g(x) + \mu_1 \sum\limits_{k=1}^{n}z_k \ln z_k + \frac{\mu_2}{2} \|p\|^2_2 \right\}
\end{equation*}
where $S_n(1)$ is the standard unit simplex in $\mathbb{R}^n$.
Let us choose $\norm{x}^2 = \norm{z}^2_1 + \norm{p}^2_2$ and assume that, in general non-convex, $g(x)$ has $L_g$--Lipschitz continuous gradient
\begin{equation*}
    \norm{\nabla g(x) - \nabla g(y)}_* \leq L_g \norm{x - y} \quad \forall x, y \in S_n(1) \times \mathbb{R}^m_{+}
\end{equation*}
and $L_g \leq \mu_1$, and $L_g \leq \mu_2$.
It can be shown (see \cite{stonyakin2019gradient}) that
\begin{equation*}
\begin{gathered}
\psi_{\delta}(x, y) = \langle \nabla g(y), x-y \rangle  - L_g \cdot {\rm KL}(z_x|z_y) - \frac{L_g }{2}\|p_x-p_y\|^2_2 \\ + \mu_1 ({\rm KL}(z_x | \textbf{1}) - {\rm KL}(z_y | \textbf{1})) + \frac{\mu_2 }{2} \left(\|p_x\|^2_2-\|p_y\|^2_2\right)
\end{gathered}
\end{equation*}
is a $(0, 2L_g,0,\min\{\mu_1, \mu_2\} - L_g,V)$-model of $f_{\mu_1, \mu_2}(x)$. Here ${\rm KL}(z_x | z_y) = \sum_{i=1}^m [z_x]_i \ln ([z_x]_i/[z_y]_i)$ and
\begin{eqnarray*}
V[y](x) = {\rm KL}(z_x | z_y) + \frac{1}{2}\|p_x - p_y\|^2_2.
\end{eqnarray*}

We finish this section by defining an approximate solution to an optimization problem. This definition will be used to allow inexact solutions of auxiliary minimization problems in each iteration of our algorithms.
\end{example}

\begin{definition}
\label{solNemirovskiy}
For a convex optimization problem
$\min_{x \in Q} \Psi(x)$
\attt{and $\widetilde{\delta} \geq 0$}, we denote by $\text{Arg}\min_{x \in Q}^{\widetilde{\delta}}\Psi(x)$~a set of points  $\widetilde{x}$ such  that
\begin{gather}\label{eqv_inex_sol}
\gav{\exists} h \in \partial\Psi(\widetilde{x})\gav{:\forall x \in Q} \,\, \gav{\to}\, \langle h, x - \widetilde{x} \rangle \geq -\widetilde{\delta}.
\end{gather}
We denote by $\argmin_{x \in Q}^{\widetilde{\delta}}\Psi(x)$  some element of $\text{Arg}\min_{x \in Q}^{\widetilde{\delta}}\Psi(x)$.
\end{definition}

\section{Gradient Method with Inexact Model.}\label{GM}
In this section, we consider adaptive gradient-type methods for problems with the $(\delta, L, \mu, m, V)$-model of the objective. First, we consider the non-accelerated gradient method and then its accelerated version. \rev1{For each algorithm, we obtain convergence rates that have two terms. The first term decreases as iterations go and corresponds to the case of the exact model with $\delta=0$ and the exact solution of the auxiliary problems with $\widetilde{\delta}=0$. The second term is non-decreasing as iterations go and corresponds to errors in the model and the solutions of the auxiliary problems.}
\rev1{Non-accelerated Algorithm \ref{Alg2} 
has a slower decaying part, but the constant term depending on errors $\delta, \widetilde{\delta}$. Accelerated Algorithm \ref{FastAlg2_strong} has a faster decaying term, but increasing term depending on errors $\delta, \widetilde{\delta}$. Moreover, to obtain acceleration, we have to consider Algorithm \ref{FastAlg2_strong} for a narrower class of problems with $(\delta, L, \mu, m, V,\|\cdot \|)$-models (for a precise definition, see Definition \ref{defRelStronglyConvexFast}).} 


\subsection{Adaptive Gradient Method with $(\delta, L, \mu, m, V)$-Model}\label{SC}

In this section, we consider the adaptive gradient method for problem \eqref{eq:Problem}, which uses a $(\delta, L, \mu, m, V)$-model of the objective. For the case when $\mu + m > 0$, our method has a linear convergence, \pd{and for a more general case $\mu=0$ and $m = 0$, we prove a sublinear convergence rate}.

We assume that in each iteration $k$, the method has access to $(\delta, \bar{L}_{k+1}, \mu, m, V)$-model of $f$ (see Definition~\ref{defRelStronglyConvex}). 
\pd{In general, constant $\bar{L}_{k+1}$ may vary from iteration to iteration and we only assume that the $(\delta, \bar{L}_{k+1}, \mu, m, V)$-model exists. We do not use $\bar{L}_{k+1}$ in Algorithm \ref{Alg2} explicitly and, moreover, our method is adaptive to this constant. \attt{Constant $\widetilde{\delta}$ is an error in terms of Definition~\ref{solNemirovskiy} and there is no need to know it explicitly in Algorithm \ref{Alg2}, $\widetilde{\delta}$ only appears in Theorem~\ref{Th:str_conv_adap}.}} 

\begin{algorithm}
\caption{Adaptive gradient method with $(\delta, L, \mu, m, V)$-model}
\label{Alg2}
\begin{algorithmic}[1]
\STATE \textbf{Input:} $x_0$ is the starting point, 
$\delta>0$ is the oracle error and $L_0 >0$.
\STATE Set
 $S_0 := 0 $
\FOR{$k \geq 0$}
\STATE Find the smallest \ga{integer} $i_k\geq 0$ such that
\begin{equation}\label{exitLDL_G_S}
f_{\delta}(x_{k+1}) \leq f_{\delta}(x_{k}) + \psi_{\delta}(x_{k+1}, x_{k}) +L_{k+1}V[x_{k}](x_{k+1}) + \delta,
\end{equation}
where \revv{$L_{k+1} = 2^{i_k-1}L_k$}, $\alpha_{k+1}:= \frac{1}{L_{k+1}}$, $S_{k+1} := S_k + \alpha_{k+1}$, and, for some $\widetilde{\delta}\geq 0$
\begin{equation}\label{equmir2DL_G_S}
\phi_{k+1}(x) := \psi_{\delta}(x, x_k)+L_{k+1}V[x_k](x), \quad
x_{k+1} := {\arg\min_{x \in Q}}^{\widetilde{\delta}} \phi_{k+1}(x).
\end{equation}
\rev1{Here the inexact minimization is in accordance with Definition~\ref{solNemirovskiy}.}
\ENDFOR
\end{algorithmic}
\end{algorithm}

For \att{$L_k > \mu$} and all $k \geq 0$, we denote
$$
q_k \eqdef \frac{L_{k}-\mu}{L_{k} + m} \rev1{= 1- \frac{m+\mu}{L_{k} + m} \leq 1 }, 
$$
and $q_k = 0$ for $L_k \leq \mu$. We assume that $Q_j^k \eqdef \prod_{i=j}^{k}q_{i}$ and $Q_j^k = 1$ for $j > k$.

\begin{theorem}\label{Th:str_conv_adap}
Assume that $\psi_{\delta}(x, y)$ is a $(\delta, L, \mu, m, V)$-model according to Definition \ref{defRelStronglyConvex}.
Denote by \revvat{$y_{N} = x_{j(N)}$, where $j(N) = \argmin_{k = 1,\dots, N}f_\delta(x_k)$}. Then, after $N$ iterations of Algorithm \ref{Alg2}, we have
\begin{align}
f(y_{N}) - f(x_*)
&\leq \min\left\{(L_{N} + m) Q_1^{N}, \frac{1}{\sum_{i=1}^{N}\frac{1}{L_{i} + m}}\right\}V[x_{0}](x_*) + \widetilde{\delta} + 3\delta, \label{analysis_algorithm_grad:ineq_2}\\
    V[x_{N}](x_*) &\leq Q_1^{N} V[x_0](x_*) + (\widetilde{\delta} + 2\delta)\sum_{i=1}^{N}\frac{Q_{i+1}^{N}}{L_{i} + m}. \label{analysis_algorithm_grad:ineq_1}
\end{align}
\end{theorem}

To prove Theorem \ref{Th:str_conv_adap} we need the following lemma.

\begin{lemma}
    \label{lemma:str_conv}
	Let $\psi(x)$ be a \attt{relatively} $m$-strongly convex function \attt{relative to $d$} with $m \geq 0$, \rev1{i.e. \eqref{mstongmodel} holds,} and
	\begin{equation}
	\label{eq:lm:str_conv_y}
	y = {\argmin_{x \in Q}}^{\widetilde{\delta}} \{\psi(x) + \beta V[z](x)\},
	\end{equation}
	where $\beta \geq 0$.
Then
	\begin{equation*}
	\psi(x) + \beta V[z](x) \geq \psi(y) + \beta V[z](y) + (\beta + m)V[y](x) - \widetilde{\delta} ,\,\,\, \forall x \in Q.
	\end{equation*}
\end{lemma}

\begin{proof}
	\pd{By optimality conditions in \eqref{eq:lm:str_conv_y} and} Definition \ref{solNemirovskiy}, 
	\begin{gather*}
		\exists g \in \partial\psi(y), \,\,\, \langle g + \att{\beta}\nabla_y V[z](y), x - y \rangle \geq -\widetilde{\delta} ,\,\,\, \forall x \in Q.
	\end{gather*}
	\rev1{From this we obtain}
	\begin{gather*}
		\psi(x) - \psi(y) \stackrel{\rev1{\eqref{mstongmodel}}}{\geq} \langle g, x - y \rangle + m V[y](x) \geq \langle \att{\beta}\nabla_y V[z](y), y - x \rangle - \widetilde{\delta} + m V[y](x),
	\end{gather*}
 which together with equality
	\begin{gather*}
	\langle \nabla_y V[z](y), y - x \rangle=\langle \nabla d(y) - \nabla d(z), y - x \rangle=d(y) - d(z) - \langle \nabla d(z), y - z \rangle +\\ + d(x) - d(y) - \langle \nabla d(y), x - y \rangle - d(x) + d(z) + \langle \nabla d(z), x - z \rangle=\\=
	V[z](y) + V[y](x) - V[z](x)
	\end{gather*}
complete the proof. \rev1{The latter chain of equalities follows from the definition of Bregman divergence, see the beginning of Section \ref{S:Inexact_min}.}
\end{proof}


\begin{proof}[Proof of Theorem \ref{Th:str_conv_adap}]
Since by Definition \ref{defRelStronglyConvex} with $x = y$, \rev1{we have} 
\revvat{
\begin{align}
\label{analysis_algorithm_grad:f_delta}
f(x) - \delta \leq  f_{\delta}(x)  \leq f(x),
\end{align} 
}
and \rev1{using} \eqref{exitLDL_G_S}, we have the following series of inequalities
\begin{align*}
f(x_{N})
&\leq f_{\delta}(x_{N}) + \delta \stackrel{\rev1{\eqref{exitLDL_G_S}}}{\leq} f_{\delta}(x_{N-1}) + \psi_{\delta}(x_{N}, x_{N-1}) +L_{N}V[x_{N-1}](x_{N}) + 2\delta.
\end{align*}
Using Lemma \ref{lemma:str_conv} for \eqref{equmir2DL_G_S} with \rev1{$y=x_N$, $\beta = L_N$, $z=x_{N-1}$, $\psi(x) = \psi_{\delta}(x,x_{N-1})$,} we have
\begin{align*}
f(x_{N})
&\leq f_{\delta}(x_{N-1}) + \psi_{\delta}(x, x_{N-1}) +L_{N}V[x_{N-1}](x) - \left(L_{N} + m\right)V[x_{N}](x) + \widetilde{\delta} + 2\delta.
\end{align*}
\rev1{Further, applying} the left inequality \eqref{eq:Str_Conv_Model}, we have
\begin{align}
\label{analysis_algorithm_grad:prove_1}
f(x_{N}) \leq f(x) + (L_{N} - \mu)V[x_{N-1}](x) - \left(L_{N} + m\right)V[x_{N}](x) + \widetilde{\delta} + 2\delta.
\end{align}
\revv{Let us first consider the case $L_N>\mu$.}
Taking $x = x_*$ in the above inequality and using inequality $f(x_*) \leq f(x_{N})$, we obtain
\begin{align}
\label{analysis_algorithm_grad:prove_2}
\left(L_{N} + m\right)V[x_{N}](x_*) &\leq (L_{N} - \mu)V[x_{N-1}](x_*) + \widetilde{\delta} + 2\delta. 
\end{align}
Thus, recursively we have that
\begin{align*}
V[x_{N}](x_*) &\leq q_{N} V[x_{N-1}](x_*) + \frac{\widetilde{\delta} + 2\delta}{L_{N} + m}\leq Q_1^{N} V[x_0](x_*) + (\widetilde{\delta} + 2\delta)\sum_{i=1}^{N}\frac{Q_{i+1}^{N}}{L_{i} + m}.
\end{align*}
The last inequality proves \eqref{analysis_algorithm_grad:ineq_1}. Now we rewrite \eqref{analysis_algorithm_grad:prove_1} for $x = x_*$ as
\begin{align*}
V[x_{N}](x_*) \leq \frac{1}{L_{N} + m}(f(x_*) - f(x_{N}) + \widetilde{\delta} + 2\delta) +  q_{N}V[x_{N-1}](x_*).
\end{align*}
Recursively, we have
\begin{align*}
V[x_{N}](x_*) \leq \sum_{i=1}^{N}\left(\frac{Q_{i+1}^{N}}{L_{i} + m}(f(x_*) - f(x_{i})+ \widetilde{\delta} + 2\delta)\right) + Q_{1}^{N} V[x_{0}](x_*).
\end{align*}
\revvat{
From the definition of $y_{N}$ and \eqref{analysis_algorithm_grad:f_delta}, we obtain
\begin{align}
\label{analysis_algorithm_grad:f_y_n}
f(y_N) - \delta \leq f_\delta(y_N) \leq f_\delta(x_k) \leq f(x_k), \quad \forall k \in [1, N].
\end{align}
}
Using that $V[x_{N}](x_*) \geq 0$ and \revvat{\eqref{analysis_algorithm_grad:f_y_n}}, we get
\begin{align*}
Q_{1}^{N} V[x_{0}](x_*)
&\geq \sum_{i=1}^{N}\left(\frac{Q_{i+1}^{N}}{L_{i} + m}(f(x_i) - f(x_*) - \widetilde{\delta} - 2\delta)\right)\\
&\geq (f(y_{N}) - f(x_*))\sum_{i=1}^{N}\frac{Q_{i+1}^{N}}{L_{i} + m} - (\widetilde{\delta} + 3\delta)\sum_{i=1}^{N}\frac{Q_{i+1}^{N}}{L_{i} + m}.
\end{align*}
\pd{Dividing} by $\sum_{i=1}^{N}\frac{Q_{i+1}^{N}}{L_{i} + m}$, \pd{we obtain}
\begin{align*}
f(y_{N}) - f(x_*)
&\leq \frac{Q_{1}^{N}}{\sum_{i=1}^{N}\frac{Q_{i+1}^{N}}{L_{i} + m}} V[x_{0}](x_*) + \widetilde{\delta} + 3\delta.
\end{align*}
\pd{Since} $\sum_{i=1}^{N}\frac{Q_{i+1}^{N}}{L_{i} + m} \geq \frac{1}{L_{N} + m}$ and $Q_{i}^{N} \geq Q_{1}^{N}$ for all $i \geq 1$, we get
\begin{align*}
f(y_{N}) - f(x_*) \leq \min\left\{(L_{N} + m) Q_1^{N}, \frac{1}{\sum_{i=1}^{N}\frac{1}{L_{i} + m}}\right\}V[x_{0}](x_*) + \widetilde{\delta} + 3\delta.
\end{align*}

This proves \eqref{analysis_algorithm_grad:ineq_2}.

\revv{
In the opposite case $L_N\leq \mu$, we obtain from \eqref{analysis_algorithm_grad:f_y_n}, \eqref{analysis_algorithm_grad:prove_2} and \eqref{analysis_algorithm_grad:prove_1} with $x=x_*$ that
\begin{align*}
&V[x_{N}](x_*) \leq \frac{ \widetilde{\delta} + 2\delta}{L_{N} + m}\\
&f(y_N)\leq f(x_{N}) + \delta \leq f(x_*) + \widetilde{\delta} + 3\delta.
\end{align*}
The r.h.s. of these two inequalities can be bounded from above by respectively the r.h.s. of \eqref{analysis_algorithm_grad:ineq_1} and \eqref{analysis_algorithm_grad:ineq_2}, which proves the latter inequalities. 
}
\end{proof}

\begin{remark}
\label{remark:grad_method}
Let us assume that $L_0 \leq L$, and we know that $\bar{L}_{k+1} \leq L$ for all $k \geq 0$ (or in other words, $(\delta, L, \mu, m, V)$-model exists for all $k \geq 0$). This means that \rev1{if $L_k$ in Algorithm \ref{Alg2} satisfies $L_k \geq L$, then inequality \eqref{exitLDL_G_S} holds.} Thus, $L_k \leq 2L$ for all $k \geq 0$.
From this fact we can obtain that $\sum_{i=1}^{N}\frac{1}{L_{i} + m} \geq \frac{N}{2L + m}$ and $q_k \leq q \eqdef \frac{2L - \mu}{2L + m} = 1-\frac{m+\mu}{2L+m}$. In view of \eqref{analysis_algorithm_grad:ineq_2} and \eqref{analysis_algorithm_grad:ineq_1}, we have
\begin{align*}
	f(y_N) - f(x_*) &\leq \min\left\{\frac{2L + m}{N}, (2L + m) q^{N}\right\}V[x_0](x_*) + \widetilde{\delta} + 3\delta,
\end{align*}
\begin{align*}
    V[x_{N}](x_*) &\leq q^{N} V[x_0](x_*) + (\widetilde{\delta} + 2\delta)\sum_{i=1}^{N}\frac{Q_{i+1}^{N}}{L_{i} + m}.
\end{align*}
\revv{The first term in the r.h.s. of the first equation above corresponds to exact gradient methods in the relatively smooth setting \cite{bauschke2016descent,lu2018relatively}. According to \cite{dragomir2019optimal}, this part of our bound can be interpreted as optimal since \cite{dragomir2019optimal} show that it can not be improved without additional assumptions on the Bregman divergence $V$.}
\end{remark}

\begin{remark}
\pd{The advantage of Algorithm \ref{Alg2} is that there is no need to know the true values of \ga{the parameters $L$ }\pd{ and $m$}.
Using the standard argument \cite{dvurechensky2016primal-dual} one can show that the number of oracle calls is less than $2N+\log_2\frac{2L}{L_0}$, where $N$ is the number of iterations of Algorithm \ref{Alg2}.}
\end{remark}

\begin{remark}
\attt{In view of Theorem \ref{Th:str_conv_adap}, if we would like to obtain the residual in the objective \eqref{analysis_algorithm_grad:ineq_2} to be smaller than $\varepsilon$, then it is sufficient to ensure $\widetilde{\delta} = O(\varepsilon)$  and $\delta = O(\varepsilon)$.} 
\end{remark}

\subsection{Adaptive Fast Gradient Method with $(\delta, L, \mu, m, V,\|\cdot \|)$-model}

In this section, we \rev1{slightly change the definition of inexact model and introduce $(\delta, L, \mu, m, V,\|\cdot \|)$-model of the objective (see Definition \ref{defRelStronglyConvexFast})}. This change allows us to obtain an accelerated method for problems with this type of the model.
The method is close to the accelerated mirror-descent type of methods (see \cite{tseng2008accelerated,lan2012optimal,dvurechensky2018computational}).
{
\rev1{In the new definition, we} use the square of the norm in the r.h.s. of \eqref{eq:Str_Conv_Model} instead of the function $V$, which gives the following modification of Definition \ref{defRelStronglyConvex}.
\begin{definition}\label{defRelStronglyConvexFast}
Let $\delta,L,\mu,m \geq 0$. We say that $\psi_{\delta}(x, y)$ is a $(\delta, L, \mu, m, V,\|\cdot \|)$-model of the function $f$ at a given point $y$
iff, for all $x \in Q$
\begin{equation}
\label{eq:Str_Conv_Model_Fast}
{\mu V[y](x)} \leq {f(x) - \left(f_{\delta}(y) + \psi_{\delta}(x, y)\right)} \leq {\frac{L}{2}\|x-y\|^2 + \delta}
\end{equation}
and $\psi_{\delta}(x, y)$ \rev1{is convex in $x$,} satisfies $\psi_{\delta}(x,x)=0$ for all $x \in Q$ and
\begin{equation}\label{mstongmodel_Fast}
\psi(x) \geqslant \psi(z) + \la \attt{g}, x- z \ra + m V[z](x), \quad \forall x, z \in Q \,\,\attt{\forall g \in \partial \psi(z)},
\end{equation}
where for fixed $y \in Q$ and any $x \in Q$ we denote $\psi(x) = \psi_{\delta}(x, y)$. \attt{In other words, $\psi(x)$ is relatively $m$-strongly convex relative to $d$.}
\end{definition}
}

As in the previous subsection, we assume that \pd{there exists} some constant $\bar{L}_{k+1}$ such that $(\delta_k, \bar{L}_{k+1}, \mu, m, V,\|\cdot \|)$-model of $f$ exists at $k$-th step ($k=0,..,N-1$) of Algorithm \ref{FastAlg2_strong}. \att{Unlike Algorithm \ref{Alg2}, we assume that \pd{the errors $\widetilde{\delta},\delta$ can depend on the iteration counter $k$, which is indicated by input} sequences $\{\widetilde{\delta}_k\}_{k\geq 0}$ and $\{\delta_k\}_{k\geq 0}$}. For instance, this allows obtaining the Universal Fast Gradient Method in which different values of $\{\delta_k\}_{k\geq 0}$ are required (see \cite{nesterov2015universal,baimurzina2017universal}) in each iteration. \attt{Moreover, \moh{the} sequence $\{\widetilde{\delta}_k\}_{k\geq 0}$ is not used explicitly and only appears in convergence rates of  Theorem~\ref{Th:fast_str_conv_adap}}.

\begin{theorem}
\label{Th:fast_str_conv_adap}
\pd{
Assume that  $\psi_{\delta}(x, y)$ is a $(\delta, L, \mu, m, V,\|\cdot \|)$-model according to Definition \ref{defRelStronglyConvexFast}. Also assume that $V[y](x)$ satisfies (1-SC) condition \ga{w.r.t. $\|\cdot\|$-norm}. Then, after $N$ iterations of Algorithm \ref{FastAlg2_strong}, we have
}
\begin{align}
    &f(x_N) - f(x_*) \leq \frac{V[u_0](x_*)}{A_N} + \frac{2\sum_{k=0}^{N-1}A_{k+1}\delta_k}{A_N} + \frac{\sum_{k=0}^{N-1}\widetilde{\delta}_k}{A_N}, \label{Th:fast_str_conv_adap:result_1} \\
    &V[u_N](x_*)\leq \frac{V[u_0](x_*)}{(1 + A_N\mu + \att{A_Nm})} + \frac{2\sum_{k=0}^{N-1}A_{k+1}\delta_k}{(1 + A_N\mu + \att{A_Nm})} + \frac{\sum_{k=0}^{N-1}\widetilde{\delta}_k}{(1 + A_N\mu + \att{A_Nm})}. \label{Th:fast_str_conv_adap:result_2}
    \end{align}
\end{theorem}

\begin{remark}
Despite the adaptive structure of Algorithm~\ref{FastAlg2_strong}, as in \cite{nesterov2015universal}, it can be shown that on average the algorithm up to logarithmic terms requires four \rev1{inexact computations of the objective} function and two computations of $(\delta, L, \mu, m, V,\|\cdot \|)$-model per iteration.
\end{remark}

\begin{algorithm}
\caption{{Fast adaptive gradient method with $(\delta, L,
\mu, m, V,\|\cdot \|)$-model}}
\label{FastAlg2_strong}
\begin{algorithmic}[1]
\STATE \textbf{Input:} $x_0$ is the starting point,
$\mu \geq 0$, $m \geq 0$,
$\{\delta_k\}_{k\geq 0}$ and
$ L_{0} > 0$. 
\STATE Set
$y_0 := x_0$, $u_0 := x_0$, $\alpha_0 := 0$, $A_0 := \alpha_0$
\FOR{$k \geq 0$}
\STATE Find the smallest \ga{integer} $i_k \geq 0$ such that
\vspace{-0.25cm}
\begin{equation}
\begin{gathered}
f_{\delta_k}(x_{k+1}) \leq f_{\delta_k}(y_{k+1}) + \psi_{\delta_k}(x_{k+1}, y_{k+1}) +\frac{L_{k+1}}{2}\norm{x_{k+1} - y_{k+1}}^2 + \delta_k,
\label{exitLDL_strong}
\end{gathered}
\end{equation}

where $L_{k+1} = 2^{i_k-1}L_k$, $\alpha_{k+1}$ is the largest root  of \moh{the} \rev1{equation}
\begin{gather}
\label{alpha_def_strong}
A_{k+1}{(1 + A_k \mu + \att{A_k m})}=L_{k+1}\alpha^2_{k+1},\quad A_{k+1} := A_k + \alpha_{k+1}, \;\text{and}
\end{gather}
\vspace{-0.7cm}
\begin{gather}
y_{k+1} := \frac{\alpha_{k+1}u_k + A_k x_k}{A_{k+1}}, \label{eqymir2DL_strong}
\end{gather}
\vspace{-0.7cm}
\begin{equation*}
\phi_{k+1}(x)=\alpha_{k+1}\psi_{\delta_k}(x, y_{k+1}) + {(1 + A_k\mu + \att{A_k m})} V[u_k](x) + {\alpha_{k+1} \mu V[y_{k+1}](x)},
\end{equation*}
\vspace{-0.7cm}
\begin{equation}\label{equmir2DL_strong}
u_{k+1} := {\argmin_{x \in Q}}^{\widetilde{\delta}_k}\phi_{k+1}(x), \;\;\text{\rev1{for some $\widetilde{\delta}_k \geq 0$}}
\end{equation}
\vspace{-0.7cm}
\begin{gather}
x_{k+1} := \frac{\alpha_{k+1}u_{k+1} + A_k x_k}{A_{k+1}}. \label{eqxmir2DL_strong}
\end{gather}
\ENDFOR
\end{algorithmic}
\end{algorithm}

In order to prove Theorem \ref{Th:fast_str_conv_adap} we need the following lemma.

\begin{lemma}
    \label{lemma:fast_str_conv}
	Let $\psi(x)$ be a \attt{relatively} $m$-strongly convex function \attt{relative to $d$} with $m \geq 0$, \rev1{i.e. \eqref{mstongmodel} holds,} and
	\begin{gather*}
	y = {\argmin_{x \in Q}}^{\widetilde{\delta}} \{\psi(x) + \beta V[z](x) + \gamma V[u](x)\},
	\end{gather*}
	where $\beta \geq 0$ and $\gamma \geq 0$.
Then
	\begin{equation*}
	\psi(x) + \beta V[z](x) + \gamma V[u](x) \geq \psi(y) + \beta V[z](y) + \gamma V[u](y) + (\beta + \gamma + m)V[y](x) - \widetilde{\delta} ,\,\,\, \forall x \in Q.
	\end{equation*}
\end{lemma}

\pd{We omit the proof of Lemma \ref{lemma:fast_str_conv} since it is similar to} the proof of Lemma \ref{lemma:str_conv}. 

\begin{lemma}
\label{lemma:fast_strong_main}
	For all $x \in Q$, we have 
	\begin{align*}
    		&A_{k+1} f(x_{k+1}) - A_{k} f(x_{k}) + ({1 + A_{k+1} \mu + \att{A_{k+1} m}}) V[u_{k+1}](x) - ({1 + A_k \mu + \att{A_k m}})V[u_{k}](x)\\
    		&\leq \alpha_{k+1}f(x) + 2\delta_k A_{k+1} + \widetilde{\delta}_k.
	\end{align*}
\end{lemma}
\begin{proof}
Since by Definition \ref{defRelStronglyConvexFast} with $x = y$, one has $f(x) - \delta \leq  f_{\delta}(x)  \leq f(x)$, and using \eqref{exitLDL_strong}, we have
	\begin{align*}
	f(x_{k+1}) \stackrel{\rev1{\eqref{eq:Str_Conv_Model_Fast}}}{\leq} f_{\delta_k}(y_{k+1}) + \delta_k \stackrel{\rev1{\eqref{exitLDL_strong}}}{\leq} f_{\delta_k}(y_{k+1}) + \psi_{\delta_k}(x_{k+1},y_{k+1})  + \frac{L_{k+1}}{2}\norm{x_{k+1} - y_{k+1}}^2 + 2\delta_k.
	\end{align*}
Substituting in this expression definitions \eqref{eqxmir2DL_strong} and \eqref{eqymir2DL_strong} of the points $x_{k+1}$, $y_{k+1}$, we obtain that
	\begin{align*}
	f(x_{k+1})
	&\leq f_{\delta_k}(y_{k+1}) + \psi_{\delta_k}\left(\frac{\alpha_{k+1}u_{k+1} + A_k x_k}{A_{k+1}},y_{k+1}\right)\\
	&\hspace{2em}+ \frac{L_{k+1}}{2}\norm{\frac{\alpha_{k+1}u_{k+1} + A_k x_k}{A_{k+1}} - \rev1{\frac{\alpha_{k+1}u_k + A_k x_k}{A_{k+1}}}}^2 + 2\delta_k\\
	&= f_{\delta_k}(y_{k+1}) + \psi_{\delta_k}\left(\frac{\alpha_{k+1}u_{k+1} + A_k x_k}{A_{k+1}},y_{k+1}\right)+\frac{L_{k+1} \alpha^2_{k+1}}{2 A^2_{k+1}}\norm{u_{k+1} - u_k}^2 + 2\delta_k.
	\end{align*}
Since \rev1{$A_{k+1}=A_k + \alpha_{k+1}$ and}, \rev1{by Definition \ref{defRelStronglyConvexFast},} $\psi_{\delta_k}(\cdot,y)$ is convex, we have
	\begin{align*}
	f(x_{k+1})
	&\leq\frac{A_k}{A_{k+1}}\left(f_{\delta_k}(y_{k+1}) + \psi_{\delta_k}(x_k, y_{k+1})\right)+\frac{\alpha_{k+1}}{A_{k+1}}\left(f_{\delta_k}(y_{k+1}) +
	 \psi_{\delta_k}(u_{k+1}, y_{k+1})\right)\\
	 &\hspace{2em}+ \frac{L_{k+1} \alpha^2_{k+1}}{2 A^2_{k+1}}\norm{u_{k+1} - u_k}^2 + 2\delta_k.
   \end{align*}
In view of the definition \eqref{alpha_def_strong} for the sequence $\alpha_{k+1}$, we obtain
   \begin{align*}
	 f(x_{k+1})&\leq\frac{A_k}{A_{k+1}}\left(f_{\delta_k}(y_{k+1}) + \psi_{\delta_k}(x_k,y_{k+1})\right) +\frac{\alpha_{k+1}}{A_{k+1}}\Big(f_{\delta_k}(y_{k+1}) + \psi_{\delta_k}(u_{k+1},y_{k+1})\\
	 &\hspace{2em}+ \frac{{1 + A_k \mu + \att{A_k m}}}{2 \alpha_{k+1}}\norm{u_{k+1} - u_k}^2\Big) + 2\delta_k.
   \end{align*}
Using (1-SC) condition w.r.t. norm for $V$ and the left inequality in \rev1{\eqref{eq:Str_Conv_Model_Fast}}, we get
   \begin{align}
   \begin{split}
   \label{lemma:fast_str_conv:ineq_1}
	 f(x_{k+1}) &\leq \frac{A_k}{A_{k+1}}f_{\delta_k}(x_{k})+\frac{\alpha_{k+1}}{A_{k+1}}\Big(f_{\delta_k}(y_{k+1}) + \psi_{\delta_k}(u_{k+1},y_{k+1})\\
	 &\hspace{2em}+ \frac{{1 + A_k \mu + \att{A_k m}}}{\alpha_{k+1}}V[u_k](u_{k+1})\Big) + 2\delta_k.
	 \end{split}
  \end{align}
By Lemma \ref{lemma:fast_str_conv} for the optimization problem in \eqref{equmir2DL_strong} with $\psi(x) = \alpha_{k+1}\psi_{\delta_k}(x, y_{k+1})$, $\beta = 1 + A_k\mu + A_k m$, $z = u_k$, $\gamma = \alpha_{k+1}\mu$, and $u = y_{k+1}$,
it holds that
\begin{align*}
&\alpha_{k+1}\psi_{\delta_k}(u_{k+1}, y_{k+1}) + (1 + A_k \mu + \att{A_k m})V[u_{k}](u_{k+1}) + \alpha_{k+1} \mu V[y_{k+1}](u_{k+1}) \\
&\hspace{2em}+ (1 + A_{k+1}\mu + \att{A_{k+1} m})V[u_{k+1}](x) - \widetilde{\delta}_k \\
&\leq \alpha_{k+1}\psi_{\delta_k}(x, y_{k+1}) + (1 + A_k \mu + \att{A_k m})V[u_{k}](x) + \alpha_{k+1} \mu V[y_{k+1}](x).
\end{align*}
From the fact that $V[y_{k+1}](u_{k+1}) \geq 0$, we have
\begin{align}
\begin{split}
   \label{lemma:fast_str_conv:ineq_2}
&\alpha_{k+1}\psi_{\delta_k}(u_{k+1}, y_{k+1}) + (1 + A_k\mu + \att{A_k m})V[u_{k}](u_{k+1})\\
&\leq \alpha_{k+1}\psi_{\delta_k}(x, y_{k+1}) + (1 + A_k\mu + \att{A_k m})V[u_{k}](x)\\
&\hspace{2em}- (1 + A_{k+1}\mu + \att{A_{k+1} m})V[u_{k+1}](x) + \alpha_{k+1} \mu V[y_{k+1}](x) + \widetilde{\delta}_k.
\end{split}
\end{align}
Combining \eqref{lemma:fast_str_conv:ineq_1} and \eqref{lemma:fast_str_conv:ineq_2}, we obtain
  \begin{align*}
  f(x_{k+1})
     &\leq \frac{A_k}{A_{k+1}} f(x_k)+\frac{\alpha_{k+1}}{A_{k+1}}\Big(f_{\delta_k}(y_{k+1}) + \psi_{\delta_k}(x,y_{k+1}) + {\mu V[y_{k+1}](x)}\\
	 &\hspace{2em}+ \frac{{1 + A_k \mu + \att{A_k m}}}{\alpha_{k+1}}V[u_k](x) - \frac{{1 + A_{k+1} \mu + \att{A_{k+1} m}}}{\alpha_{k+1}}V[u_{k+1}](x) + \frac{\widetilde{\delta}_k}{\alpha_{k+1}}\Big) + 2\delta_k.
\end{align*}
We finish the proof of Lemma \ref{lemma:fast_strong_main} applying the left inequality in \eqref{eq:Str_Conv_Model_Fast}
\begin{align*}
    f(x_{k+1})
	 &\leq \frac{A_k}{A_{k+1}} f(x_k) +
	 \frac{\alpha_{k+1}}{A_{k+1}} f(x) \\
	 &\hspace{2em}+ \frac{{1 + A_k \mu + \att{A_k m}}}{A_{k+1}}V[u_k](x) - \frac{{1 + A_{k+1} \mu + \att{A_{k+1} m}}}{A_{k+1}}V[u_{k+1}](x) + 2\delta_k + \frac{\widetilde{\delta}_k}{A_{k+1}}.
	\end{align*}
\end{proof}
\begin{proof}[Proof of Theorem \ref{Th:fast_str_conv_adap}]
We telescope the inequality in Lemma \ref{lemma:fast_strong_main} for $k$ from $0$ to $N-1$ and take $x = x_*$:
\begin{align}
\begin{split}
\label{Th:str_conv_adap:proof:ineq_1}
		A_{N} f(x_{N})\leq A_{N}f(x_*) + V[u_0](x_*) - (1 + A_N(\mu + m)) V[u_N](x_*) + 2\sum_{k=0}^{N-1}A_{k+1}\delta_k + \sum_{k=0}^{N-1}\widetilde{\delta}_k.
\end{split}
\end{align}
Since $V[u_{k+1}](x_*) \geq 0$ for all $k \geq 0$, we have
\begin{align*}
		&A_{N} f(x_{N}) - A_{N}f(x_*) \leq V[u_0](x_*)  + 2\sum_{k=0}^{N-1}A_{k+1}\delta_k + \sum_{k=0}^{N-1}\widetilde{\delta}_k.
\end{align*}
The last inequality proves \eqref{Th:fast_str_conv_adap:result_1}.
\pd{Inequality \eqref{Th:fast_str_conv_adap:result_2} is a straightforward from \eqref{Th:str_conv_adap:proof:ineq_1} since $f(x) \geq f(x_*)$ for all $x \in Q$.}
\end{proof}

\pd{The next lemma is proved in Appendix \ref{proof:a_n_sequence} and gives \rev1{a lower bound for} the growth rate of the sequence $A_N$, see \cite{nesterov2013gradient, devolder2013firstCORE, gasnikov2018universal}.}

\begin{lemma}
\label{lemma:a_n_sequence}
For all $N\geq 0$,
\begin{align*}
A_N \geq \max\left\{\frac{1}{4}\left(\sum_{k=0}^{N-1}\frac{1}{\sqrt{L_{k+1}}}\right)^2, \frac{1}{L_1}\prod_{k=1}^{N-1}\left(1 + \sqrt{\frac{\mu + m}{4L_{k+1}}}\right)^2\right\}..
\end{align*}
\end{lemma}
\begin{remark}
\label{remark:fast_strong_lip}
Let us assume that the function $f$ has $L$-Lipschitz continuous gradient. This means that for $L_k \geq L$, the inequality \eqref{exitLDL_strong} always holds, whence, $L_k \leq 2L$, assuming that $L_0 \leq L$. \rev1{Substituting this inequality in the bound of} Lemma  \ref{lemma:a_n_sequence}, we have $A_N \geq  \frac{N^2}{4L}$
and
\begin{align*}
A_N \geq  \frac{1}{2L} \left(1 + \frac{1}{4}\sqrt{\frac{\mu + m}{L}}\right)^{2(N-1)} \geq \frac{1}{2L}\exp\left(\frac{N-1}{4}\sqrt{\frac{\mu + m}{L}}\right).
\end{align*}
In the last inequality we used inequality $\log(1+2x) \geq x$ for all $x \in [0,\frac{1}{4}]$.
Combining Theorem \ref{Th:fast_str_conv_adap} and Lemma \ref{lemma:a_n_sequence} we have
\begin{align}
\begin{split}
\label{remark:optimal_conv_rate}
f(x_N) - f(x_*) &\leq \min\left\{\frac{8L}{N^2}, 2L\exp\left(-\frac{N-1}{4}\sqrt{\frac{\mu + m}{L}}\right)\right\}V[u_0](x_*)\\
&\hspace{2em}+ \frac{2\sum_{k=0}^{N-1}A_{k+1}\delta_k}{A_N} + \frac{\sum_{k=0}^{N-1}\widetilde{\delta}_k}{A_N}.
\end{split}
\end{align}
The first term in the right hand side of ineqaulity \eqref{remark:optimal_conv_rate}  is up to a constant factor optimal for $\mu$-strongly convex functions with $L$-Lipschitz continuous gradient.

Note that in \cite{devolder2013firstCORE} for a non-adaptive fast gradient method with $(\delta, L, \mu)$-oracle, \rev1{which is a particular case of our inexact model}, for the case when $\delta_k$ is a constant it is shown that
\begin{align}
\label{devolder_error}
\frac{\sum_{k=0}^{N-1}A_{k+1}\delta}{A_N} \leq \min\left\{\left(\frac{1}{3}N + 2.4\right), \left(1 + \sqrt{\frac{L}{\mu}}\right)\right\}\delta.
\end{align}
This means that for $\mu > 0$, the error $\delta$ does not accumulate.

\attt{Let us briefly discuss what happens if we can control the errors $\delta_k,\widetilde{\delta}_k$ and choose them such that the residual in the objective is smaller than $\varepsilon$. For this, we need to ensure $\frac{\sum_{k=0}^{N-1}\widetilde{\delta}_k}{A_N} = O(\varepsilon)$ for all $k \geq 0$. One way to guarantee it is to take $\widetilde{\delta}_k = O(\varepsilon \alpha_{k+1})$ for all $k \geq 0$. From Lemma  \ref{lemma:a_n_sequence}, $L_k \leq 2L$ for all $k \geq 0$, and \eqref{alpha_def_strong}, we have $\alpha_k \geq \frac{\sqrt{2} k}{4L}$. It means that the sequence $\{\widetilde{\delta_k}\}_{k\geq 0}$ can be an increasing sequence.}

\attt{Additionally, taking $\delta_k = O(\frac{\alpha_{k+1}}{A_{k+1}}\varepsilon)$, we can guarantee $\frac{\sum_{k=0}^{N-1}A_{k+1}\delta_k}{A_N} = O(\varepsilon)$.
Therefore, the sequence $\{\delta_k\}_{k\geq 0}$, in general, should decrease with the rate of the order $1/k$ in order to obtain an $\varepsilon$-solution for problem \eqref{eq:Problem}.
}

\end{remark}





\begin{remark}
\label{remark:FastGradConvRate}
In view of assumptions from Remark \ref{remark:fast_strong_lip}. For the case when $\mu = m = 0$ Algorithm \ref{FastAlg2_strong} can guarantee the following convergence rate:
\begin{align}
\label{eq:FastGradConvRate}
	f(x_N) - f_* &\leq \frac{8LV[x_0](x_*)}{N^2} + 2N\delta + \frac{4L\widetilde{\delta}}{N}.
\end{align}
A similar result was shown in \cite{tyurin2017fast}.
\end{remark}

\begin{remark}
Let us analyze the convergence rate \rev1{of $V[u_N](x_*)$ which is guaranteed by} \eqref{Th:fast_str_conv_adap:result_2} in Theorem \ref{Th:fast_str_conv_adap}. \rev1{Since we make a (1-SC) assumption on the function $d$, we have that $\frac{1}{2}\|u_N-x_*\|^2 \leq V[u_N](x_*)$ and the convergence rate for $V[u_N](x_*)$ gives also convergence rate for $\|u_N-x_*\|^2$.} There are two different scenarios:
\begin{enumerate}
\item $\mu = m = 0$.
    In this case we have:
        \begin{align*}
    V[u_N](x_*)\leq V[u_0](x_*) + 2\sum_{k=0}^{N-1}A_{k+1}\delta_k +\sum_{k=0}^{N-1}\widetilde{\delta}_k.
    \end{align*}
    We can only bound $V[u_N](x_*)$ by $V[u_0](x_*)$ up to additive noise.
\item
    $\mu + m > 0$.
    Using Lemma \ref{lemma:a_n_sequence} we can see that Theorem \ref{Th:fast_str_conv_adap} guarantees the linear convergence in argument up to additive noise.
\end{enumerate}
Note that convergence rates for the objective and the argument are obtained for different sequences $x_N$ and $u_N$, respectively.
\end{remark}

\subsection{Universal conditional gradient (Frank--Wolfe) method}
\label{univ_frank_wolfe}
Let us show an example of $(\delta, L, \mu, m, V,\|\cdot \|)$-model's application. \rev1{We assume that $\mu = 0, m = 0$ and use Algorithm \ref{FastAlg2_strong} to propose a universal Frank--Wolfe method.}
In order to construct the universal Frank--Wolfe method let us introduce the following constraints to the optimization problem \eqref{eq:Problem}:
\begin{enumerate}
    \item The set $Q$ is bounded w.r.t $V[y](x)$: $\exists R_Q \in \mathbb{R}:~   V[y](x) \leq R_Q^2\quad \forall x, y \in Q.$
    \item The function $f(x)$ has H\"older-continuous subgradients:
    \begin{equation*}
    \norm{\nabla f(x) - \nabla f(y)}_* \leq L_\nu\norm{x - y}^\nu\,\,\,\,\forall x,y \in Q.
    \end{equation*}
    From this we can get, for any $\delta > 0$, an inequality (see \cite{nesterov2015universal})
    \begin{gather}
    0 \leq f(x) - f(y) - \langle\nabla f(y), x - y \rangle \leq \frac{L(\delta)}{2}\norm{x - y}^2 + \delta \,\,\,\, \forall x,y \in Q,
    \end{gather}
    where \begin{gather*}L(\delta)=L_\nu\left[\frac{L_\nu}{2\delta}\frac{1-\nu}{1+\nu}\right]^\frac{1-\nu}{1+\nu}.\end{gather*}
\end{enumerate}
\rev1{Using the same arguments as in \cite{tyurin2017fast}, we can show that
the point $u_{k+1}=\arg \min_{x \in Q} \alpha_{k+1}\psi_{\delta_k}(x, y_{k+1})$ satisfies $u_{k+1} = {\argmin_{x \in Q}}^{\widetilde{\delta}_k}\phi_{k+1}(x)$ with $\widetilde{\delta}_k = 2R^2_Q$. It means that iteration $u_{k+1}$ defined in this way satisfies \eqref{equmir2DL_strong}. We also set $\delta_k = \varepsilon \frac{\alpha_{k+1}}{4A_{k+1}}$ as it was proposed in \cite{nesterov2015universal} for Universal Fast Gradient Method.
This choice of $\delta_k$ and the fact that the objective function has H\"older continuous subgradient, using the same arguments as in Theorem 3 of \cite{nesterov2015universal}, guarantees that in \eqref{Th:fast_str_conv_adap:result_1} $\frac{2\sum_{k=0}^{N-1}A_{k+1}\delta_k}{A_N} \leq \frac{\varepsilon}{2}$.
Combining it again with \eqref{Th:fast_str_conv_adap:result_1}, $V[y](x) \leq R_Q^2$ and $\widetilde{\delta}_k = 2R^2_Q$, we obtain from Theorem  \ref{Th:fast_str_conv_adap} that
\begin{equation*}
f(x_N) - f(x_*) \leq \frac{R_Q^2}{A_N} + \frac{\e}{2} + \frac{2R^2_QN}{A_N} \leq \frac{3R^2_QN}{A_N} + \frac{\e}{2}.
\end{equation*}}
Using the same arguments as in Theorem 3 of \cite{nesterov2015universal}, we obtain that
\begin{equation}\label{AN_universal2} A_N \geq \frac{N^\frac{1+3\nu}{1+\nu}\epsilon^\frac{1-\nu}{1+\nu}}{2^\frac{3+5\nu}{1+\nu}L_\nu^\frac{2}{1+\nu}}.\end{equation}
Using this inequality, we obtain the following upper bound for the number of steps in order to guarantee $f(x_N) - f(x_*) \leq \e$
\begin{gather*}
N \leq \inf_{\nu\in(0,1]}\left[2^\frac{3+4\nu}{\nu}\left(\frac{L_\nu R_Q^{1+\nu}}{\e}\right)^\frac{1}{\nu}\right],
\end{gather*}
\rev1{where the infimum can be taken since neither $\nu$ nor $L_{\nu}$ is not used in the algorithm. This bound justifies the word "universal" in the name of the algorithm since the algorithm is the same for any values of $\nu \in [0,1]$ and works with the best possible convergence guarantee. We call the proposed algorithm conditional gradient (Frank--Wolfe) method since in a simple situation $\psi_{\delta_k}(x, y_{k+1}) = \langle \nabla f(y_{k+1}), x - y_{k+1} \rangle$, finding $u_{k+1}=\arg \min_{x \in Q} \alpha_{k+1}\psi_{\delta_k}(x, y_{k+1})$ is reduced to solving a linear minimization problem over the set $Q$. This method for composite optimization with H\"older-continuous gradient was analyzed in \cite{nesterov2018complexity} and for a non-composite case was very recently analyzed in \cite{zhao2020analysis}.
Our method is more general, since it is universal and does not require the knowledge of the parameters $\nu, L_{\nu}$, and also can utilize general inexact models described by Definition \ref{defRelStronglyConvexFast}, including inexact gradients, composite optimization, inexact linear minimization oracle.
}

\section{Inexact Model for Variational Inequalities}
\label{VI}

In this section, we go beyond minimization problems and propose an abstract inexact model counterpart for variational inequalities. \pd{As a special case in Example \ref{Examp_Rel_Smooth_VI}, we introduce relative smoothness for operators in the spirit of \cite{lu2018relatively}, where it was introduced for optimization problems. Further, we propose a generalization of the Mirror-Prox algorithm for this general case of the inexact model of the operator and abstract variational inequalities. One of the main features of our algorithm is its adaptation to
generalized inexact parameter of smoothness.}
\pd{As a special case}, we propose a universal method for variational inequalities with complexity $
O\left(\left(\frac{1}{\varepsilon}\right)^{\frac{2}{1+\nu}}\right)$, where $\varepsilon$ is the desired accuracy of the solution and $\nu$ is the H\"older exponent of the operator.
According to the lower bounds in \cite{Optimal}, this algorithm is optimal for $\nu = 0$ (bounded variation of the operator) and $\nu = 1$ (Lipschitz continuity of the operator).
Based on the model for VI and functions, we introduce an inexact model for saddle-point problems (see Definition~\ref{DefSaddleModel}). We are also motivated by mixed variational inequalities \cite{Konnov_2017,Bao_Khanh} and composite saddle-point problems \cite{chambolle2011first-order}.

Formally speaking, we consider the problem of finding the solution $x_*\in Q$ for VI in the following abstract form
\begin{equation}\label{eq13}
\psi(x,x_*)\geqslant 0 \quad \forall x \in Q
\end{equation}
for some convex compact set $Q\subset\mathbb{R}^n$ and some function $\psi:Q\times Q\rightarrow\mathbb{R}$. Assuming the abstract monotonicity of the function $\psi$
\begin{equation}\label{eq:abstr_monot}
\psi(x,y)+\psi(y,x)\leq0\;\;\;\forall x,y\in Q,
\end{equation}
any solution to \eqref{eq13} is a solution of the following inequality
\begin{equation}\label{eq115}
\max_{x\in Q}\psi(x_*,x)\leq 0  
\end{equation}

In the general case, we make an assumption about the existence of a solution $x_*$ of the problem \eqref{eq13}. As a particular case, if for some operator $g: Q \rightarrow\mathbb{R}^n$ we set $\psi(x,y)=\langle g(y),x-y\rangle\;\;\forall x,y\in Q$,
then \eqref{eq13} and \eqref{eq115} are equivalent, respectively, to a standard strong and weak variational inequality with the operator $g$.

We start with a concept of $(\delta, L, V)$-model for problems \eqref{eq13} and \eqref{eq115}.

\begin{definition}\label{Def_Model_VI}
We say that a function $\psi$ has $(\delta, L\pd{,V})$-model $\psi_{\delta} (x, y)$ for some fixed values \pd{$\delta>0$ and $L=L(\delta)>0$}
if the following properties hold for each $x, y, z \in Q$:
\begin{enumerate}
\item[(i)] $\psi(x, y) \leq \psi_{\delta}(x, y) + \delta$;
\item[(ii)] $\psi_{\delta} (x, y)$ convex in the first variable; \item[(iii)] $\psi_{\delta}(x,x)=0$;
\item[(iv)] ({\it abstract $\delta$-monotonicity})
\begin{equation}\label{eq:abstr_monot_delta}
\psi_{\delta}(x,y)+\psi_{\delta}(y,x)\leq \delta;
\end{equation}
\item[(v)] ({\it generalized relative smoothness})
\begin{equation}\label{VIeq20}
\psi_{\delta}(x,y)\leq\psi_{\delta}(x,z)+\psi_{\delta}(z,y)+ LV[z](x)+ LV[y](z)+\delta.
\end{equation}
\end{enumerate}
\end{definition}

\begin{example}
For some operator $g:Q\rightarrow\mathbb{R}^n$ and a convex function $h:Q\rightarrow\mathbb{R}^n$, the choice
\begin{equation}\label{eq17}
\psi(x,y)=\langle g(y),x-y\rangle+h(x)-h(y)
\end{equation}
leads to a {\it mixed variational inequality} \cite{Konnov_2017,Bao_Khanh}
\begin{equation}\label{eq18}
\langle g(y),y-x\rangle+h(y)-h(x)\leq 0,
\end{equation}
which in the case of the monotonicity of the operator $g$ implies
\begin{equation}\label{eq19}
\langle g(x),y-x\rangle+h(y)-h(x)\leq 0.
\end{equation}
\end{example}

\begin{remark}
\label{BregmanVI}
Similarly to Definition~\ref{defRelStronglyConvex} above, in general case, we do not need the (1-SC) assumption for $V[y](x)$ in Definition \ref{Def_Model_VI}.
In some situations we make (1-SC) assumption for $V[y](x)$ (see Example \ref{UMP_Example} and Section~\ref{UMPStrongApp}).
\end{remark}

Note that for $\delta=0$ the following analogue of \eqref{VIeq20} for some fixed $a, b > 0$
\begin{equation}\label{eq200}
\psi(x,y)\leq\psi(x,z)+\psi(z,y)+ a\|z - y\|^2 + b\|x - z\|^2 \quad  \forall x, y, z \in Q
\end{equation}
was introduced in \cite{Mastroeni}. Condition \eqref{eq200} is used in many works on equilibrium programming. Our approach allows us to work with non-Euclidean set-up without (1-SC) assumption and inexactness $\delta$, which is important for the framework of universal methods \cite{nesterov2015universal} (see Example \ref{UMP_Example} below).

One can directly verify that if $\psi_{\delta}(x,y)$ is $(\delta/3, L, 0, 0, V)$-model of the function $f$ at a given point $y$ then $\psi_{\delta}(x,y)$ is $(\delta, L, V)$-model in the sense of Definition \ref{Def_Model_VI}.

Let us consider some examples.
\begin{example}\label{Examp_Rel_Smooth_VI}
\pd{{\bf Relative smoothness for optimization and VI.}
Let us consider a minimization problem \eqref{eq:Problem} with the function $f$ being smooth, convex and relatively $L$-smooth \rev1{relative} to $d$ \cite{lu2018relatively}, i.e., for all $x,y \in Q$,
\[
f(x) - f(y) - \la\nabla f(y), x-y\ra \leq LV[y](x).
\]
In this case, \eqref{eq:Problem} is equivalent to abstract VI \eqref{eq115} with $\psi_{\delta}(x, y):= \langle \nabla f(y), x - y \rangle$. Properties (i)-(iv) in Definition \ref{Def_Model_VI} obviously hold with $\delta = 0$. Let us check that (v) also holds. Indeed,
\begin{align}
    &\psi_{\delta}(x,y)-\psi_{\delta}(x,z)-\psi_{\delta}(z,y) = \la \nabla f(y),x-y\ra - \la \nabla f(z),x-z\ra - \la \nabla f(y),z-y\ra \notag \\
    &\hspace{2em}= (f(x)-f(z)-\la \nabla f(z),x-z \ra) + (f(z)-f(y) - \la \nabla f(y), z-y\ra) \notag \\
    &\hspace{3em} - (f(x) -f(y) - \la \nabla f(y),x-y\ra) \leq LV[z](x)+ LV[y](z), \notag
\end{align}
where we used relative $L$-smoothness and convexity of $f$. This example shows that our inexact model for VI 
as a particular case contains the concept of relative smoothness introduced in optimization. In this particular case, we say that an operator $g$ is relatively $L$-smooth if
\[
\la g(y)-g(z) ,x - z \ra \leq LV[z](x)+ LV[y](z) \;\; \forall x,y,z \in Q.
\]
}
\end{example}

\revv{
Let us make a precise definition
\begin{definition}
We say that an operator $g(x)$ is relatively smooth relative to function $d(x)$ if for the corresponding Bregman divergence $V[y](x)$ and a constant $L \geq 0$ it holds 
\begin{equation}
    \label{eq:oper_rel_smooth_def}
    \la g(y)-g(z) ,x - z \ra \leq LV[z](x)+ LV[y](z) \;\; \forall x,y,z \in Q.
\end{equation}
\end{definition}
This definition (and its inexact version when there is some error $\delta>0$ added to the r.h.s.) contains as particular case several definitions which were introduced in the literature after the initial draft of this paper appeared as a preprint \cite{stonyakin2019inexact}. The definition of relative Lipschitzness, introduced in  \cite{cohen2021relative} and used to provide a simple explanation for accelerated methods, coincides with the above definition  of relatively smooth operator. Two other examples are explained next.
}

{\color{black}
\begin{example}[Bregman continuity and Metric  regularity]\label{Bregman_cont_Example}
Let us analyze the connection of our notion of relative smoothness and the notion of Bregman continuity introduced in \cite{antonakopoulos2019adaptive} and Metric  regularity introduced in \cite{antonakopoulos2020adaptive}. Both papers consider local norm $\|\cdot\|_x$ which is continuous for $x \in Q$, its dual $\|\cdot\|_{x,*}$ defined in a standard way, Bregman function $d$ and its Bregman divergence $V[z](x) = d(x)-d(z)-\la \nabla d(z),x-z\ra$ satisfying for some $K>0$
\begin{equation}
    \label{eq:Bregm_loc_str_conv}
    V[z](x) \geq \frac{K}{2}\|x-z\|_z^2.
\end{equation}

The paper \cite{antonakopoulos2019adaptive} introduces a notion of $\beta$-Bregman continuous operators which satisfy
\begin{equation}
    \label{eq:Bregman_cont}
    \|g(y)-g(z)\|_{z,*}\leq \beta \sqrt{2V[y](z)}.
\end{equation}
Let us show that such operators are relatively smooth relative to $d$.
Indeed, for any $x,y,z \in Q$,
\begin{gather*}
    \la g(y)-g(z) ,x - z \ra \leq \|g(y)-g(z)\|_{z,*} \|x-z\|_z \stackrel{\eqref{eq:Bregman_cont}}{\leq} \beta \sqrt{2V[y](z)} \|x-z\|_z \\
    \stackrel{\eqref{eq:Bregm_loc_str_conv}}{\leq} \beta \sqrt{2V[y](z)} \cdot \sqrt{\frac{2}{K}V[z](x)} \leq \frac{\beta}{\sqrt{K}} (V[z](x)+ V[y](z)),
\end{gather*}
where in the last inequality we used that $ab\leq \frac{1}{2}(a^2+b^2)$ for any $a,b\geq 0$. Thus, $g$ satisfies \eqref{eq:oper_rel_smooth_def} with $L=\beta/\sqrt{K}$ and is relatively smooth relative to $d$.

The paper \cite{antonakopoulos2020adaptive} uses additional assumption that the local metric is regular, i.e., for $\beta > 0$, satisfies 
\begin{equation}
    \label{eq:regular_metric}
    \|v\|_{z,*}\leq \|v\|_{y,*}(1+\beta\|z-y\|_y).
\end{equation}
The following two notions are also introduced by saying that the operator $g$ is
\begin{itemize}
    \item Metrically bounded if there exists some $M_0 >0$ such that 
    \[
    \|g(z)\|_{z,*} \leq M_0, \quad \forall z \in Q.
    \]
    \item Metrically smooth if there exists some $M_1 >0$ such that
    \[
    \|g(y)-g(z)\|_{z,*} \leq M_1 \|y-z\|_{y}, \quad \forall y,z \in Q.
    \]
\end{itemize}
Let us show that in these two cases $g$ is inexactly relatively smooth relative to $d$.
Let $g$ be metrically bounded. Then
\begin{align}
    \label{eq:metr_bounded_bound_var} 
    &\|g(y)-g(z)\|_{z,*} \leq \|g(y)\|_{z,*}+\|g(z)\|_{z,*} \notag \\ &\stackrel{\eqref{eq:regular_metric}}{\leq} \|g(y)\|_{y,*}(1+\beta\|z-y\|_y) +\|g(z)\|_{z,*} 
    \leq 2M_0 + \beta M_0 \|z-y\|_y.
\end{align}
Further, for any $\delta >0$,
\begin{gather*}
    \la g(y)-g(z) ,x - z \ra \leq \|g(y)-g(z)\|_{z,*} \|x-z\|_z \stackrel{\eqref{eq:metr_bounded_bound_var}}{\leq} 
    (2M_0 + \beta M_0 \|z-y\|_y)\|x-z\|_z \\
    =2M_0\|z-y\|_y^0\|x-z\|_z +\beta M_0 \|z-y\|_y\|x-z\|_z \\
    \stackrel{*}{\leq}  \frac{1}{\delta} \frac{M_0^2}{2}(\|z-y\|_y^2+\|x-z\|_z^2) + \frac{\delta}{2} + \frac{\beta M_0}{2}(\|z-y\|_y^2+\|x-z\|_z^2) \\
    \stackrel{\eqref{eq:Bregm_loc_str_conv}}{\leq} \left(\frac{M_0^2}{2\delta} + \frac{\beta M_0}{2}  \right) \frac{2}{K} (V[z](x)+ V[y](z)) + \frac{\delta}{2},
\end{gather*}
where $*$ uses that, for any $a,b,c \geq 0$, any $\delta >0$, any $\nu \in [0,1]$ (and in particular for $\nu=0$), it holds that $ab^{\nu}c \leq \left(\frac{1}{\delta}\right)^{\frac{1-\nu}{1+\nu}} \frac{a^{\frac{2}{1+\nu}}}{2} \left(b^2+c^2\right) + \frac{\delta}{2}$, see \cite{stonyakin2018generalized}[Lemma 1]. Thus, we obtain that, for any $\delta >0$,
\[
\la g(y)-g(z) ,x - z \ra \leq  L(\delta)(V[z](x)+ V[y](z)) + \frac{\delta}{2}
\]
with 
\begin{equation}
\label{eq:metr_bounded_bound_var_L}
L(\delta) = \left(\frac{M_0}{\delta} + \beta  \right) \frac{M_0}{K},
\end{equation}
i.e. \eqref{eq:oper_rel_smooth_def} holds with any error $\delta>0$ and $g$ is inexactly relatively smooth relative to $d$. Note that this is also covered by our concept of inexact model.

Let $g$ be metrically smooth. Then, from \eqref{eq:Bregm_loc_str_conv} it is straightforward that 
\[
\|g(y)-g(z)\|_{z,*} \leq M_1 \|y-z\|_{y} \leq M_1 \sqrt{\frac{2}{K}V[y](z)}
\]
and, thus, $g$ is $M_1/\sqrt{K}$-Bregman continuous, whence, as it was shown above, $g$ is also relatively smooth relative to $d$ with $L=M_1/K$. 

Thus, we conclude that our algorithms for variational inequalities are also applicable for solving variational inequalities with Bregman continuous, metrically bounded and metrically smooth operators. After presenting our method and its convergence rate we discuss their convergence rate for these types of operators. Moreover, the methods which we develop  in this paper  for variational inequalities and saddle-point problems are applicable to such problems in the presence of inexact information, which is a more general setting than in \cite{antonakopoulos2019adaptive,antonakopoulos2020adaptive}.
\end{example}

\begin{example}[Resource sharing problem]\label{L_condition_Example}
An interesting example of a problem in which a variational inequality with a relatively smooth and monotone operator naturally arises is a resource sharing problem  considered in \cite{antonakopoulos2019adaptive} (Example 2.3). 
In this example, $g(x) = \left(l_1(x_1), \ldots, l_n(x_n)\right)$ with a loss function $l_i(x) = \frac{1}{\alpha_i -  x_i}$, $i=1,...,n$
where  $x$ belongs to the set of feasible resource allocations $Q = \{x=(x_1, \ldots, x_n): 0 \leq x_i < \alpha_i, \; x_1 + \ldots +x_n = R\}$ for $R> 0$, and $\alpha_i$ denotes the capacity of the resource $i$. This operator has singularities corresponding to $\alpha_i$. Let $d(x) = \sum_{i =1}^{n}\frac{1}{1-x_i}$ be the prox-function defined on the feasible set $Q$. Then the associated Bregman divergence is
$$
V{[y]}(x) = \sum_{i= 1}^{n} \frac{(x_i - y_i)^2}{(1-x_i)(1-y_i)^2}.
$$
In \cite{antonakopoulos2019adaptive} the authors show that the operator $g$ is Bregman continuous.
Since in the previous example we showed that such operators are relatively smooth, we conclude that our algorithms are also suitable for the resource sharing problem. Experiments on this problem can be found in Appendix \ref{app_RSP}.

\end{example}
}

\begin{example}\label{UMP_Example} {\bf Variational Inequalities with monotone H\"older continuous operator.}
\label{example_universal_g}
\pd{Assume that $V$ satisfies (1-SC) condition w.r.t. some norm $\|\cdot\|$} and for a monotone operator $g$ there exists $\nu\in[0,1]$ such that
\begin{equation}\label{Hold_cont_g}
\norm{g(x) - g(y)}_* \leq L_{\nu}\norm{x -  y}^\nu \,\,\,\forall x,y \in Q.
\end{equation}
Then, we have $
\langle g(z)-g(y), z-x\rangle\leq
\|g(z) - g(y) \|_* \|z-x\| \leq L_{\nu}\|z-y\|^{\nu}  \|z-x\|
$
\begin{equation}\label{Hold_interpol}
    \leq \frac{L(\delta)}{2}||z-x||^2+\frac{L(\delta)}{2}||z-y||^2+\delta  \leq L(\delta)V[z](x)+ L(\delta)V[y](z) +\delta
\end{equation}
with
\begin{equation}\label{UMP_constant} L(\delta)  = \left(\frac{1}{2\delta}\right)^\frac{1-\nu}{1+\nu} L_{\nu}^{\frac{2}{1+\nu}}
\end{equation}
with arbitrary $\delta > 0$.  In this case $\psi_{\delta}(x, y):= \langle g(y), x - y \rangle$
is a $(\delta, L,V)$-model.
\end{example}

Note that for the previous two examples in Algorithm \ref{Alg:UMPModel} and Theorem~\ref{thmm1inexact} we need $V[z](x)$ to satisfy (1-SC) condition.

Next, we introduce our adaptive method (Algorithm \ref{Alg:UMPModel}) for abstract variational inequalities with inexact $(\delta, L, V)$-model.
This method adapts to the local values of $L$ and allows us to construct a universal method for variational inequalities by applying it to VI with H\"older \rev1{condition \eqref{Hold_cont_g}}
for $\delta = \frac{\varepsilon}{2}$ and $L = L\left(\frac{\varepsilon}{2}\right)$.
\begin{algorithm}[ht]
\caption{Generalized Mirror Prox for VI}
\label{Alg:UMPModel}
\begin{algorithmic}[1]
   \REQUIRE accuracy $\e > 0$, oracle error $\delta >0$,
   initial guess $L_{0} >0$,
   prox set-up: $d(x)$, $V[z] (x)$.
   \STATE Set $k=0$, $z_0 = \arg \min_{u \in Q} d(u)$.
   \REPEAT
%
				\STATE Find the smallest integer $i_k \geq 0$ such that
				\begin{equation}\label{eqUMP23}
                \begin{split}
                \hspace{-3em}\psi_{\delta}(z_{k+1}, z_{k})\leq \psi_{\delta}(z_{k+1}, w_{k})+\psi_{\delta}(w_k,z_k)+ L_{k+1}(V[z_k](w_k) + V[w_k](z_{k+1})) + \delta,
                \end{split}
                \end{equation}
			where 	$L_{k+1}=2^{i_k-1}L_{k}$ and
			\begin{align}
			    w_k&={\argmin_{x \in Q}}^{\widetilde{\delta}} \left\{\psi_{\delta}(x, z_k)+ L_{k+1}V[z_k](x) \right\}.
		\label{eq:UMPwStepMod} \\
				z_{k+1}&={\argmin_{x \in Q}}^{\widetilde{\delta}} \left\{\psi_{\delta}(x, w_k) + L_{k+1}V[z_k](x) 		 \right\}. \label{eq:UMPzStepMod}
			\end{align}
	\UNTIL
	\begin{equation}\label{eq_Alg3}
	S_N:= \sum_{k=0}^{N-1}\frac{1}{L_{k+1}}\geqslant \frac{\max\limits_{x \in Q}V[x^0](x)}{\varepsilon}.
	\end{equation}
	\ENSURE $\widehat{w}_N = \frac{1}{\sum_{k=0}^{N-1}\frac{1}{L_{k+1}}}\sum_{k=0}^{N-1}\frac{1}{L_{k+1}}w_k$.
\end{algorithmic}
\end{algorithm}

\pd{Next, we state the convergence rate result for the proposed method}.
\begin{theorem}\label{thmm1inexact}
For Algorithm \ref{Alg:UMPModel} the following inequality holds
\begin{equation}
- \frac{1}{S_N}\sum_{k=0}^{N-1}\frac{\psi_{\delta}(x,w_{k})}{L_{k+1}} \leq \frac{V[z_0](x)}{S_N} + \delta + 2\widetilde{\delta} \quad \forall x \in Q.
\end{equation}
Moreover,
\begin{equation*}
\max\limits_{u \in Q}\psi(\widehat{w}_N,u)
\leq \frac{2L \max_{u\in Q}V[z_0](u)}{N} + 3\delta + 2\widetilde{\delta},
\end{equation*}
and Algorithm \ref{Alg:UMPModel} stops in no more than
\begin{equation}\label{est_th4.5}
\left\lceil\frac{2L \max_{u\in Q}V[z_0](u)}{\varepsilon}\right\rceil
\end{equation}
iterations.
\end{theorem}
\revv{Before we prove the theorem, we note that similarly to \cite{antonakopoulos2020adaptive}, we can consider any nonempty bounded convex subset $C$ of the set $Q$ and have similar guarantees for $\max\limits_{u \in C}\psi(\widehat{w}_N,u)$ that plays a role of generalized restricted gap (or merit) function. Thus, there is no need for the set $Q$ to be bounded.}
\begin{proof}
After $(k+1)$-th iteration ($k=0,1,2\ldots$) from \eqref{eq:UMPwStepMod} and \eqref{eq:UMPzStepMod} we have, for each $u \in Q$,
$$ \psi_{\delta}(w_k, z_k)\leq\psi_{\delta}(u, z_k) +L_{k+1}V[z_k](u)-L_{k+1}V[w_k](u)- L_{k+1}V[z_k](w_k) + \widetilde{\delta}$$
and
$$\psi_{\delta}(z_{k+1}, w_k) \leq \psi_{\delta}(u, w_k)+L_{k+1}V[z_k](u)-L_{k+1}V[z_{k+1}](u)-L_{k+1}V[z_k](z_{k+1}) + \widetilde{\delta}.$$
The first inequality means that
$$
\psi_{\delta}(w_k, z_k)\leq\psi_{\delta}(z_{k+1}, z_k) +L_{k+1}V[z_k](z_{k+1})-L_{k+1}V[w_k](z_{k+1})- L_{k+1}V[z_k](w_k) + \widetilde{\delta}.
$$
Taking into account \eqref{eqUMP23}, we obtain, for all $u \in Q$,
$$
-\psi_{\delta}(u, w_k) \leq L_{k+1}V[z_k](u)-L_{k+1}V[z_{k+1}](u) + \delta + 2\widetilde{\delta}.
$$
So, the following inequality holds:
$$-\sum_{k=0}^{N-1} \frac{\psi_{\delta}(u,w_{k})}{L_{k+1}} \leq V[z_0](u) - V[z_N](u) + S_N (\delta + 2\widetilde{\delta}).$$
By virtue of \eqref{VIeq20} and the choice of $L_{0}\leq 2L$, it is guaranteed that $L_{k+1}\leq 2L\;\;\forall k=\overline{0,N-1}$
and we have from Definition \ref{Def_Model_VI}
\begin{align}\label{eq30}
& \max\limits_{u \in Q} \psi(\widehat{w}_N, u) \leq  \max\limits_{u \in Q} \psi_{\delta}(\widehat{w}_N, u)+\delta \\
&\leq
-\frac{1}{S_N}\sum_{k=0}^{N-1}\frac{\psi_{\delta}(u,w_{k})}{L_{k+1}} + 2\delta \leq \frac{2L \max_{u\in Q}V[z_0](u)}{N} + 3\delta + 2\widetilde{\delta}. \notag
\end{align}
\end{proof}

\begin{remark}[Bregman continuity and Metric  regularity]
\revv{Let us consider the setting of Example \ref{Bregman_cont_Example}. In all the cases we define $\psi_{\delta}(w,u)=\la g(u),w-u\ra$. First, if the operator $g$ is $\beta$-Bregman continuous, then, as we show in Example \ref{Bregman_cont_Example}, it is relatively smooth relative to $d$ with constant $L=\beta/\sqrt{K}$. Applying Theorem \ref{thmm1inexact} with $\delta=\widetilde{\delta}=0$, we obtain that
\[
\max\limits_{u \in Q} \la g(u),\widehat{w}_N-u\ra \leq \frac{2\beta \max_{u\in Q}V[z_0](u)}{N\sqrt{K}},
\]
which has the same dependence on $N$ as the rate in \cite{antonakopoulos2019adaptive}. 
}

\revv{
Let us now consider $M_0$-metrically bounded operator $g$. As it was shown in Example \ref{Bregman_cont_Example}, this operator is inexactly relatively smooth, which means that $\psi_{\delta}(w,u)=\la g(u),w-u\ra$ is an inexact model with $L(\delta)=\left(\frac{M_0}{\delta} + \beta  \right) \frac{M_0}{K}$ defined in \eqref{eq:metr_bounded_bound_var_L}. In this setting, we have $\widetilde{\delta}=0$ and set $\delta=\varepsilon/6$. Applying Theorem \ref{thmm1inexact}, we obtain 
\[
\max\limits_{u \in Q} \la g(u),\widehat{w}_N-u\ra \leq \frac{2\max_{u\in Q}V[z_0](u)}{N}\left(\frac{6M_0}{\varepsilon} + \beta  \right) \frac{M_0}{K} + \frac{\varepsilon}{2}.
\]
Thus, to obtain an $\varepsilon$-solution, i.e. guarantee $\max\limits_{u \in Q} \la g(u),\widehat{w}_N-u\ra \leq \varepsilon$, it is sufficient to take $N\geq \Omega\left( \frac{M_0^2\max_{u\in Q}V[z_0](u)}{K\varepsilon^2}\right)$, which has the same dependence on $\varepsilon$ as the bound in \cite{antonakopoulos2020adaptive}.
}

\revv{
Finally, we consider $M_1$-metrically smooth operator $g$. According to Example \ref{Bregman_cont_Example}, this operator is also relatively smooth with $L=M_1/K$.
Applying Theorem \ref{thmm1inexact}, we obtain 
\[
\max\limits_{u \in Q} \la g(u),\widehat{w}_N-u\ra \leq \frac{2M_1 \max_{u\in Q}V[z_0](u)}{NK}.
\]
A very close convergence rate is obtained in \cite{antonakopoulos2020adaptive}, yet without explicit dependence on $M_1$ and $K$.
}

\revv{
Note that in all three examples the same rate holds for $\max\limits_{u \in C} \la g(u),\widehat{w}_N-u\ra$ where $C$ is any non-empty bounded convex subset of $Q$. Thus, there is no need for the set $Q$ to be bounded.
}

\end{remark}

\begin{remark}
\rev1{In the setting of Example \ref{UMP_Example}, we set, for any desired accuracy $\varepsilon >0$,  $\delta = \frac{\varepsilon}{2}$ and $L = L\left(\frac{\varepsilon}{2}\right)$ according to \eqref{UMP_constant}. Then from \eqref{est_th4.5}, we obtain that Algorithm \ref{Alg:UMPModel} has the following complexity to guarantee} 
\begin{equation}\label{est_rem8}
\left\lceil 2 \inf_{\nu\in[0,1]}\left(\frac{2L_{\nu}}{\e} \right)^{\frac{2}{1+\nu}} \cdot \max_{u\in Q}V[z_0](u)\right\rceil.
\end{equation}
Note that estimate \eqref{est_rem8} is optimal for variational inequalities and saddle-point problems in the cases $\nu = 0$ and $\nu = 1$.
\end{remark}

Thus, the introduced concept of the $(\delta, L, V)$-model for variational inequalities allows us to extend the previously proposed universal method for VI \dd{to} a wider class of problems, including {\it mixed variational inequalities} \cite{Konnov_2017,Bao_Khanh} and {\it composite saddle-point problems} \cite{chambolle2011first-order}.

\begin{remark}\label{remark_comparision}
The authors of \cite{antonakopoulos2019adaptive} propose an adaptive Mirror-Prox algorithm for variational inequalities with Bregman continuous operator, i.e. satisfying inequality similar to \eqref{eq:Bregman_cont}. As we show above in Example \ref{L_condition_Example}, this setting is covered by our concept of relative smoothness for operators and, thus, also by our notion of inexact model for variational inequalities. At the same time their algorithm has an important difference with our algorithm in terms of the choice of the stepsize. In our Algorithm \ref{Alg:UMPModel}, the role of the stepsize  is played by $1/L_{k+1}$, and it may happen, since $L_{k+1}=2^{i_k-1}L_k$ with $i_k\geq 0$, that the stepsize is increasing between iterations leading to faster convergence in practice. On the contrary, the stepsize in \cite{antonakopoulos2019adaptive} is non-increasing. We illustrate the importance of the possibility to increase the stepsize in computational experiments in Appendix \ref{app_numerical_1} for non-smooth saddle-point problems by comparing our method with $i_k\geq 0$ and with $i_k\geq 1$. 
The same remark holds also for a very recent preprint \cite{antonakopoulos2020adaptive}, where an adaptive Mirror-Prox algorithm  for metrically regular problems was proposed with optimal bounds for metrically bounded and metrically smooth operators. This method also uses decreasing stepsizes. Finally, both papers \cite{antonakopoulos2019adaptive,antonakopoulos2020adaptive} do not consider any inexactness for their adaptive methods. Only for the non-adaptive method in \cite{antonakopoulos2019adaptive} they prove convergence rate in the setting of stochastic errors in the operator values, the setting we do not consider here.
\end{remark}

Now we introduce an inexact model for saddle-point problems. The solution of variational inequalities reduces the so-called saddle-point problems, in which for a convex in $u$ and concave in $v$ functional $f(u,v):\mathbb{R}^{n_1+n_2}\rightarrow\mathbb{R}$ ($u\in Q_1\subset\mathbb{R}^{n_1}$ and $v\in Q_2\subset\mathbb{R}^{n_2}$) needs to be found the point $(u_*, v_*)$ such that:
\begin{equation}\label{eq31}
f(u_*,v)\leq f(u_*,v_*)\leq f(u,v_*)
\end{equation}
for arbitrary $u\in Q_1$ and $v\in Q_2$. Let $Q=Q_1\times Q_2\subset\mathbb{R}^{n_1+n_2}$. For $x=(u,v)\in Q$, we assume that $||x||=\sqrt{||u||_1^2+||v||_2^2}$ ($||\cdot||_1$ and $||\cdot||_2$ are the norms in the spaces $\mathbb{R}^{n_1}$ and $\mathbb{R}^{n_2}$, respectively). We agree to denote $x=(u_x,v_x),\;y=(u_y,v_y)\in Q$.

It is well known that for a sufficiently smooth function $f$ with respect to $u$ and $v$ the problem \eqref{eq31} reduces to VI with an operator $
g(x)= (f_u'(u_x,v_x), \; - f_v'(u_x,v_x))$.

For saddle-point problems we propose some adaptation of the concept of the $(\delta, L, V)$-model for abstract variational inequality.

\begin{definition}\label{DefSaddleModel}
We say that the function $\psi_{\delta}(x,y)$ $(\psi_{\delta}:\mathbb{R}^{n_1+n_2}\times\mathbb{R}^{n_1 + n_2}\rightarrow\mathbb{R})$ is a $(\delta,L,V)$-model for the saddle-point problem \eqref{eq31} if the conditions (ii) -- (v) of Definition \ref{Def_Model_VI} hold and in addition
\begin{equation}\label{eq33}
f(u_y,v_x)-f(u_x,v_y)\leq-\psi_{\delta}(x,y) + \delta \quad \forall x, y \in Q.
\end{equation}
\end{definition}

\begin{example}
The proposed concept of the $(\delta,L,V)$-model for saddle-point problems is quite applicable, for example, for composite saddle-point problems of the form considered in the popular paper \cite{chambolle2011first-order}:
\begin{equation}\label{eq34}
f(u,v)=\tilde{f}(u,v)+h(u)-\varphi(v)
\end{equation}
for some convex in $u$ and concave in $v$ subdifferentiable functions $\tilde{f}$, as well as convex functions $h$ and $\varphi$. In this case, we can put
\begin{equation}\label{eq35}
\psi_{\delta}(x,y)=\langle\tilde{g}(y),x-y\rangle+h(u_x)+\varphi(v_x)-h(u_y)-\varphi(v_y),
\end{equation}
where\rev1{
$$
\tilde{g}(y)=
\begin{pmatrix}
\tilde{g}_u(u_y,v_y)\\
-\tilde{g}_v(u_y,v_y)
\end{pmatrix}.
$$
for some $\tilde{g}_u \in \partial_u f(u_y,v_y)$ and $\tilde{g}_v \in - \partial_v f(u_y,v_y)$.}
\end{example}

Theorem \ref{thmm1inexact} implies
\begin{theorem}
If for the saddle-point problem \eqref{eq31} there is a $(\delta,L,V)$-model $\psi_{\delta}(x,y)$, then after stopping the algorithm we get a point
\begin{equation}\label{eq36}
\widehat{y}_N=(u_{\widehat{y}_N},v_{\widehat{y}_N}):=(\widehat{u}_N, \widehat{v}_N):=\frac{1}{S_N}\sum_{k=0}^{N-1}\frac{y_{k}}{L_{k+1}},
\end{equation}
for which the following inequality is true:
\begin{equation}\label{eq37}
\max_{v\in Q_2}f(\widehat{u}_N, v)-\min_{u\in Q_1}f(u, \widehat{v}_N)\leq \frac{2L \max_{(u, v) \in Q} V[u_0, v_0](u, v)}{N} +2\tilde{\delta}+2\delta.
\end{equation}
\end{theorem}

\section{Inexact Model for Strongly Monotone VI}
\label{UMPStrongApp}

In this section similarly with the concept of ($\delta, L, \mu, m, V$)-model in optimization we  consider an inexact model for VI with a stronger version of the monotonicity condition \eqref{eq:abstr_monot_delta}.
\begin{definition}\label{Def_Model_VI_2}
We say that functional $\psi$ has $(\delta, L, \mu, V)$-model $\psi_{\delta} (x, y)$ at a given point $y$ if the following properties hold for each $x, y, z \in Q$:
\begin{enumerate}
\item[(i)] $\psi(x, y) \leq   \psi_{\delta}(x, y) + \delta$;
\item[(ii)] $\psi_{\delta} (x, y)$ convex in the first variable;
\item[(iii)]
$\psi_{\delta} (x, y)$ continuous in $x$ and $y$;
\item[(iii)] $\psi_{\delta}(x,x)=0$;
\item[(iv)] ({\it $\mu$-strong $\delta$-monotonicity})
\begin{equation}\label{eq:abstr_trong_monot}
\psi_{\delta}(x,y)+\psi_{\delta}(y,x)+\mu \|x-y\|^2\leq \delta;
\end{equation}
\item[(v)] ({\it generalized relative smoothness})
\begin{equation}\label{Strongeq20}
\psi_{\delta}(x,y)\leq\psi_{\delta}(x,z)+\psi_{\delta}(z,y)+ LV[z](x)+ LV[y](z)+\delta
\end{equation}
for some fixed values $L>0$, $\delta>0$.
\end{enumerate}
\end{definition}

\begin{remark}
We note that we can not replace $\|x-y\|^2$
by $V[y](x)$ in \eqref{eq:abstr_trong_monot} since it is essentially used in the proof of Theorem \ref{Th:RUMPCompl}.
\end{remark}

Now we propose a method with a linear rate of convergence for VI with $(\delta, L, \mu, V)$-model. We slightly modify the assumptions on prox-function $d(x)$. Namely, we assume that $\argmin_{x \in Q} d(x) = 0$ and that $d$ is bounded on the unit ball in the chosen norm $\|\cdot\|$, that is
\begin{equation}
d(x) \leq \frac{\Omega}{2}, \quad \forall x\in Q : \|x \| \leq 1,
\label{eq:dUpBound}
\end{equation}
where $\Omega$ is some known constant. Note that for standard proximal setups $\Omega = O(\ln \text{dim}E)$. Finally, we assume that we are given a starting point $x_0 \in Q$ and a number $R_0 >0$ such that $\| x_0 - x_* \|^2 \leq R_0^2$, where $x_*$ is the solution to abstract VI. The procedure of restating of Algorithm \ref{Alg:UMPModel} is applicable for abstract strongly monotone variational inequalities.

\begin{algorithm}
\caption{Restarted Generalized Mirror Prox}
\label{Alg:RUMP}
\begin{algorithmic}[1]
   \REQUIRE accuracy $\e > 0$, $\mu >0$, $\Omega$ s.t. $d(x) \leq \frac{\Omega}{2} \ \forall x\in Q: \|x\| \leq 1$; $x_0, R_0 \ s.t. \|x_0-x_*\|^2 \leq R_0^2.$
      \STATE Set $p=0,d_0(x)=R_0^2d\left(\frac{x-x_0}{R_0}\right)$.
   \REPEAT
			\STATE Set $x_{p+1}$ as the output of Algorithm \ref{Alg:UMPModel} after $N_p$ iterations of Algorithm \ref{Alg:UMPModel} with prox-function $d_{p}(\cdot)$ and stopping criterion $\sum_{k=0}^{N_p-1}\frac{1}{L_{k+1}} \geq \frac{\Omega}{\mu}$.
			\STATE Set $R_{p+1}^2 = R_0^2 \cdot 2^{-(p+1)} + 2(1- 2^{-(p+1)})\frac{\delta + 2\tilde{\delta}}{\mu}$.
			\STATE Set $d_{p+1}(x) \leftarrow R_{p+1}^2 d\left(\frac{x-x_{p+1}}{R_{p+1}}\right)$.
			\STATE Set $p=p+1$.
			\UNTIL $p > \log_2\frac{R_0^2}{\e}$	
		\ENSURE $x_{p+1}$.
\end{algorithmic}
\end{algorithm}

\begin{theorem}
\label{Th:RUMPCompl}
    Assume that $\psi_{\delta}$ is a $(\delta, L, \mu, V)$-model for $\psi$. Also assume that the prox function $d(x)$ satisfies \eqref{eq:dUpBound} and the starting point $x_0 \in Q$ and a number $R_0 >0$ are such that $\| x_0 - x_* \|^2 \leq R_0^2$, where $x_*$ is the solution to \eqref{eq115}. Then, for each $p\geq 0$
		\[
		\|x_p - x_*\|^2 \leq R_0^2\cdot 2^{-p}  +\frac{2\delta}{\mu}+\frac{4\widetilde{\delta}}{\mu} \leq \varepsilon + \frac{2\delta}{\mu}+\frac{4\widetilde{\delta}}{\mu}.
		\]
		 The total number of iterations of the inner Algorithm \ref{Alg:UMPModel} does not exceed
    \begin{equation}\label{eq_abst_strong_monot}
         \left\lceil \frac{2L\Omega}{\mu}\cdot \log_2 \frac{R_0^2}{\e}\right\rceil,
    \end{equation}
where $\Omega$ satisfies \eqref{eq:dUpBound}.
\end{theorem}
\begin{proof}
We show by induction that for $p \geq 0$
\begin{equation}\label{induction_ineq}
    \|x_p - x_*\|^2 \leq R_0^2\cdot 2^{-p} + 2(1 - 2^{-p})\left(\frac{\delta}{\mu}+\frac{2\widetilde{\delta}}{\mu}\right),
\end{equation}
which leads to the statement of the Theorem.
For $p=0$ this inequality holds by the theorem's assumption. Assuming that \eqref{induction_ineq} holds for some $p\geq 0$, our goal is to prove \eqref{induction_ineq} for $p+1$ considering the outer iteration $p+1$.
Observe that the function $d_{p}(x)$ defined in Algorithm \ref{Alg:RUMP} is 1-strongly convex w.r.t. the norm $\|\cdot\| / R_{p}$.

Using the definition of $d_{p}(\cdot)$ and \eqref{eq:dUpBound}, we have, since $x_p = \argmin_{x \in Q} d_p(x)$
\[
	V_{p}[x_{p}](x_*) = d_{p}(x_{*}) - d_{p}(x_{p}) - \la \nabla d_{p}(x_{p}), x_{*} - x_{p} \ra \leq  d_{p}(x_{*}) \leq \frac{R_p^2\Omega}{2}.
\]
Denote by $$S_{N_p}:= \sum_{k=0}^{N_p-1}\frac{1}{L_{k+1}}.$$

Thus, by Theorem \ref{thmm1inexact}, taking $u = x_*$, we obtain
\[
- \frac{1}{S_{N_p}} \sum_{k=0}^{N_p-1} \frac{\psi_{\delta}(x_*, w_k)}{L_{k+1}} \leq  \frac{R_p^2 V_{p}[x_{p}](x_{*})}{S_{N_p}} + \delta + 2\widetilde{\delta} \leq \frac{\Omega R_p^2}{2S_{N_{p}}} + \delta + 2\widetilde{\delta}.
\]
Since the operator $\psi$ is continuous and abstract monotone, we can assume that the solution to weak VI \eqref{eq13} is also a strong solution and
$- \psi(w_k, x_*) \leq 0$, $k=0,...,N_p-1$
and, by Definition \ref{Def_Model_VI_2} (i),
$-\psi_{\delta}(\omega_k,x_*)\leq\delta$ ($k=0,\ldots,N_p-1$).
This and \eqref{eq:abstr_trong_monot} give, that for each $k=0,...,N_p-1$,
\begin{align}
- \psi_{\delta} (x_*, w_k) &\geq - \delta - \psi_{\delta}(x_*, w_k) - \psi_{\delta}(w_k, x_*) \geq - \delta + \mu\|w_k-x_*\|^2, \notag \\
-\psi_{\delta}(x_*,\omega_k)&\geq-\delta -\psi_{\delta}(x_*,\omega_k)-\psi_{\delta}(\omega_k,x_*)\geq - \delta +\mu\|\omega_k-x_*\|^2. \notag
\end{align}

Thus, by the convexity of the squared norm, we obtain
\begin{align}
- 2\delta + \mu \|x_{p+1}-x_*\|^2 & = - 2 \delta + \mu \left\|\frac{1}{S_{N_p}} \sum_{k=0}^{N_p-1}  \frac{w_k}{L_{k+1}}-x_*\right\|^2 \leq - 2\delta + \frac{\mu}{S_{N_p}} \sum_{k=0}^{N_p-1} \frac{\|w_k-x_*\|^2}{L_{k+1}} \notag \\
&  \leq - 2 \delta - \frac{1}{S_{N_p}} \sum_{k=0}^{N_p-1} \frac{\psi_{\delta}(x_*, w_k)}{L_{k+1}} \leq \frac{\Omega R_p^2}{2 S_{N_p}} - \delta + 2\widetilde{\delta}. \notag
\end{align}
Using the stopping criterion $S_{N_p} \geq \frac{\Omega}{\mu}$  {\color{black} we can put $N_p = \left\lceil \frac{2L\Omega}{\mu}\right\rceil $} and
\begin{align}
    \|x_{p+1}-x_*\|^2 &\leq \frac{R_p^2}{2}  +\frac{\delta+2\widetilde{\delta}}{\mu}= \frac{1}{2}\left(R_0^2 \cdot 2^{-p} + 2(1 - 2^{-p})\frac{\delta+2\widetilde{\delta}}{\mu} \right)  +\frac{\delta+2\widetilde{\delta}}{\mu} \notag \\
    & = R_0^2 \cdot 2^{-(p+1)} + 2(1 - 2^{-(p+1)}) \left(\frac{\delta}{\mu}+\frac{2\widetilde{\delta}}{\mu}\right), \notag
\end{align}
which finishes the proof by induction.
\end{proof}

\begin{remark}
If for some $m>0$ $\psi_{\delta}(x,y)$ is an $m$-strongly convex function in $x$, then for Algorithm \ref{Alg:RUMP} we can prove the following estimate
		\[
		\|x_p - x_*\|^2 \leq R_0^2\cdot 2^{-p}  +\frac{2\delta}{m+\mu}+\frac{4\widetilde{\delta}}{m+\mu} \leq \varepsilon + \frac{2\delta}{m+\mu}+\frac{4\widetilde{\delta}}{m+\mu}
		\]
for each $p\geq 0$
and instead of \eqref{eq_abst_strong_monot} we obtain
\begin{equation}\label{eq_abst_strong_monot_1}
\left\lceil \frac{2L\Omega}{m+\mu}\cdot \log_2 \frac{R_0^2}{\e}\right\rceil.
\end{equation}
\end{remark}

\section{Conclusion}
\label{S:Conclusion}

\pd{In this paper,} we consider convex optimization problem \eqref{eq:Problem}.
It is well known (see \cite{devolder2014first, dvurechensky2017universal,gorbunov2019optimal}) that if \pd{there is an inexact gradient $\nabla_{\delta} f(y)$ of $f$, s.t.}, for all $x,y \in Q$,
\begin{equation}\label{eq:inexact_oracle_DGN}
	f(y) + \la \nabla_{\delta} f(y), x-y \ra  - \delta_1 \le f(x)  \le f(y) + \la \nabla_{\delta} f(y), x-y \ra + \frac{L}{2}\|x-y\|^2_2 + \delta_2,
\end{equation}
then
\pd{the corresponding versions of}
Gradient Method (GM) and Fast Gradient Method (FGM) 
\pd{have the convergence rate}
\begin{equation}
\label{estimate}
f(x_N) - f(x_*) = O\left(\frac{LR^2}{N^p} + \delta_1 + N^{p-1}\delta_2\right),
\end{equation}
where $p=1$ {corresponds to} GM and $p=2$ {corresponds to} FGM, $x_*$ is a solution of \eqref{eq:Problem}\pd{, $R$ is an upper bound for $\|x_0-x_*\|_2$.}
\pd{We show\footnote{For simplicity in the paper we consider the case $\delta_1 = \delta_2 = \delta$, but one can easily rewrite all the results of this paper to obtain \eqref{estimate}. See \cite{gorbunov2019optimal} for details.} that under an appropriate generalization of \eqref{eq:inexact_oracle_DGN} to
$$
f(y) + \psi_{\delta}(x, y)  - \delta_1 \le f(x)  \le f(y) + \psi_{\delta}(x, y)+ \frac{L}{2}\|x-y\|^2_2 + \delta_2
$$
as well as appropriate generalizations of GM and FGM, the sequence generated by these methods satisfies \eqref{estimate}.}
\dd{It should be noted that}, \pd{despite there are many variants of FGM, we are aware of only one which can be generalized for problems with inexact model, namely accelerated mirror descent type of FGM \cite{tseng2008accelerated,lan2012optimal,dvurechensky2017adaptive}. An important feature of this method is that it requires only one projection step on each iteration.} \pd{A primal-dual extension of the proposed framework is made in \cite{tyurin2019primal}.}



\pd{We also show that} in the case of  $\mu$-strongly convex objective (model) the estimate \eqref{estimate} can be improved to
$$f(x_N) - f(x_*) = O\left(\Delta f \exp\left(-O(1)\left(\frac{\mu}{L}\right)^{\frac{1}{p}}N\right) + \delta_1 + \left(\frac{L}{\mu}\right)^{\frac{p-1}{2}}\delta_2  \right),$$
where $\Delta f = f(x^0) - f(x_*)$, $p=1$ for GM and $p=2$ for restarted FGM.

In this paper \pd{we also propose a generalization of this inexact model framework for} saddle-point problems and variational inequalities. We consider universal (adaptive) generalizations \pd{in the spirit of} \cite{nesterov2015universal} and relative smoothness generalizations, generalizing \pd{the framework} \cite{bauschke2016descent,lu2018relatively} \pd{from optimization problems to saddle-point problems and VI}. We also investigate the sensitivity of the convergence results to the accuracy \pd{of auxiliary minimization on each iteration.}

\pd{Due to the lack of the space we only briefly mention here an extension of our framework for block-coordinate descent using the randomized version of FGM in \cite{dvurechensky2017randomized} and stochastic optimization problems using the ideas from \cite{gasnikov2017universal}.} For the \pd{latter case} we indicate that if we additionally assume that $\delta_1, \delta_2$ are independently chosen at each iteration random variables such that
$\mathds{E} \delta_1 = 0$ and $\delta_1,\sqrt{\delta_2}$ have correspondingly $\left(\delta_1'\right)^2$-subgaussian  variance and $\delta_2'$-subgaussian second moment
\cite{gorbunov2019optimal} then with a high probability \eqref{estimate} changes to
\begin{equation*}
f(x_N) - f(x_*) = \tilde{O}\left(\frac{LR^2}{N^p} + \frac{\delta_1'}{\sqrt{N}} + N^{p-1}\delta_2'\right).
\end{equation*}
From this result and mini-batch trick \cite{gasnikov2017universal} one can obtain the main estimates for convex and strongly convex stochastic optimization problems \cite{dvurechensky2016stochastic,gasnikov2016stochasticInter,gorbunov2019optimal,kulunchakov2019estimate}.

\pd{As further generalizations we point a generalization for tensor methods \cite{nesterov2018implementable,gasnikov2019near} and for incremental and variance reduction methods for finite-sum minimization \cite{defazio2016simple,lan2018random}.
}


\bibliography{PD_references}

\begin{thebibliography}{10}
\providecommand{\MR}{\relax\unskip\space MR }
\providecommand{\url}[1]{\normalfont{#1}}
\providecommand{\urlprefix}{Available at }

\bibitem{anikin2015modern}
A. Anikin, P. Dvurechensky, A. Gasnikov, A. Golov, A. Gornov, Y. Maximov, M.
  Mendel, and V. Spokoiny, \emph{Modern efficient numerical approaches to
  regularized regression problems in application to traffic demands matrix
  calculation from link loads}, in \emph{Proceedings of International
  conference ITAS-2015. Russia, Sochi}. 2015. arXiv:1508.00858.

\bibitem{anikin2017dual}
A.S. Anikin, A.V. Gasnikov, P.E. Dvurechensky, A.I. Tyurin, and A.V. Chernov,
  \emph{Dual approaches to the minimization of strongly convex functionals with
  a simple structure under affine constraints}, Computational Mathematics and
  Mathematical Physics 57 (2017), pp. 1262--1276.

\bibitem{antonakopoulos2020adaptive}
K. Antonakopoulos, E.V. Belmega, and P. Mertikopoulos, \emph{Adaptive
  extra-gradient methods for min-max optimization and games}, arXiv:2010.12100
  (2020).

\bibitem{antonakopoulos2019adaptive}
K. Antonakopoulos, V. Belmega, and P. Mertikopoulos, \emph{An adaptive
  mirror-prox method for variational inequalities with singular operators}, in
  \emph{Advances in Neural Information Processing Systems 32}, H. Wallach, H.
  Larochelle, A. Beygelzimer, F. dAlch\'{e}  Buc, E. Fox, and R. Garnett, eds.,
  Curran Associates, Inc.,  2019, pp. 8455--8465.

\bibitem{baimurzina2017universal}
D. Baimurzina, A. Gasnikov, E. Gasnikova, P. Dvurechensky, E. Ershov, M.
  Kubentaeva, and A. Lagunovskaya, \emph{Universal similar triangulars method
  for searching equilibriums in traffic flow distribution models}, Journal of
  Computational Mathematics and Mathematical Physics 59 (2019), pp. 21--36.

\bibitem{Bao_Khanh}
T.Q. Bao and P.Q. Khanh, \emph{Some algorithms for solving mixed variational
  inequalities}, Acta Mathematica Vietnamica 31 (2006), pp. 77--98.

\bibitem{bauschke2016descent}
H.H. Bauschke, J. Bolte, and M. Teboulle, \emph{A descent lemma beyond
  lipschitz gradient continuity: first-order methods revisited and
  applications}, Mathematics of Operations Research 42 (2016), pp. 330--348.

\bibitem{beck2009fast}
A. Beck and M. Teboulle, \emph{A fast iterative shrinkage-thresholding
  algorithm for linear inverse problems}, SIAM Journal on Imaging Sciences 2
  (2009), pp. 183--202. \urlprefix\url{https://doi.org/10.1137/080716542}.

\bibitem{bogolubsky2016learning}
L. Bogolubsky, P. Dvurechensky, A. Gasnikov, G. Gusev, Y. Nesterov, A.M.
  Raigorodskii, A. Tikhonov, and M. Zhukovskii, \emph{Learning supervised
  pagerank with gradient-based and gradient-free optimization methods}, in
  \emph{Advances in Neural Information Processing Systems 29}, D.D. Lee, M.
  Sugiyama, U.V. Luxburg, I. Guyon, and R. Garnett, eds., Curran Associates,
  Inc.,  2016, pp. 4914--4922. arXiv:1603.00717.

\bibitem{chambolle2011first-order}
A. Chambolle and T. Pock, \emph{A first-order primal-dual algorithm for convex
  problems with applications to imaging}, Journal of Mathematical Imaging and
  Vision 40 (2011), pp. 120--145.

\bibitem{chen1993convergence}
G. Chen and M. Teboulle, \emph{Convergence analysis of a proximal-like
  minimization algorithm using bregman functions}, SIAM Journal on Optimization
  3 (1993), pp. 538--543.

\bibitem{chernov2016fast}
A. Chernov, P. Dvurechensky, and A. Gasnikov, \emph{Fast Primal-Dual Gradient
  Method for Strongly Convex Minimization Problems with Linear Constraints}, in
  \emph{Discrete Optimization and Operations Research: 9th International
  Conference, DOOR 2016, Vladivostok, Russia, September 19-23, 2016,
  Proceedings}, Y. Kochetov, M. Khachay, V. Beresnev, E. Nurminski, and P.
  Pardalos, eds. Springer International Publishing, 2016, pp. 391--403.

\bibitem{cohen2021relative}
M.B. Cohen, A. Sidford, and K. Tian, \emph{Relative Lipschitzness in
  Extragradient Methods and a Direct Recipe for Acceleration}, in \emph{12th
  Innovations in Theoretical Computer Science Conference, {ITCS} 2021, January
  6-8, 2021, Virtual Conference}, J.R. Lee, ed., LIPIcs Vol. 185. Schloss
  Dagstuhl - Leibniz-Zentrum f{\"{u}}r Informatik, 2021, pp. 62:1--62:18.
  \urlprefix\url{https://doi.org/10.4230/LIPIcs.ITCS.2021.62}.

\bibitem{defazio2016simple}
A. Defazio, \emph{A simple practical accelerated method for finite sums}, in
  \emph{Advances in neural information processing systems}. 2016, pp. 676--684.

\bibitem{devolder2013firstCORE}
O. Devolder, F. Glineur, and Y. Nesterov, \emph{First-order methods with
  inexact oracle: the strongly convex case}, CORE Discussion Papers 2013/16
  (2013).

\bibitem{devolder2014first}
O. Devolder, F. Glineur, and Y. Nesterov, \emph{First-order methods of smooth
  convex optimization with inexact oracle}, Mathematical Programming 146
  (2014), pp. 37--75.
  \urlprefix\url{http://dx.doi.org/10.1007/s10107-013-0677-5}.

\bibitem{dragomir2019optimal}
R.A. Dragomir, A. Taylor, A. d'Aspremont, and J. Bolte, \emph{Optimal
  complexity and certification of bregman first-order methods},
  arXiv:1911.08510  (2019).

\bibitem{drusvyatskiy2019nonsmooth}
D. Drusvyatskiy, A.D. Ioffe, and A.S. Lewis, \emph{Nonsmooth optimization using
  {T}aylor-like models: error bounds, convergence, and termination criteria},
  Mathematical Programming  (2019). arXiv:1610.03446.

\bibitem{dvinskikh2019primal}
D. Dvinskikh, E. Gorbunov, A. Gasnikov, P. Dvurechensky, and C.A. Uribe,
  \emph{On Primal-Dual Approach for Distributed Stochastic Convex Optimization
  over Networks}, in \emph{2019 IEEE Conference on Decision and Control (CDC)}.
  2019. (accepted), arXiv:1903.09844.

\bibitem{dvurechensky2018decentralize}
P. Dvurechensky, D. Dvinskikh, A. Gasnikov, C.A. Uribe, and A. Nedić,
  \emph{Decentralize and Randomize: Faster Algorithm for {W}asserstein
  Barycenters}, in \emph{Advances in Neural Information Processing Systems 31},
  S. Bengio, H. Wallach, H. Larochelle, K. Grauman, N. Cesa-Bianchi, and R.
  Garnett, eds. Curran Associates, Inc., NIPS'18, 2018, pp. 10783--10793.
  \urlprefix\url{http://papers.nips.cc/paper/8274-decentralize-and-randomize-faster-algorithm-for-wasserstein-barycenters.pdf},
  arXiv:1802.04367.

\bibitem{dvurechensky2016stochastic}
P. Dvurechensky and A. Gasnikov, \emph{Stochastic intermediate gradient method
  for convex problems with stochastic inexact oracle}, Journal of Optimization
  Theory and Applications 171 (2016), pp. 121--145.
  \urlprefix\url{http://dx.doi.org/10.1007/s10957-016-0999-6}.

\bibitem{dvurechensky2016primal-dual}
P. Dvurechensky, A. Gasnikov, E. Gasnikova, S. Matsievsky, A. Rodomanov, and I.
  Usik, \emph{Primal-Dual Method for Searching Equilibrium in Hierarchical
  Congestion Population Games}, in \emph{Supplementary Proceedings of the 9th
  International Conference on Discrete Optimization and Operations Research and
  Scientific School (DOOR 2016) Vladivostok, Russia, September 19 - 23, 2016}.
  2016, pp. 584--595. arXiv:1606.08988.

\bibitem{dvurechensky2017universal}
P. Dvurechensky, A. Gasnikov, and D. Kamzolov, \emph{Universal intermediate
  gradient method for convex problems with inexact oracle}, arXiv:1712.06036,
  Opt. Meth. \& Software (accepted)  (2019).
  \urlprefix\url{http://dx.doi.org/10.1080/10556788.2019.1711079}.

\bibitem{dvurechensky2018computational}
P. Dvurechensky, A. Gasnikov, and A. Kroshnin, \emph{Computational Optimal
  Transport: Complexity by Accelerated Gradient Descent Is Better Than by
  {S}inkhorn’s Algorithm}, in \emph{Proceedings of the 35th International
  Conference on Machine Learning}, J. Dy and A. Krause, eds., Proceedings of
  Machine Learning Research Vol.~80. 2018, pp. 1367--1376. arXiv:1802.04367.

\bibitem{dvurechensky2017adaptive}
P. Dvurechensky, A. Gasnikov, S. Omelchenko, and A. Tiurin, \emph{Adaptive
  similar triangles method: a stable alternative to sinkhorn's algorithm for
  regularized optimal transport}, arXiv:1706.07622  (2017).

\bibitem{dvurechensky2017randomized}
P. Dvurechensky, A. Gasnikov, and A. Tiurin, \emph{Randomized similar triangles
  method: A unifying framework for accelerated randomized optimization methods
  (coordinate descent, directional search, derivative-free method)},
  arXiv:1707.08486  (2017).

\bibitem{frank1956algorithm}
M. Frank and P. Wolfe, \emph{An algorithm for quadratic programming}, Naval
  Research Logistics Quarterly 3 (1956), pp. 95--110.

\bibitem{gasnikov2016stochasticInter}
A.V. Gasnikov and P.E. Dvurechensky, \emph{Stochastic intermediate gradient
  method for convex optimization problems}, Doklady Mathematics 93 (2016), pp.
  148--151.

\bibitem{gasnikov2019adaptive}
A.V. Gasnikov, P.E. Dvurechensky, F.S. Stonyakin, and A.A. Titov, \emph{An
  adaptive proximal method for variational inequalities}, Computational
  Mathematics and Mathematical Physics 59 (2019), pp. 836--841.
  \urlprefix\url{https://doi.org/10.1134/S0965542519050075}.

\bibitem{gasnikov2017universal}
A. Gasnikov, \emph{Universal gradient descent}, arXiv preprint arXiv:1711.00394
   (2017).

\bibitem{gasnikov2019near}
A. Gasnikov, P. Dvurechensky, E. Gorbunov, E. Vorontsova, D. Selikhanovych,
  C.A. Uribe, B. Jiang, H. Wang, S. Zhang, S. Bubeck, Q. Jiang, Y.T. Lee, Y.
  Li, and A. Sidford, \emph{Near Optimal Methods for Minimizing Convex
  Functions with Lipschitz $p$-th Derivatives}, in \emph{Proceedings of the
  Thirty-Second Conference on Learning Theory}, A. Beygelzimer and D. Hsu,
  eds., Proceedings of Machine Learning Research Vol.~99, 25--28 Jun, Phoenix,
  USA. PMLR, 2019, pp. 1392--1393.
  \urlprefix\url{http://proceedings.mlr.press/v99/gasnikov19b.html},
  arXiv:1809.00382.

\bibitem{tyurin2017fast}
A. Gasnikov and A. Tyurin, \emph{Fast gradient descent for convex minimization
  problems with an oracle producing a $(\delta, l)$-model of function at the
  requested point}, Computational Mathematics and Mathematical Physics 59
  (2019), pp. 1085--1097.

\bibitem{gasnikov2018universal}
A.V. Gasnikov and Y.E. Nesterov, \emph{Universal method for stochastic
  composite optimization problems}, Computational Mathematics and Mathematical
  Physics 58 (2018), pp. 48--64. First appeared in arXiv:1604.05275.

\bibitem{gorbunov2019optimal}
E. Gorbunov, D. Dvinskikh, and A. Gasnikov, \emph{Optimal decentralized
  distributed algorithms for stochastic convex optimization}, arXiv preprint
  arXiv:1911.07363  (2019).

\bibitem{guminov2019accelerated}
S.V. Guminov, Y.E. Nesterov, P.E. Dvurechensky, and A.V. Gasnikov,
  \emph{Accelerated primal-dual gradient descent with linesearch for convex,
  nonconvex, and nonsmooth optimization problems}, Doklady Mathematics 99
  (2019), pp. 125--128.

\bibitem{har2015cond}
Z. Harchaoui, A. Juditsky, and A. Nemirovski, \emph{Conditional gradient
  algorithms for norm-regularized smooth convex optimization}, // Mathematical
  Programming 152 (2015), pp. 75--112.

\bibitem{iusem2019variance}
A.N. Iusem, A. Jofré, R.I. Oliveira, and P. Thompson, \emph{Variance-based
  extragradient methods with line search for stochastic variational
  inequalities}, SIAM Journal on Optimization 29 (2019), pp. 175--206.
  arXiv:1703.00262.

\bibitem{jaggi2013revisiting}
M. Jaggi, \emph{Revisiting Frank-Wolfe: Projection-Free Sparse Convex
  Optimization.}, in \emph{ICML (1)}. 2013, pp. 427--435.

\bibitem{Konnov_2017}
I. Konnov and R. Salahutdin, \emph{Two-level iterative method for
  non-stationary mixed variational inequalities}, Izvestija vysshih uchebnyh
  zavedenij. Matematika 61 (2017), pp. 50--61.

\bibitem{kroshnin2019complexity}
A. Kroshnin, N. Tupitsa, D. Dvinskikh, P. Dvurechensky, A. Gasnikov, and C.
  Uribe, \emph{On the Complexity of Approximating {W}asserstein Barycenters},
  in \emph{Proceedings of the 36th International Conference on Machine
  Learning}, K. Chaudhuri and R. Salakhutdinov, eds., Proceedings of Machine
  Learning Research Vol.~97, 09--15 Jun, Long Beach, California, USA. PMLR,
  2019, pp. 3530--3540. arXiv:1901.08686.

\bibitem{kulunchakov2019estimate}
A. Kulunchakov and J. Mairal, \emph{Estimate sequences for stochastic composite
  optimization: Variance reduction, acceleration, and robustness to noise},
  arXiv preprint arXiv:1901.08788  (2019).

\bibitem{lan2012optimal}
G. Lan, \emph{An optimal method for stochastic composite optimization},
  Mathematical Programming 133 (2012), pp. 365--397.
  \urlprefix\url{https://doi.org/10.1007/s10107-010-0434-y}, Firs appeared in
  June 2008.

\bibitem{lan2015bundle}
G. Lan, \emph{Bundle-level type methods uniformly optimal for smooth and
  nonsmooth convex optimization}, Mathematical Programming 149 (2015), pp.
  1--45.

\bibitem{lan2018random}
G. Lan and Y. Zhou, \emph{Random gradient extrapolation for distributed and
  stochastic optimization}, SIAM Journal on Optimization 28 (2018), pp.
  2753--2782.

\bibitem{lin2015universal}
H. Lin, J. Mairal, and Z. Harchaoui, \emph{A Universal Catalyst for First-order
  Optimization}, in \emph{Proceedings of the 28th International Conference on
  Neural Information Processing Systems}, Cambridge, MA, USA. MIT Press,
  NIPS'15, 2015, pp. 3384--3392.
  \urlprefix\url{http://dl.acm.org/citation.cfm?id=2969442.2969617}.

\bibitem{lu2018relatively}
H. Lu, R.M. Freund, and Y. Nesterov, \emph{Relatively smooth convex
  optimization by first-order methods, and applications}, SIAM Journal on
  Optimization 28 (2018), pp. 333--354.

\bibitem{mairal2013optimization}
J. Mairal, \emph{Optimization with first-order surrogate functions}, in
  \emph{International Conference on Machine Learning}. 2013, pp. 783--791.

\bibitem{malitsky2020proximal}
Y. Malitsky, \emph{Proximal extrapolated gradient methods for variational
  inequalities}, Optimization Methods and Software 33 (2018), pp. 140--164.
  \urlprefix\url{https://doi.org/10.1080/10556788.2017.1300899}, PMID:
  29348705.

\bibitem{malitsky2020golden}
Y. Malitsky, \emph{Golden ratio algorithms for variational inequalities},
  Mathematical Programming 184 (2020), pp. 383--410.
  \urlprefix\url{https://doi.org/10.1007/s10107-019-01416-w}.

\bibitem{Mastroeni}
G. Mastroeni, \emph{On auxiliary principle for equilibrium problems},
  Publicatione del Departimento di Mathematica Dell’Universita di Pisa 3
  (2000), pp. 1244--1258.

\bibitem{nemirovski2004prox}
A. Nemirovski, \emph{Prox-method with rate of convergence ${O}(1/t)$ for
  variational inequalities with {L}ipschitz continuous monotone operators and
  smooth convex-concave saddle point problems}, SIAM Journal on Optimization 15
  (2004), pp. 229--251.

\bibitem{nemirovski2010accuracy}
A. Nemirovski, S. Onn, and U.G. Rothblum, \emph{Accuracy certificates for
  computational problems with convex structure}, Mathematics of Operations
  Research 35 (2010), pp. 52--78.

\bibitem{nemirovskii1985optimal}
A. Nemirovskii and Y. Nesterov, \emph{Optimal methods of smooth convex
  minimization}, USSR Computational Mathematics and Mathematical Physics 25
  (1985), pp. 21 -- 30.
  \urlprefix\url{http://www.sciencedirect.com/science/article/pii/0041555385901004}.

\bibitem{nesterov2009primal-dual}
Y. Nesterov, \emph{Primal-dual subgradient methods for convex problems},
  Mathematical Programming 120 (2009), pp. 221--259.
  \urlprefix\url{https://doi.org/10.1007/s10107-007-0149-x}, First appeared in
  2005 as CORE discussion paper 2005/67.

\bibitem{nesterov2013gradient}
Y. Nesterov, \emph{Gradient methods for minimizing composite functions},
  Mathematical Programming 140 (2013), pp. 125--161. First appeared in 2007 as
  CORE discussion paper 2007/76.

\bibitem{nesterov2015universal}
Y. Nesterov, \emph{Universal gradient methods for convex optimization
  problems}, Mathematical Programming 152 (2015), pp. 381--404.
  \urlprefix\url{http://dx.doi.org/10.1007/s10107-014-0790-0}.

\bibitem{nesterov2018complexity}
Y. Nesterov, \emph{Complexity bounds for primal-dual methods minimizing the
  model of objective function}, Math. Program. 171 (2018), pp. 311--330.
  \urlprefix\url{https://doi.org/10.1007/s10107-017-1188-6}.

\bibitem{nesterov2018implementable}
Y. Nesterov, \emph{Implementable tensor methods in unconstrained convex
  optimization}, Tech. {R}ep., CORE UCL,  2018.
  \urlprefix\url{https://alfresco.uclouvain.be/alfresco/service/guest/streamDownload/workspace/SpacesStore/aabc2323-0bc1-40d4-9653-1c29971e7bd8/coredp2018_05web.pdf},
  CORE Discussion Paper 2018/05.

\bibitem{nesterov2018lectures}
Y. Nesterov, \emph{Lectures on convex optimization}, Vol. 137, Springer
  International Publishing, 2018.

\bibitem{RePEc:cor:louvco:2018001}
Y. Nesterov, \emph{{Soft clustering by convex electoral model}}, CORE
  Discussion Papers 2018001, Université catholique de Louvain, Center for
  Operations Research and Econometrics (CORE),  2018.
  \urlprefix\url{https://ideas.repec.org/p/cor/louvco/2018001.html}.

\bibitem{nesterov2021gradient}
Y. Nesterov and M.I. Florea, \emph{Gradient methods with memory}, Optimization
  Methods and Software 0 (2021), pp. 1--18.
  \urlprefix\url{https://doi.org/10.1080/10556788.2020.1858831}.

\bibitem{nesterov2018primal-dual}
Y. Nesterov, A. Gasnikov, S. Guminov, and P. Dvurechensky, \emph{Primal-dual
  accelerated gradient methods with small-dimensional relaxation oracle},
  arXiv:1809.05895  (2018). Submitted to Optimization Methods \& Software.

\bibitem{nesterov2006cubic}
Y. Nesterov and B. Polyak, \emph{Cubic regularization of newton method and its
  global performance}, Mathematical Programming 108 (2006), pp. 177--205.
  \urlprefix\url{http://dx.doi.org/10.1007/s10107-006-0706-8}.

\bibitem{ochs2017non}
P. Ochs, J. Fadili, and T. Brox, \emph{Non-smooth non-convex bregman
  minimization: Unification and new algorithms}, arXiv preprint
  arXiv:1707.02278  (2017).

\bibitem{ogaltsov2019adaptive}
A. Ogaltsov, D. Dvinskikh, P. Dvurechensky, A. Gasnikov, and V. Spokoiny,
  \emph{Adaptive gradient descent for convex and non-convex stochastic
  optimization}, arXiv:1911.08380  (2019). Submitted to IFAC 2020 Congress.

\bibitem{Optimal}
Y. Ouyang and Y. Xu, \emph{Lower complexity bounds of first-order methods for
  convex-concave bilinear saddle-point problems}, arXiv preprint arXiv:
  1808.02901  (2018).

\bibitem{parikh2014prox}
N. Parikh and S. Boyd, \emph{Proximal algorithms}, Foundations and Trends® in
  Optimization 1 (2014), pp. 127--239.
  \urlprefix\url{http://dx.doi.org/10.1561/2400000003}.

\bibitem{scaman2017optimal}
K. Scaman, F. Bach, S. Bubeck, Y.T. Lee, and L. Massouli{\'e}, \emph{Optimal
  Algorithms for Smooth and Strongly Convex Distributed Optimization in
  Networks}, in \emph{Proceedings of the 34th International Conference on
  Machine Learning}, D. Precup and Y.W. Teh, eds., Proceedings of Machine
  Learning Research Vol.~70, 06--11 Aug, International Convention Centre,
  Sydney, Australia. PMLR, 2017, pp. 3027--3036.
  \urlprefix\url{http://proceedings.mlr.press/v70/scaman17a.html}.

\bibitem{stonyakin2018generalized}
F. Stonyakin, A. Gasnikov, P. Dvurechensky, M. Alkousa, and A. Titov,
  \emph{Generalized {M}irror {P}rox for monotone variational inequalities:
  Universality and inexact oracle}, arXiv:1806.05140  (2018).

\bibitem{stonyakin2019inexact}
F. Stonyakin, A. Gasnikov, A. Tyurin, D. Pasechnyuk, A. Agafonov, P.
  Dvurechensky, D. Dvinskikh, A. Kroshnin, and V. Piskunova, \emph{Inexact
  model: A framework for optimization and variational inequalities},
  arXiv:1902.00990  (2019). WIAS Preprint No. 2679.

\bibitem{stonyakin2019gradient}
F.S. Stonyakin, D. Dvinskikh, P. Dvurechensky, A. Kroshnin, O. Kuznetsova, A.
  Agafonov, A. Gasnikov, A. Tyurin, C.A. Uribe, D. Pasechnyuk, and S.
  Artamonov, \emph{Gradient Methods for Problems with Inexact Model of the
  Objective}, in \emph{Mathematical Optimization Theory and Operations
  Research}, M. Khachay, Y. Kochetov, and P. Pardalos, eds., Cham. Springer
  International Publishing, 2019, pp. 97--114. arXiv:1902.09001.

\bibitem{tseng2008accelerated}
P. Tseng, \emph{On accelerated proximal gradient methods for convex-concave
  optimization}, Tech. {R}ep., MIT,  2008.
  \urlprefix\url{http://www.mit.edu/~dimitrib/PTseng/papers/apgm.pdf}.

\bibitem{tyurin2019primal}
A. Tyurin, \emph{Primal-dual fast gradient method with a model}, arXiv preprint
  arXiv:1906.10107  (2019).

\bibitem{uribe2018distributed}
C.A. Uribe, D. Dvinskikh, P. Dvurechensky, A. Gasnikov, and A. Nedi\'c,
  \emph{Distributed Computation of {W}asserstein Barycenters over Networks}, in
  \emph{2018 IEEE 57th Annual Conference on Decision and Control (CDC)}. 2018.
  Accepted, arXiv:1803.02933.

\bibitem{xie2018fast}
Y. Xie, X. Wang, R. Wang, and H. Zha, \emph{A fast proximal point method for
  wasserstein distance}, arXiv preprint arXiv:1802.04307  (2018).

\bibitem{zhao2020analysis}
R. Zhao and R.M. Freund, \emph{Analysis of the frank-wolfe method for
  logarithmically-homogeneous barriers, with an extension}, arXiv:2010.08999
  (2020).

\end{thebibliography}

\appendix

\section{Auxiliary facts}
Let us comment the inexact solution of the auxiliary problem.
\begin{remark}
\label{RemarkInexact}
We can show that if $\widetilde{x} \in \text{Arg}\min_{x \in Q}^{\widetilde{\delta}}\Psi(x)$, then $\Psi(\widetilde{x}) - \Psi(x_*) \leq \delta$. Indeed, we have $\Psi(x_*) \geq \Psi(\widetilde{x}) + \langle h, x_* - \widetilde{x} \rangle \geq \Psi(\widetilde{x}) - \widetilde{\delta}$. The converse statement is not always true. However, for some general cases we can resolve the problem (see \cite{tyurin2017fast} and Example \ref{appendix_inexact}).
\end{remark}

\begin{example}
\label{appendix_inexact}
Let us show an example, how we can resolve the problem in Remark \ref{RemarkInexact}. \gav{Note, that if $\Psi(x)$ is $\mu$-strongly convex; has $L$-Lipschitz continuous gradient in $\|\cdot\|$ norm. To say more precisely $$L=\max_{\|h\|\le 1, x \in[\widetilde{x},x_*]}\la h,\nabla^2 \Psi(x)h \ra.$$
and $R = \max_{x,y\in Q} \|x-y\|$, then  $\Psi(\widetilde{x}) - \Psi(x_*)\le\widetilde{\epsilon}$ entails that \cite{stonyakin2019gradient}
\begin{gather}\label{inexact}
\widetilde{\delta}\le (LR+\|\nabla\Psi(x_*)\|_*)\sqrt{2\tilde{\e}/\mu},
\end{gather}
where $x_* = \argmin_{x\in Q}\Psi(x)$.}
\gav{If one can guarantee that $\nabla\Psi(x_*) = 0$, then \eqref{inexact} can be improved $\widetilde{\delta}\le R\sqrt{2L\widetilde{\e}}.$ }
\end{example}

\section{Proof for Lemma \ref{lemma:a_n_sequence}}
\label{proof:a_n_sequence}

\begin{proof}
In view of the definition \eqref{alpha_def_strong} of sequence $\alpha_{k+1}$, we have:
\begin{align*}
A_N &\leq A_N(1 + \mu A_{N-1} + m A_{N-1}) = L_{N}(A_N - A_{N-1})^2\\
&\leq L_{N}(A_N^{1/2} - A_{N-1}^{1/2})^2(A_N^{1/2} + A_{N-1}^{1/2})^2 \leq 4L_{N} A_N (A_N^{1/2} - A_{N-1}^{1/2})^2.
\end{align*}
We can see that
\begin{align*}
A_N^{1/2} \geq A_{N-1}^{1/2} + \frac{1}{2\sqrt{L_N}}
\end{align*}
and
\begin{align*}
A_N \geq \frac{1}{4}\left(\sum_{k=0}^{N-1}\frac{1}{\sqrt{L_{k+1}}}\right)^2.
\end{align*}
For the case when $\mu + m> 0$ we obtain:
\begin{align*}
(\mu + m) A_{N-1} A_{N} \leq A_N(1 + \mu A_{N-1} + m A_{N-1}) \leq 4L_{N} A_N (A_N^{1/2} - A_{N-1}^{1/2})^2.
\end{align*}
From the fact that $A_1 = 1 / L_1$ and the last inequality we can show that
\begin{align*}
A_N^{1/2} \geq \left(1 + \sqrt{\frac{\mu + m}{4L_N}}\right)A_{N-1}^{1/2} \geq \frac{1}{\sqrt{L_1}}\prod_{k=1}^{N-1}\left(1 + \sqrt{\frac{\mu + m}{4L_{k+1}}}\right).
\end{align*}
\end{proof}

\section{Fast gradient method with $(\delta, L, \mu, m, V, \|\cdot\|)$-model. Restart technique.}

Let us consider the case of a smooth strongly convex function $f$ and show how to accelerate the work of  Algorithm \ref{FastAlg2_strong} using the restart technique. Let us assume that
\begin{equation}
\psi_{\delta}(x,x_*) \ge 0\,\,\, \forall x \in Q.
\end{equation}
Note that this assumption is natural, e.g. $\psi_{\delta}(x, y):= \langle\nabla f(y), x - y \rangle \,\,\, \forall x,y \in Q$. We also modify the concept of $\mu$-strong convexity in the following way.

\begin{definition}\label{defRelStronglyConvexRest}
 We say that the function $f$ is a left relative $\mu$-strongly convex if the following inequality
   \begin{gather*}
     \mu V[x](y) \leq f(x) - f(y)- \psi_{\delta}(x, y) \quad \forall x, y \in Q
   \end{gather*}
holds.
\end{definition}

\begin{remark}\label{RemEuclidBreg}
Let us remind that if $d(x-y) \leq C_n\norm{x-y}^2$ for $C_n = O(\log n)$, (where $n$ is the dimension of vectors from $Q$)
then $V[y](x) \leq C_n\norm{x-y}^2$. This assumption is true for many standard proximal setups. In this case the condition of $(\mu C_n)$-strong convexity
$$\mu C_n \norm{x-y}^2 + f_\delta(y) + \psi_\delta(x,y) \leq f(x)$$
entails right relative strong convexity:
$$\mu V[y](x) + f_\delta(y) + \psi_\delta(x,y) \leq f(x).$$
\end{remark}

Note that the concepts of right and left relative strongly convexity from Definitions \ref{defRelStronglyConvex} and \ref{defRelStronglyConvexRest} are equivalent in the case of an assumption from Remark \ref{RemEuclidBreg} ($V[x](y) \leq C_n \|x - y\|^2$ for each $x, y \in Q$).

We show that using the restart technique can also accelerate the work of non-adaptive version of Algorithm \ref{Alg2} ($L_{k+1} = L$) for $(\delta, L, 0, 0, V, \|\cdot\|)$-model and relative $\mu$-strongly convex function $f$ in the sense of Definition \ref{defRelStronglyConvexRest}:
\begin{gather*}
  \mu V[x](y) + f(y)+ \psi_{\delta}(x, y) - \delta \leq f(x) \leq f(y) + \psi_{\delta}(x, y) + \frac{L}{2}\|x - y\|^2 + \delta.
\end{gather*}
for each $x, y \in Q$.
By Theorem \ref{Th:fast_str_conv_adap} and Remark \ref{remark:FastGradConvRate}, we have:
\begin{equation}\label{FGeq1}
f(x_N) - f(x_*) \leq \dfrac{4LV[x_0](x_*)}{N^2}+\dfrac{4L\widetilde{\delta}}{N} + 2N\delta.
\end{equation}
Consider the case of relatively $\mu$-strongly convex function $f$. We will use the restart technique to obtain the method for strongly convex functions.

\begin{theorem}
   Let $f$ be a left relative $\mu$-strongly convex function and $ \psi_{\delta}(x,y)$ is a \\ $(\delta, L, 0, 0, V, \|\cdot\|)$-model. Let $\delta$ and $\widetilde{\delta}$ satisfy
   $\frac{4\mu\sqrt{10}}{L}\left(5\delta{\left \lceil{\sqrt{\frac{L}{\mu}}}\right \rceil}^3+{\widetilde{\delta}}L{\left \lceil{\sqrt{\frac{L}{\mu}}}\right \rceil} \right) \leq \varepsilon.$
   Then, using the restarts of Algorithm \ref{Alg2}, we need
   $$ N = \left \lceil \log_2 \dfrac{\mu R^2}{\varepsilon} \right\rceil \cdot \left \lceil \sqrt{\dfrac{10L}{\mu}}\right \rceil.$$
   iterations to achieve an $\varepsilon$ accuracy by function: $f(x_N) - f(x_*) \leq \varepsilon$.
\end{theorem}
\begin{proof}

By $\eqref{FGeq1}$ and Definition \ref{defRelStronglyConvexRest}, we have:
\begin{equation}\label{FGeq2}
\mu V[x_{N_1}](x_*) \le f(x_{N_1}) - f(x_*) \le \dfrac{4LV[x_0](x_*)}{N^2} + \dfrac{4L\widetilde{\delta}}{N} +2N\delta.
\end{equation}
Let's choose $N_1$ so that the following inequality holds:
\begin{equation}\label{FGeq4}
\dfrac{4L\widetilde{\delta}}{N_1}+2N_1\delta \leq \dfrac{LV[x_0](x_*)}{N_1^2}.
\end{equation}
We restart method as $V[x_{N_1}](x_*)\leq\frac{V[x_0](x_*)}{2}.$
Using  \eqref{FGeq2}, we obtain an estimation for the number of iterations on the first restart:
 $\dfrac{5L}{\mu N_1^2} \le \dfrac{1}{2}.$
Therefore, let's choose
\begin{equation}\label{FGeq3}
N_1 = \left \lceil \sqrt{\dfrac{10L}{\mu}}\right \rceil.
\end{equation}
Then after $N_1$ iterations we restart method. Similarly, we restart after $N_2$ iterations, such that $V[x_{N_2}](x_*)\leq \frac{V[x_{N_1}](x_*)}{2}$. We obtain $N_2 = \left \lceil \sqrt{\frac{10L}{\mu}}\right \rceil.$
So, after $p$-th restart the total number of iterations is $M = p \cdot  \left \lceil \sqrt{\frac{10L}{\mu}}\right \rceil.$

Now let's consider how many iterations are needed to achieve the accuracy $\varepsilon = f(x_{N_p}) - f(x_*)$.
From \eqref{FGeq1} and \eqref{FGeq3} we take $p  = \left \lceil \log_2 \dfrac{\mu R^2}{\varepsilon} \right\rceil$
and the total number of iterations is $M = \left \lceil \log_2 \dfrac{\mu R^2}{\varepsilon} \right\rceil \cdot \left \lceil \sqrt{\dfrac{10L}{\mu}}\right \rceil.$

We have chosen our errors as $\delta$ and $\widetilde{\delta}$ to satisfy \eqref{FGeq4}. Indeed, from \eqref{FGeq4} using $N_k = \left \lceil \sqrt{\dfrac{10L}{\mu}}\right \rceil$ we can deduce the following inequality:
$$\varepsilon \ge \dfrac{4\sqrt{10}\mu}{L}\left(5\delta\left \lceil \sqrt{\dfrac{L}{\mu}}\right \rceil^3 +  \widetilde{\delta} L \left \lceil \sqrt{\dfrac{L}{\mu}}\right \rceil \right).$$
One can see that such a choice of $\delta$ and $\widetilde{\delta}$ as above satisfies that inequality.
\end{proof}

\begin{remark}
Partially, we can choose $\delta = O\left(\frac{\varepsilon L}{\mu  \left \lceil \sqrt{\frac{10L}{\mu}}\right \rceil^3} \right)$ and $\widetilde{\delta} = O\left(\frac{\varepsilon}{\mu  \left \lceil \sqrt{\frac{10L}{\mu}}\right \rceil}\right)$.
\end{remark}

\section{Numerical results for some different geometrical problems }\label{app_numerical_1}
To show the advantages of the proposed Algorithm \ref{Alg:UMPModel}, a series of numerical experiments were performed for different examples.

We consider an example of a variational inequality with a non-smooth and non-strongly monotone operator, i.e. with $\nu = 0$, for which the proposed Algorithm \ref{Alg:UMPModel}, due to its adaptivity to the level of smoothness of the problem, in practice works with iteration complexity much smaller than the predicted by theory.
We compare the performance of Algorithm \ref{Alg:UMPModel}, with two different approaches for step-size selection, the first one is the proposed procedure in this paper, i.e. with $L_{k+1} = 2^{i_k - 1} L_k$, the second procedure with $L_{k+1} = 2^{i_k} L_k$, i.e.  with a non-increasing stepsize similarly to \cite{antonakopoulos2019adaptive}, for three different geometrical problems with some non-smooth functional constraints.

Let $Q$ be a convex compact set and $A_k \in \mathbb{R}^n, k=1,...,N$ be a given set of $N$ points.

 \textbf{The first problem} is naturally connected with the analogue of the well-known {\it smallest covering circle problem}. For this problem we need to solve the following optimization problem
\begin{equation}\label{prob_1}
     \min_{x \in Q} \left\{f(x):= \max_{1\leq k \leq N} \|x - A_k\|_2 \left| \; \varphi_p(x):= \sum_{i=1}^n \alpha_{pi}|x_i| -  1 \leq 0 \right., \; p=1,...,m \right\}.
\end{equation}

\textbf{The second problem} is naturally connected with the analogue of the well-known {\it Fermat--Torricelli--Steiner problem}. For this problem we need to solve the following optimization problem
\begin{equation}\label{prob_2}
	 \min_{x \in Q} \left\{f(x):= \sum\limits_{k=1}^N \|x - A_k\|_2 \left| \; \varphi_p(x):= \sum_{i=1}^n \alpha_{pi}|x_i| -  1 \leq 0 \right., \; p=1,...,m \right\}.
\end{equation}
	
\textbf{The third problem} is connected with \textit{problem of the best approximation of the distance from the given set $Q$}. For this problem, let $A \notin Q$ be a given point,  we need to solve the following optimization problem
\begin{equation}\label{prob_3}
	\min_{x \in Q} \left\{f(x):= \|x - A\|_2 \left| \; \varphi_p(x):= \sum_{i=1}^n \alpha_{pi}|x_i| -  1 \leq 0 \right., \; p=1,...,m \right\}.
\end{equation}

The coefficients $\alpha_{pi}$, in \eqref{prob_1}, \eqref{prob_2} and \eqref{prob_3} are drawn from the standard normal distribution and then truncated to be positive.
	
The corresponding Lagrange saddle-point problem is defined as
$$
	\min_{x \in Q_1} \max_{\overrightarrow{\lambda} = (\lambda_1,\lambda_2,\ldots,\lambda_m)^T \in Q_2} L(x,\lambda):=f(x)+\sum\limits_{p=1}^m\lambda_p\varphi_p(x),
$$
This problem is equivalent to the variational inequality with monotone bounded operator
$$
	g(x,\lambda)=
	\begin{pmatrix}
	\nabla f(x)+\sum\limits_{p=1}^m\lambda_p\nabla\varphi_p(x), \\
	(-\varphi_1(x),-\varphi_2(x),\ldots,-\varphi_m(x))^T
	\end{pmatrix}.
$$
For simplicity, we assume that there exist (potentially very large) bound for the optimal Lagrange multiplier $\overrightarrow{\lambda}^*$, which allows us to compactify the feasible set for the pair $(x,\overrightarrow{\lambda})$ to be an Euclidean ball of some radius.

We run Algorithm \ref{Alg:UMPModel} for different values of  $n$ and  $m$ with standard Euclidean prox-structure and the starting point $(x^0, \overrightarrow{\lambda}^0) = \frac{1}{\sqrt{m+n}} \textbf{1} \in \mathbb{R}^{n+m}$, where $\textbf{1} $ is the vector of all ones. Points $A_k$, $k=1,...,N$, in \eqref{prob_1}, \eqref{prob_2} and the point $A$ in \eqref{prob_3} are chosen randomly from standard normal distribution. The set $Q_1$ is chosen as the unit ball in $\mathbb{R}^n$, also $Q_2$ is chosen as the unit ball in $\mathbb{R}^m_+$.

All experiments were implemented in Python 3.4, on a computer fitted with Intel(R) Core(TM) i7-8550U CPU @ 1.80GHz, 1992 Mhz, 4 Core(s), 8 Logical Processor(s). RAM of the computer is 8 GB.

The results of the comparison, for problems \eqref{prob_1}, \eqref{prob_2} and \eqref{prob_3}, are represented in Tables \ref{table_1}, \ref{table_2} and \ref{table_3}, respectively. These results demonstrate the number of the produced iterations to reach the $\varepsilon$-solution of the three considered problems, and the running time of the algorithm in seconds, with different values of $\varepsilon \in \{ 1/2^i, i=1,2,3,4,5,6\}$.

In general, from all conducted experiments, we can see that, in order to the number of iterations and the running time, the proposed procedure of the step-size selection in Algorithm \ref{Alg:UMPModel} is the best, the efficiency of this procedure is represented by its very high execution speed, where by the proposed  Algorithm \ref{Alg:UMPModel} one needs a few seconds and part of seconds to achieve the solution and to reach its stopping criterion.

\begin{table}[htp]
\centering
\caption{The results for the  problem \ref{prob_1}, with  $n =100, m=100$  and $N = 50$.}
\label{table_1}
\begin{tabular}{|c|c|c|c|c|}
\hline
\multirow{2}{*}{$\frac{1}{\varepsilon}$} & \multicolumn{2}{c|}{\begin{tabular}[c]{@{}c@{}}Algorithm \ref{Alg:UMPModel}, with \\ $L_{k+1} = 2^{i_k -1} L_k$.\end{tabular}} & \multicolumn{2}{c|}{\begin{tabular}[c]{@{}c@{}}Algorithm \ref{Alg:UMPModel}, with \\ $L_{k+1} = 2^{i_k} L_k$.\end{tabular}} \\ \cline{2-5}
& Iterations & Time (sec.)& Iterations & Time (sec.)\\ \hline
2	& 6   &0.242 &82    &3.985 \\
4	& 7   &0.272 &163   &9.207  \\
8	& 8   &0.330 &646   &35.780 \\
16	& 10  &0.443 &2587  &142.047 \\
32	& 11  &0.479 &10371 &530.594  \\
64	& 12  &0.513 &41558 &2081.414  \\ \hline
\end{tabular}
\end{table}

\begin{table}[htp]
	\centering
	\caption{The results for the problem \eqref{prob_2}, with $n=100, m=100$  and $N = 50$.}
	\label{table_2}
	\begin{tabular}{|c|c|c|c|c|}
		\hline
		\multirow{2}{*}{$\frac{1}{\varepsilon}$} & \multicolumn{2}{c|}{\begin{tabular}[c]{@{}c@{}}Algorithm \ref{Alg:UMPModel}, with \\ $L_{k+1} = 2^{i_k -1} L_k$.\end{tabular}} & \multicolumn{2}{c|}{\begin{tabular}[c]{@{}c@{}}Algorithm \ref{Alg:UMPModel}, with \\ $L_{k+1} = 2^{i_k} L_k$.\end{tabular}} \\ \cline{2-5}
		& Iterations & Time (sec.)& Iterations & Time (sec.)\\ \hline
		2	&8   &2.997 & 95   &27.527  \\
		4	&9   &3.459 &189   &52.421 \\
		8	&11  &4.346 &378   &105.997 \\
		16	&12  &4.124 &756   &218.519  \\
		32	&13  &4.615 &1511  &473.350 \\
		64	&15  &5.195 &3021  &881.917 \\ \hline
	\end{tabular}
\end{table}

\begin{table}[htp]
	\centering
	\caption{The results for the problem \ref{prob_3}, with $n = 500$ and  $m = 100$.}
	\label{table_3}
	\begin{tabular}{|c|c|c|c|c|}
		\hline
		\multirow{2}{*}{$\frac{1}{\varepsilon}$} & \multicolumn{2}{c|}{\begin{tabular}[c]{@{}c@{}}Algorithm \ref{Alg:UMPModel}, with \\ $L_{k+1} = 2^{i_k -1} L_k$.\end{tabular}} & \multicolumn{2}{c|}{\begin{tabular}[c]{@{}c@{}}Algorithm \ref{Alg:UMPModel}, with \\ $L_{k+1} = 2^{i_k} L_k$.\end{tabular}} \\ \cline{2-5}
		& Iterations & Time (sec.)& Iterations & Time (sec.)\\ \hline
		2	&8   &1.273  &86   &11.293 \\
		4	&9   &1.365  &342  &42.407 \\
		8	&10  &1.491  &684  &100.926 \\
		16	&12  &1.837  &1367 &171.827 \\
		32	&13  &1.982  &2733 &358.685  \\
		64	&15  &2.382  &5465 &732.999 \\ \hline
	\end{tabular}
\end{table}

\section{Numerical results for relatively smooth objective function }
\label{appendix_relative_obj}
Let
$$
    f(x):=\frac{1}{4}\|A x-b\|_{4}^{4}+\frac{1}{2}\|C x-d\|_{2}^{2}.
$$
where $A, C \in \mathbb{R}^{n\times n}$ and $b, d \in \mathbb{R}^n$ \cite{lu2018relatively}. We consider the following constrained optimization problem
\begin{equation}\label{prob_relative}
	\min_{x \in Q} \left\{f(x)  \left| \; \varphi_p(x):= \sum_{i=1}^n \alpha_{pi}x_i -  1 \leq 0 \right., \; p=1,...,m \right\},
\end{equation}
where $Q$ is a convex compact, $\alpha_{pi}$ are drawn from the standard normal distribution and then truncated to be positive. The corresponding Lagrange saddle-point problem and the monotone non-smooth operator associated with the variational inequality problem, defined similarly as in the Appendix \ref{app_numerical_1}.

We run Algorithm \ref{Alg:UMPModel}, for different values of  $n$ and  $m$ with standard Euclidean prox-structure and the same inputs as in Appendix \ref{app_numerical_1}. The results are represented in Table \ref{table_UMP}. These results demonstrate the number of the produced iterations, by Algorithm \ref{Alg:UMPModel}, to reach the $\varepsilon$-solution of \eqref{prob_relative}, and the running time of the algorithm in seconds, with different values of  $\varepsilon \in \{ 1/2^i, i=1,2,3,4,5,6\}$.

\begin{table}[h]
\centering
\caption{The results of Algorithm \ref{Alg:UMPModel}, for problem \eqref{prob_relative}, with  different values of $m$ and $n$.}
\label{table_UMP}
\begin{tabular}{|c|c|c|c|c|}
\hline
\multirow{2}{*}{$\frac{1}{\varepsilon}$} & \multicolumn{2}{c|}{$n = 100, m = 10$} & \multicolumn{2}{c|}{$n =200, m = 10$} \\ \cline{2-5}
	& Iterations & Time (sec.)& Iterations & Time (sec.)\\ \hline
2	&2486  &12.603  &5271   &38.381  \\
4	&4792  &23.380  &10084  &73.649  \\
8	&9380  &43.940  &19589  &140.668  \\
16	&18503 &86.753  &38406  &269.261 \\
32	&36686 &186.048 &75610  &537.010   \\
64	&72990 &369.457 &149761 &1066.677  \\ \hline \hline
\multirow{2}{*}{} & \multicolumn{2}{c|}{$n = 300, m = 10$} & \multicolumn{2}{c|}{$n = 400, m = 10$} \\ \cline{2-5}
2	&7921   &79.872   &11076  & 134.754 \\
4	&15581  &145.664  &22027  & 264.115 \\
8	&31082  &318.656  &43917  & 527.536 \\
16	&62085  &612.985  &87686  & 1041.121\\
32	&124099 &1258.532 &175198 & 2098.559  \\
64	&248135 &2424.470 &350201 & 4189.408 \\ \hline \hline
\multirow{2}{*}{} & \multicolumn{2}{c|}{$n = 100, m = 50$} & \multicolumn{2}{c|}{$n = 200, m = 50$} \\ \cline{2-5}
2	&2491   &38.386   &5472  &139.710  \\
4	&4813   &72.719   &10767  &273.139 \\
8	&9456   &142.670  &21341  &551.234  \\
16	&18803  &282.562  &42465  &1077.094 \\
32	&37630  &569.258  &84703  &2256.059   \\
64	&75438  &1140.276 &169159  &4388.328  \\ \hline
\end{tabular}
\end{table}

Now let us consider the unconstrained case of this relatively smooth optimization problem:
\begin{equation}
    \label{poly_problem}
    \min_x f(x):=\frac{1}{4}\|A x-b\|_{4}^{4}+\frac{1}{2}\|C x-d\|_{2}^{2},
\end{equation}
where $A, C \in \mathbb{R}^{n\times n}$ and $b, d \in \mathbb{R}^n$. $f$ is $L$-smooth relative to $h(x) := \frac{1}{4} \|x\|^4_2 + \frac{1}{2} \|x\|^2_2$
with constant
$$L = 3 \|A\|^4 + 6 \|A\|^3 \|b\|_2 + 3 \|A\|^2 \|b\|_2^2 + \|C\|^2.$$

In Figure \ref{fig_poly} we compare the performance of Gradient Method for relatively smooth problems \cite{lu2018relatively} and Adaptive Gradient Method proposed in this paper. Components of $A$, $C$, $b$, and $d$ are drawn from the uniform distribution $\mathcal{U}(0.95, 1.05)$. The horizontal axis is the number of solved auxiliary problems (\ref{equmir2DL_G_S}), and therefore the plot is not strictly decreasing. The graph shows that adaptive choice of local Lipschitz constant improves the convergence of the gradient method.

\begin{figure}[htp]
\begin{center}
	\includegraphics[width=0.47\linewidth]{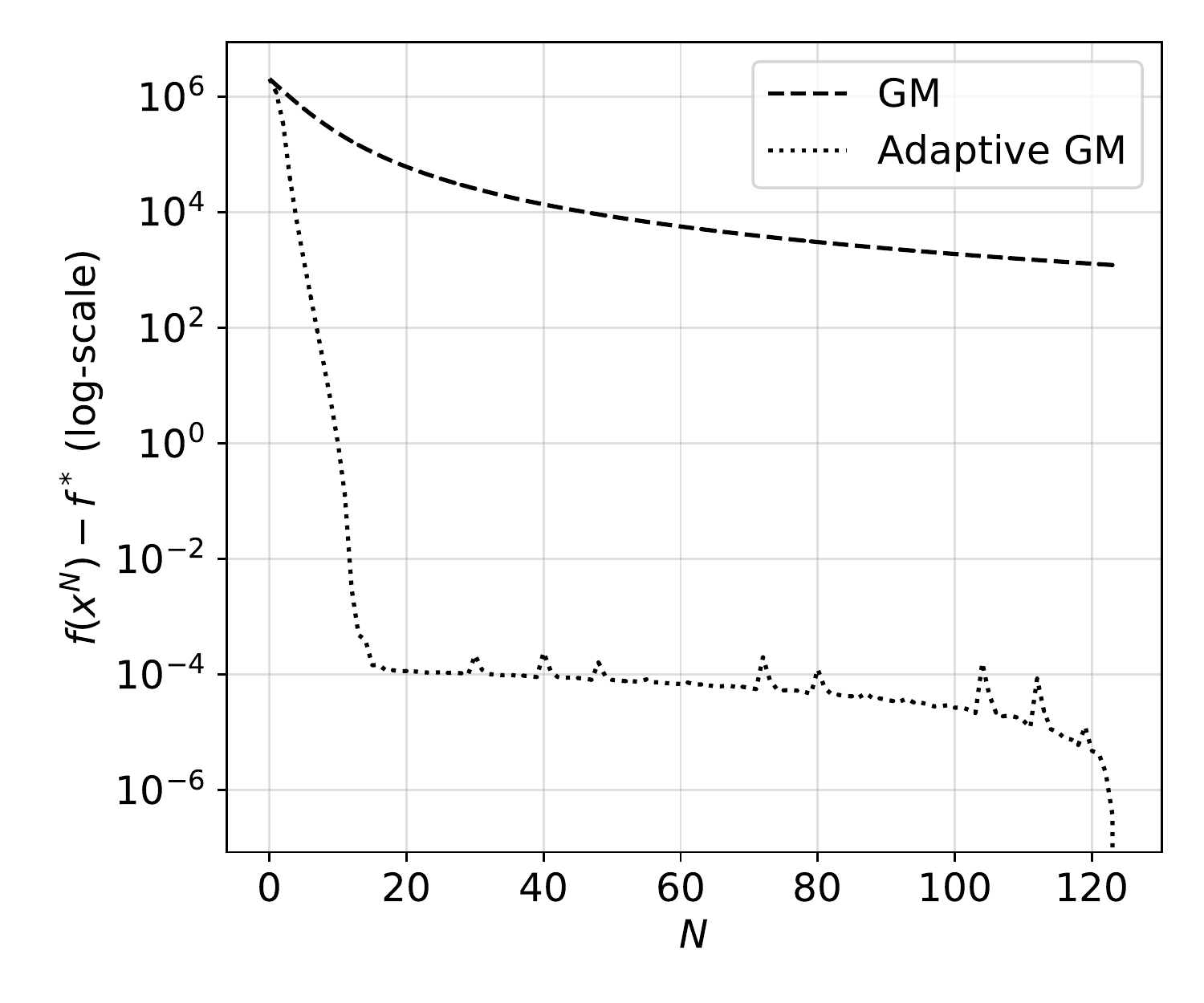}
\end{center}
\caption{The results of Adaptive GM (Algorithm \ref{Alg2}) for problem (\ref{poly_problem}) with $n=100$.}
\label{fig_poly}
\end{figure}

\section{Numerical results for D-optimal design problem }\label{D_optimal_problem}
Let us consider the D-optimal design problem:
\begin{equation}
    \label{log_problem}
    \min_x f(x):= -\ln \det (H X H^\top)
    \quad \text{s.t.\;} \langle e, x\rangle = 1, \quad x \geq 0,
\end{equation}
where $x \in \mathbb{R}^n$, $X = \text{Diag}(x)$, $H \in \mathbb{R}^{m \times n}$, $e = (1, 1, ..., 1)^\top$. $f$ is $1$-smooth relative to $h(x) := - \sum_{j=1}^n \ln x_j$.

Note that simplicial constraints of the problem are satisfied for the solutions supplied by the considered methods  (Algorithm \ref{Alg2}) due to the use of a special procedure described in \cite{lu2018relatively} for solving the auxiliary problem (\ref{equmir2DL_G_S}), which is based on the form of the function $h$ and the structure of constraints of this particular problem.

In this experiment, we consider several cases for different $L$ step sizes in the adaptive choice procedure in Algorithm \ref{Alg2}. Instead of iteration over values of local Lipschitz constant with fixed step $\zeta = 2$:
\[
L_{k+1} = 2^{i_k}L_k
\]
we introduce $\zeta > 1$ as a hyperparameter of the algorithm:
\[
L_{k+1} = \zeta^{i_k}L_k,
\]

In Figure \ref{fig_log} we compare the performance of the Gradient Method for relatively smooth problems \cite{lu2018relatively} and Adaptive Gradient Method proposed in this paper with different $L$ stepsizes $\zeta$. Part of diagonal components of $H$ are drawn from the uniform distribution $\mathcal{U}(0, 1)$ and the other part from the $\mathcal{U}(0, 200)$, off-diagonal elements are sparsely drawn from $\mathcal{U}(0, 1)$ (proportion of nonzero elements is $0.05$). The graph shows that the adaptive choice of the local Lipschitz constant improves the convergence of the gradient method, and the convergence rate can be controlled by the $\zeta$ hyperparameter tuning. The vertical axis of the graph measures $f(x^N) - f^*$ in logarithmic scale with $f^*$ estimated as $f^* \geq f(x^N) - R$ where $R$ is defined by \eqref{analysis_algorithm_grad:ineq_2}.

\begin{figure}[htp]
\begin{center}
	\minipage{0.47\textwidth}
	\includegraphics[width=\linewidth]{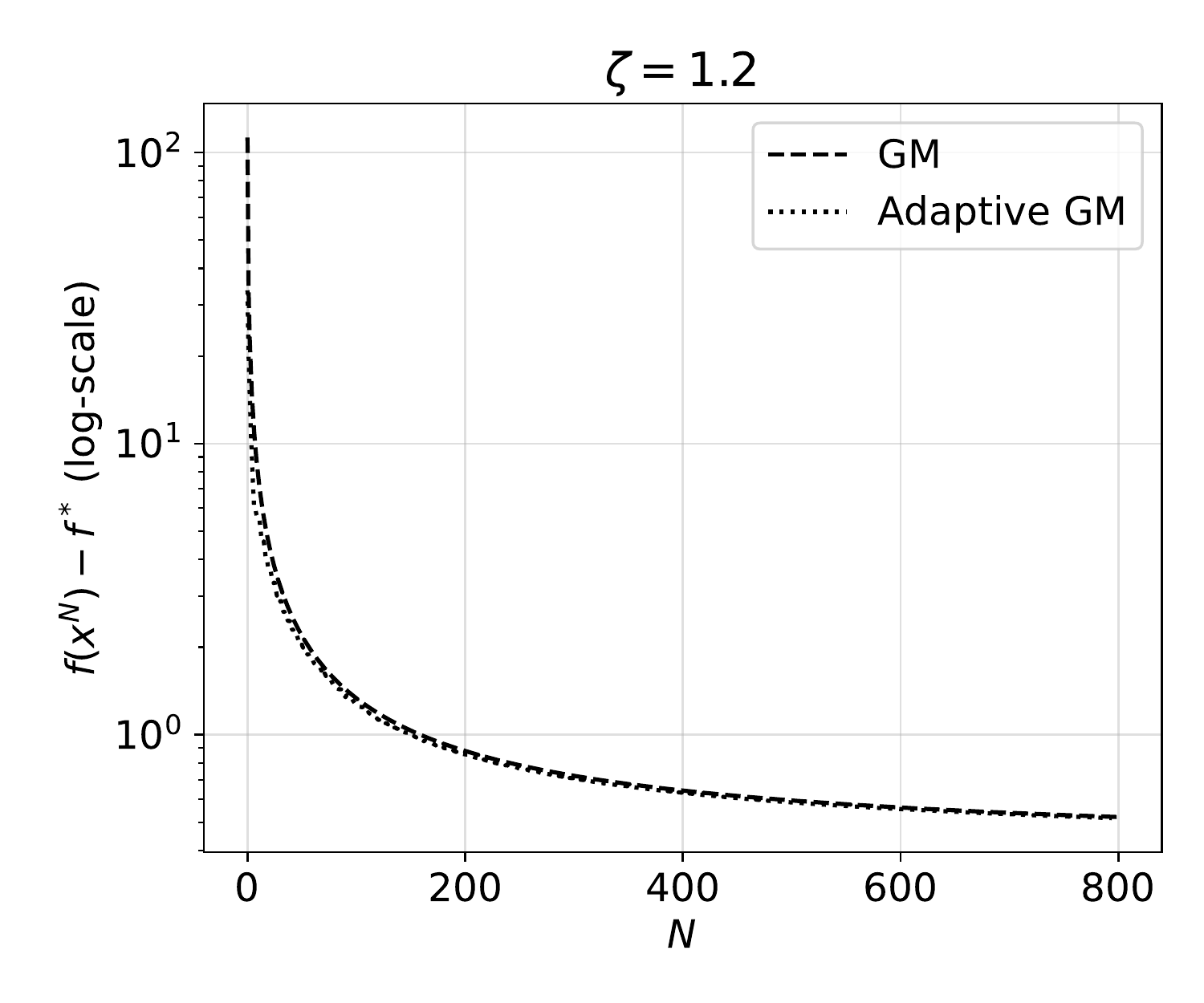}
	\endminipage
	\minipage{0.47\textwidth}
	\includegraphics[width=\linewidth]{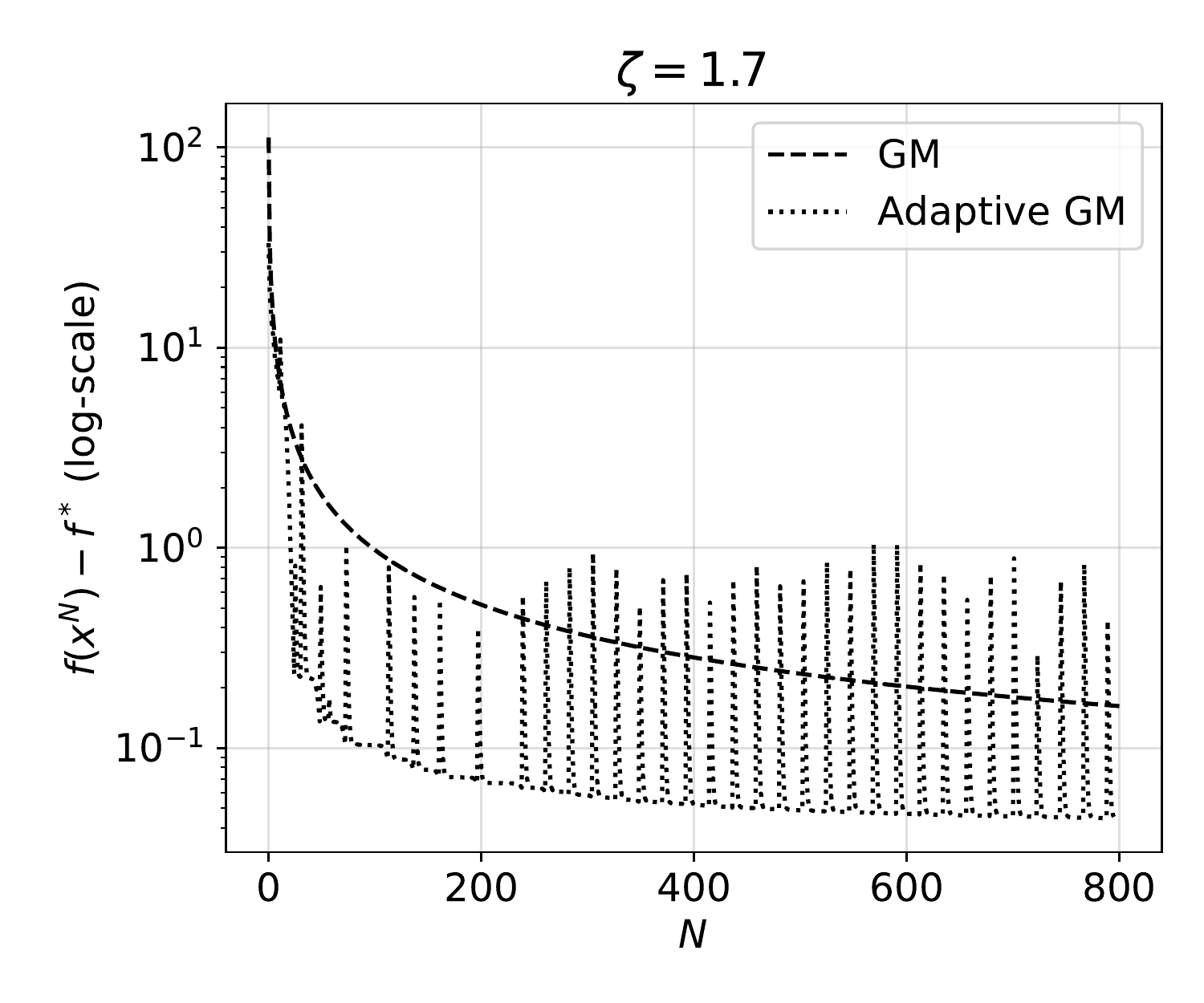}
	\endminipage
\end{center}
\caption{The results of Adaptive GM (Algorithm \ref{Alg2}) for problem (\ref{log_problem}) with \dip{$m=100$, $n=200$}.}
\label{fig_log}
\end{figure}

\section{Numerical results for resource sharing problem}\label{app_RSP}
 
In this section, we report the results of a series of numerical experiments for the \textit{resource sharing problem} (RS-Problem) \cite{antonakopoulos2019adaptive}. This problem can be briefly described as follows.

Consider a set of $n$ resources $i \in I = \{1,2,\ldots, n\}$ serving a stream of demands that arrive at a rate of $r$ per unit of time. If the load on the $i$-th resource is $x_i$, the expected service time in the standard Kleinrock model\footnote{see: L. Kleinrock: Queueing Systems, volume 1: Theory. John Wiley \& Sons, New York, NY, 1975.} is given by the following  loss function
$$
    l_i(x) = \frac{1}{\alpha_i -  x_i},
$$
where $\alpha_i$ denotes the capacity of the resource. The set of feasible resource allocations is $Q = \{(x_1, \ldots, x_n): 0 \leq x_i < \alpha_i, \; x_1 + \ldots +x_n = R\}$. If we set the operator $g: Q \to \mathbb{R}^n$, such that $g(x) = \left(l_1(x_1), \ldots, l_n(x_n)\right)$. The point $x^*$ is an equilibrium allocation if and only if it is a solution of the corresponding variational inequality problem for the operator $g$.

Let $d(x) = \sum_{i =1}^{n}\frac{1}{1-x_i}$, the prox set-up given on the feasible set $Q$, then the associated Bregman divergence is
$$
V{[y]}(x) = \sum_{i= 1}^{n} \frac{(x_i - y_i)^2}{(1-x_i)(1-y_i)^2}
$$
By this setting, the operator $g$ is $\frac{1}{\sqrt{2}}$-Bregman continuous \cite{antonakopoulos2019adaptive} and relatively $\frac{1}{2}$-smooth relative to $d$ (see Examples \ref{Examp_Rel_Smooth_VI} and \ref{L_condition_Example}).

For the experiments, in the operator $g$, we take $\alpha_i = \frac{\sqrt{3}}{2}$, for all $i= {1,...,n}$ and $R = \frac{n}{2}$. We run Algorithm \ref{Alg:UMPModel}, for different values of $n$ and $\varepsilon$. 
The starting point $\left(\frac{1}{2}, \ldots, \frac{1}{2}\right) \in \mathbb{R}^n$, for this, in order to the stopping criterion \eqref{eq_Alg3}, we have the following estimate $\max\limits_{x \in Q}V[x^0](x) \leq 4n$.

We compare the performance of Algorithm \ref{Alg:UMPModel}, for $n = 100, 250, 500$ and $\varepsilon\in \{\frac{1}{2}, \frac{1}{4}, \frac{1}{8}, \frac{1}{16}, \frac{1}{32},\frac{1}{64}\}$, with three different approaches for step-size selection, the first one is a non-adaptive variant with fixed step-size $L$= 0.5, the second is the proposed procedure in this paper, i.e. with $L_{k+1}= 2^{i_k-1}L_k$, the third one is the procedure with $L_{k+1}= 2^{i_k}L_k$, i.e. non-increasing stepsize. The results of the comparison are presented in Tables \ref{tab_2}. 

As shown by the results of experiments, in Table \ref{tab_2}, the required number of steps and required time to achieve  the  necessary  stopping  criterion increases if the stepsize is not allowed to increase as in Algorithm 1 in \cite{antonakopoulos2019adaptive}. Moreover, from the results listed in Table \ref{tab_2}, we can see that the performance of the algorithm, in this case, is better than the non-adaptive version with the fixed stepsize. 

\begin{table}[htp]
	\centering
	\caption{The results for RS-Problem, with  different values of $n$ and $\varepsilon$.}
	\label{tab_2}
	\begin{tabular}{|c|c|c|c|c|c|c|}
		\hline
		\multirow{2}{*}{$\frac{1}{\varepsilon}$} & 
		\multicolumn{2}{c|}{\begin{tabular}[c]{@{}c@{}}Non-Adaptive Algor-\\ithm \ref{Alg:UMPModel}, with $L = 0.5$.\end{tabular}} & \multicolumn{2}{c|}{\begin{tabular}[c]{@{}c@{}}Algorithm \ref{Alg:UMPModel}, with \\ $L_{k+1} = 2^{i_k -1} L_k$.\end{tabular}} &
		\multicolumn{2}{c|}{\begin{tabular}[c]{@{}c@{}}Algorithm \ref{Alg:UMPModel}, with \\ $L_{k+1} = 2^{i_k} L_k$.\end{tabular}} \\ \cline{2-7}
		& Iterations & Time (sec.) & Iterations & Time (sec.)& Iterations & Time (sec.)\\ \hline
		\multicolumn{7}{|c|}{$n = 100$}     \\ \hline
		2	&400&11.67684&5 &0.14099 &40 &1.13174  \\
		4	&800&20.13218&6 &0.15329 &80 &2.31981  \\
		8	&1600&40.18136&7 &0.18168 &160 &4.55859 \\
		16	&3200&89.0930&8 &0.20135 &320 &11.27592  \\
		32	&$>4000$&$>100$&9 &0.23584 &640 &16.20664  \\
		64	&$>4000$&$>100$&10&0.25457 &1280 &32.53708  \\ \hline \hline
		\multicolumn{7}{|c|}{$n = 250$}     \\ \hline
		2   &1000&164.46193&6&1.0216 & 100 &16.49735  \\
		4	&2000&326.25757&7&1.3804& 200& 31.96703  \\
		8	&$>3400$&$>550$&8&1.32271&400 &63.91887   \\
		16	&$>3400$&$>550$&9&1.52729&800 & 135.69973   \\
		32	&$>3400$&$>550$&10&1.686& 1600& 264.14918   \\
		64	&$>3400$&$>550$&11&1.82691&3200 & 525.27188   \\ \hline \hline
		\multicolumn{7}{|c|}{$n = 500$}     \\ \hline
		2   &$>600$&$>500$&7 &5.78455 & 200&167.49536  \\
		4	&$>600$&$>500$&8 &6.60655 &400 & 331.99859 \\
		8	&$>600$&$>500$&9 &7.34555 &$>554$ &$>500$ \\
		16	&$>600$&$>500$&10 &8.38423 & $>554$ &$>500$  \\
		32	&$>600$&$>500$&11 &9.3096 & $>554$ &$>500$ \\
		64	&$>600$&$>500$&12 &10.15076 & $>554$ &$>500$  \\ \hline
	\end{tabular}
\end{table}

Since the running time of the non-adaptive variant with fixed step-size of Algorithm 3 and the variant with non-increasing stepsize is big (as we saw from the listed results in Table \ref{tab_2}) and for more illustration of the results of the proposed Algorithm \ref{Alg:UMPModel} (with $L_{k+1} = 2^{i_k - 1} L_k$)  and to show the slope of the lines, that illustrate the running time and the number of iterations as a function of $\varepsilon$, we run only proposed Algorithm \ref{Alg:UMPModel} for $n = 100, 250, 500, 1000, 1500, 2000$ with $\varepsilon\in \left\{\frac{1}{2}, \frac{1}{4}, \frac{1}{8}, \frac{1}{16}, \frac{1}{32},\frac{1}{64}, \text{and} \;\; 5\times 10^{-i}, 10^{-i}: i=3,4,5,6\right\}$.
 (see Fig. \ref{fig1} and Fig. \ref{fig2}).

 As we see, the slope of the lines is much smaller than what is predicted by the theory, and the adaptivity of the method allows to accelerate the convergence in practice.

\begin{figure}[htp]
\begin{center}
\minipage{0.32\textwidth}
\includegraphics[width=\linewidth]{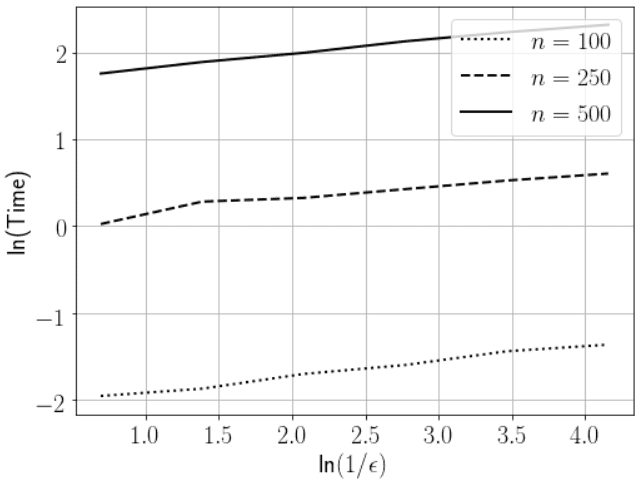}
\endminipage
\minipage{0.32\textwidth}
\includegraphics[width=\linewidth]{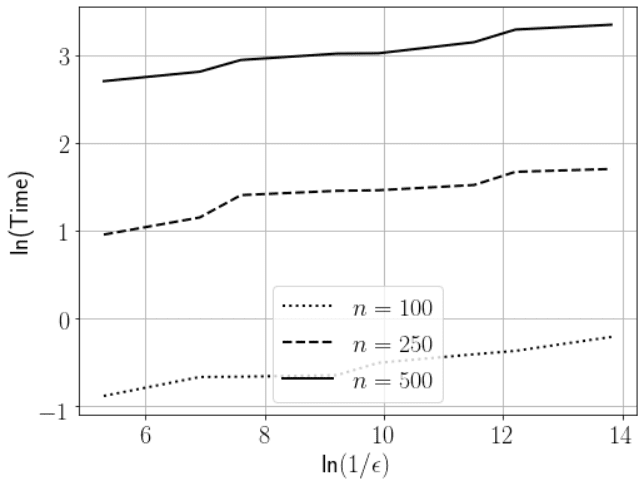}
\endminipage
\minipage{0.32\textwidth}
\includegraphics[width=\linewidth]{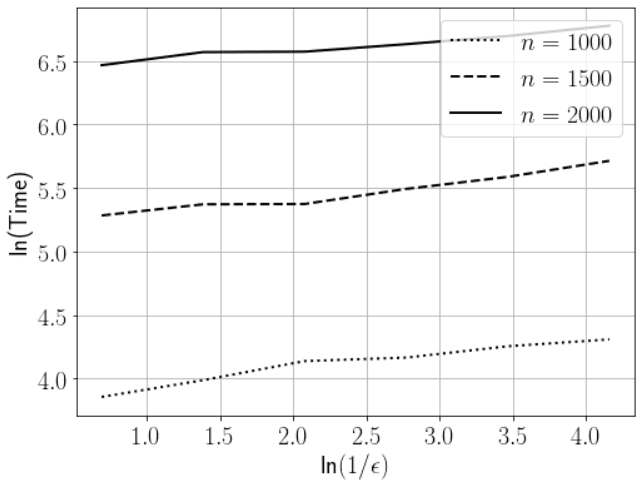}
\endminipage
\end{center}
\caption{The results of Algorithm \ref{Alg:UMPModel}. The running time as a function of $\varepsilon$, for RS-Problem, with  different values of $n$ and $\varepsilon$. The pictures on the left and the right show the results with $\varepsilon\in \{\frac{1}{2}, \frac{1}{4}, \frac{1}{8}, \frac{1}{16}, \frac{1}{32},\frac{1}{64}\}$. The picture on the middle shows the results with $\varepsilon\in \left\{5\times 10^{-i}, 10^{-i}: i=3,4,5,6\right\}$.}
\label{fig1}
\end{figure}

\begin{figure}[htp]
	\begin{center}
		\minipage{0.32\textwidth}
		\includegraphics[width=\linewidth]{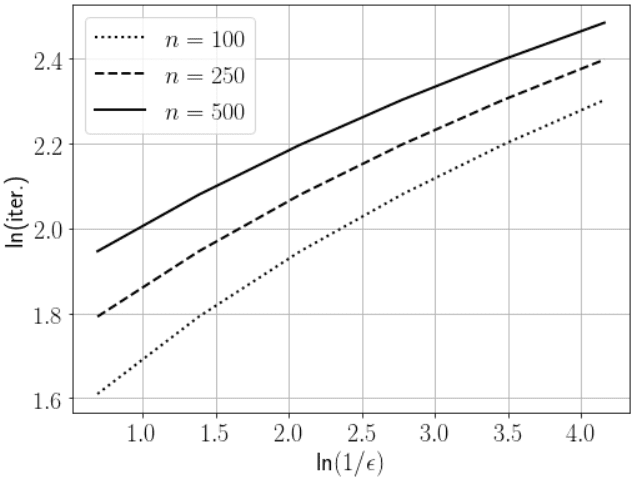}
		\endminipage
		\minipage{0.32\textwidth}
		\includegraphics[width=\linewidth]{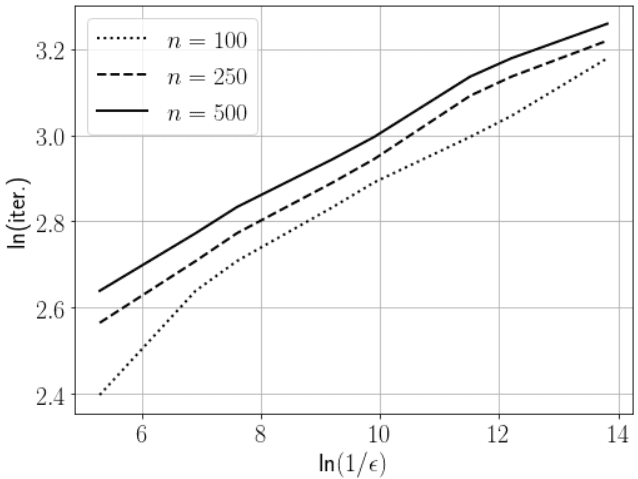}
		\endminipage
		\minipage{0.32\textwidth}
		\includegraphics[width=\linewidth]{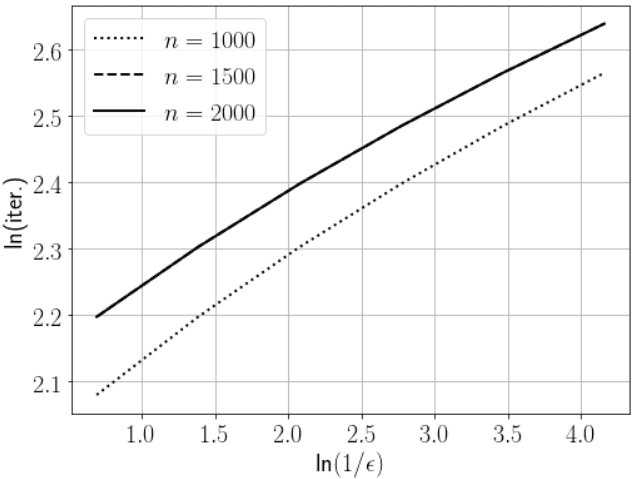}
		\endminipage
	\end{center}
	\caption{The results of Algorithm \ref{Alg:UMPModel}. The number of iterations as a function of $\varepsilon$, for RS-Problem, with  different values of $n$ and $\varepsilon$. The pictures on the left and the right show the results with $\varepsilon\in \{\frac{1}{2}, \frac{1}{4}, \frac{1}{8}, \frac{1}{16}, \frac{1}{32},\frac{1}{64}\}$. The picture on the middle shows the results with $\varepsilon\in \left\{5\times 10^{-i}, 10^{-i}: i=3,4,5,6\right\}$.}
	\label{fig2}
\end{figure}

\end{document}